\documentclass[11pt]{amsart}

\usepackage[T1]{fontenc}
\usepackage{lmodern}
\usepackage{amsmath,amssymb,amsthm,mathtools}
\usepackage{enumitem}
\usepackage{aliascnt}
\usepackage{graphicx}
\usepackage{tikz}
\usetikzlibrary{arrows.meta,positioning,calc,fit,shapes.geometric,decorations.pathmorphing}
\usepackage{hyperref}
\usepackage[nameinlink,capitalize]{cleveref}
\usepackage{microtype}
\hypersetup{
  colorlinks=true,
  linkcolor=blue,
  citecolor=blue,
  urlcolor=blue
}

\theoremstyle{plain}
\newtheorem{theorem}{Theorem}[section]

\newaliascnt{conditionaltheorem}{theorem}

\aliascntresetthe{conditionaltheorem}

\newaliascnt{conjecture}{theorem}
\newtheorem{conjecture}[conjecture]{Conjecture}
\aliascntresetthe{conjecture}

\newaliascnt{proposition}{theorem}
\newtheorem{proposition}[proposition]{Proposition}
\aliascntresetthe{proposition}

\newaliascnt{conditionalproposition}{theorem}

\aliascntresetthe{conditionalproposition}

\newaliascnt{lemma}{theorem}
\newtheorem{lemma}[lemma]{Lemma}
\aliascntresetthe{lemma}

\newaliascnt{corollary}{theorem}
\newtheorem{corollary}[corollary]{Corollary}
\aliascntresetthe{corollary}

\newaliascnt{conditionalcorollary}{theorem}

\aliascntresetthe{conditionalcorollary}

\theoremstyle{definition}
\newaliascnt{definition}{theorem}
\newtheorem{definition}[definition]{Definition}
\aliascntresetthe{definition}

\newaliascnt{assumption}{theorem}

\aliascntresetthe{assumption}

\newaliascnt{convention}{theorem}
\newtheorem{convention}[convention]{Convention}
\aliascntresetthe{convention}

\newaliascnt{example}{theorem}

\aliascntresetthe{example}

\newaliascnt{construction}{theorem}

\aliascntresetthe{construction}

\newaliascnt{problem}{theorem}
\newtheorem{problem}[problem]{Problem}
\aliascntresetthe{problem}

\theoremstyle{remark}
\newaliascnt{remark}{theorem}
\newtheorem{remark}[remark]{Remark}
\aliascntresetthe{remark}

\newaliascnt{observation}{theorem}
\newtheorem{observation}[observation]{Observation}
\aliascntresetthe{observation}

\crefname{theorem}{Theorem}{Theorems}
\Crefname{theorem}{Theorem}{Theorems}
\crefname{conditionaltheorem}{Conditional Theorem}{Conditional Theorems}
\Crefname{conditionaltheorem}{Conditional Theorem}{Conditional Theorems}
\crefname{conjecture}{Conjecture}{Conjectures}
\Crefname{conjecture}{Conjecture}{Conjectures}
\crefname{proposition}{Proposition}{Propositions}
\Crefname{proposition}{Proposition}{Propositions}
\crefname{conditionalproposition}{Conditional Proposition}{Conditional Propositions}
\Crefname{conditionalproposition}{Conditional Proposition}{Conditional Propositions}
\crefname{lemma}{Lemma}{Lemmas}
\Crefname{lemma}{Lemma}{Lemmas}
\crefname{corollary}{Corollary}{Corollaries}
\Crefname{corollary}{Corollary}{Corollaries}
\crefname{conditionalcorollary}{Conditional Corollary}{Conditional Corollaries}
\Crefname{conditionalcorollary}{Conditional Corollary}{Conditional Corollaries}
\crefname{assumption}{Assumption}{Assumptions}
\Crefname{assumption}{Assumption}{Assumptions}
\crefname{definition}{Definition}{Definitions}
\Crefname{definition}{Definition}{Definitions}
\crefname{convention}{Convention}{Conventions}
\Crefname{convention}{Convention}{Conventions}
\crefname{example}{Example}{Examples}
\Crefname{example}{Example}{Examples}
\crefname{construction}{Construction}{Constructions}
\Crefname{construction}{Construction}{Constructions}
\crefname{remark}{Remark}{Remarks}
\Crefname{remark}{Remark}{Remarks}
\crefname{observation}{Observation}{Observations}
\Crefname{observation}{Observation}{Observations}
\crefname{problem}{Problem}{Problems}
\Crefname{problem}{Problem}{Problems}

\newcommand{\R}{\mathbb{R}}
\newcommand{\Sph}{\mathbb{S}}

\newcommand{\Len}{\operatorname{Len}}
\newcommand{\Thi}{\operatorname{Thi}}
\newcommand{\Rop}{\operatorname{Rop}}

\newcommand{\id}{\operatorname{id}}

\newcommand{\Vtx}{\operatorname{Vert}}
\newcommand{\Edge}{\operatorname{Edge}}
\newcommand{\rad}{\operatorname{rad}}
\newcommand{\diam}{\operatorname{diam}}
\newcommand{\MT}{\operatorname{MT}}

\newcommand{\crn}{\operatorname{cr}}

\newcommand{\calX}{\mathcal{X}}
\newcommand{\calY}{\mathcal{Y}}
\newcommand{\calD}{\mathcal{D}}
\newcommand{\calG}{\mathcal{G}}
\newcommand{\calH}{\mathcal{H}}

\newcommand{\calA}{\mathcal{A}}
\newcommand{\calR}{\mathcal{R}}
\newcommand{\eps}{\varepsilon}
\newcommand{\mirror}[1]{\overline{#1}}

\newcommand{\Sim}{\operatorname{Sim}}
\newcommand{\SO}{\operatorname{SO}}
\newcommand{\Ogp}{\operatorname{O}}
\newcommand{\calP}{\mathcal{P}}

\newcommand{\val}{\operatorname{val}}
\newcommand{\Kmax}{\kappa_\infty}
\newcommand{\dOne}{d_{C^1}}
\newcommand{\dStar}{d_\star}
\newcommand{\Roploc}{\operatorname{Rop}_{\mathrm{loc}}}
\newcommand{\Ropglob}{\operatorname{Rop}_{\mathrm{glob}}}
\newcommand{\Diff}{\operatorname{Diff}}

\usetikzlibrary{decorations.markings,backgrounds,shapes.misc,intersections,fit}
\graphicspath{{figures/}}
\definecolor{wallRone}{HTML}{D55E00}
\definecolor{wallRtwo}{HTML}{0072B2}
\definecolor{wallRthree}{HTML}{009E73}
\definecolor{figink}{HTML}{1A1A1A}
\definecolor{figmid}{HTML}{707070}
\definecolor{figsoft}{HTML}{B4B4B4}
\definecolor{figlight}{HTML}{E4E4E4}
\definecolor{figpale}{HTML}{F3F3F3}
\tikzset{
  Rone/.style={wallRone, line width=0.9pt},
  Rtwo/.style={wallRtwo, line width=0.9pt, dash pattern=on 3pt off 1.6pt},
  Rthree/.style={wallRthree, line width=1.15pt, dash pattern=on 0pt off 1.9pt, line cap=round},
  knot/.style={figink, line width=0.9pt, line cap=round, line join=round},
  gvert/.style={circle, fill=figink, inner sep=0pt, minimum size=3.4pt},
  gvertopen/.style={circle, draw=figink, fill=white, inner sep=0pt, minimum size=3.4pt, line width=0.5pt},
  panellabel/.style={font=\small, anchor=north west, inner sep=0pt},
}

\title{Ropelength-Filtered Reidemeister Graphs}

\author{Makoto Ozawa}
\address{Department of Natural Sciences, Faculty of Arts and Sciences, Komazawa University, Tokyo, Japan}
\email{w3c@komazawa-u.ac.jp}

\date{November 2026}

\begin{document}

\begin{abstract}
We prove relative smoothing with uniform reach control for one-parameter
families of positive-reach $C^{1,1}$ knots and combine it with relative
multijet transversality to obtain projection-generic regularization inside
strict ropelength sublevels.  For families over a compact polyhedron we
prove relative smoothing of the slices after an arbitrarily small budget
increase.  Passing to a lifted Reidemeister multigraph reconstructs the
admissible $\pi_0$ of every strict projection-framed ropelength sublevel.
We also compare the strong, quotient, and uniform-slack path relations and
show that the corresponding path and admissible merge thresholds agree.

The normalised projection-framed sublevels are direction-marked spaces over
the normalised ropelength sublevels of \cite{OzawaRFKS}; in particular they
are compact Hausdorff, so connected merge levels are attained.  At closed
critical levels we use only arbitrarily small right relaxation and do not
assert exact genericity within the critical level itself.

For finite recognition we distinguish three notions that play different
roles: full rooted-ball occurrence in the completed $S^2$-Reidemeister
multigraph, visibility in the monotone diagram-image filtration, and
saturation of a visible occurrence.  Every finite Reidemeister submultigraph
becomes visible at a finite ropelength level, and a visible saturated
BC-characteristic certificate determines the knot type up to mirroring.
A universal computable crossing-complexity bound controls visibility.
Finally, vertex-coherent lifting of a finite pattern is characterized by a
face-consistent planar labeling, in particular for tree-shaped patterns and
separated cube systems.

The degree-zero results are organized by a filtered relational cellular
cosheaf of sets on the planar Reidemeister multigraph, whose stalks are the
strict fibre components and whose costalks are the relations given by
Reidemeister movies.  The lifted multigraph is its Grothendieck construction,
the reconstruction theorem is a descent statement for it, and the diagram image
is its support.  The diagram image computes merges exactly when no stalk is
persistently separated, and a canonical flat coarsening is the coarsest model
indexed by diagrams that does.  For amphichiral knots the spherical image
cannot see mirror barriers, whereas the planar image can; we show that no
planar type of the standard diagrams of $4_1$ and $6_3$ is fixed by crossing
switch, and exhibit seven-crossing diagrams of $6_3$ whose crossing-number
barrier is invisible to ropelength.
\end{abstract}



\maketitle

\section{Introduction}

Ropelength filters the space of realisations of a knot type: for a budget
$\Lambda$ one keeps only those thick representatives that fit inside it.  The
resulting sublevel spaces and their component structure are the subject of
\cite{OzawaRFKS}; a finite-state discrete counterpart is developed in
\cite{OzawaLattice}, and the geometric cost of a deformation, once one is
available, is measured in \cite{OzawaSweptArea}.  The present paper supplies
the analytic and diagrammatic comparison that this programme needs in degree
zero.

Two analytic and topological questions guide the construction. Relative
smoothing replaces strong parameterized positive-reach families by smooth
ones after an arbitrarily small budget increase. Projection-generic
regularization and fibre-aware transitions then reconstruct admissible
components. We prove the required one-parameter comparison directly, extend
relative smoothing to compact polyhedra, and state separately the additional
relative-genericity hypothesis needed for higher reconstruction.

The mechanism in both cases is slack.  On a \emph{strict} sublevel
$\Rop<\Lambda$, a compact family in the strong quotient or $\star$
topology has a uniform positive gap to the budget.  For a merely
$C^1$-continuous family this uniform gap is required separately.
Normalised arclength convolution smooths the family while controlling reach
through the thickness formula, so the gap is preserved; relative multijet
transversality then makes the projection generic.  A $C^1$ path in the
closed level $\Rop\le\Lambda$ can be replaced, with the same endpoints,
by a strictly admissible path below $\Lambda+\varepsilon$,
so the closed-level statements follow by right relaxation.  We do not assert
that a path in a closed critical level can be made generic without leaving
that level, and nothing below requires it.

\subsection*{The two spaces}

Fix $u\in\Sph^2$.  Quotienting a space of curves by all rotations and then
fixing a projection direction does not produce a well-defined projection, so
the frame has to be carried along.  We therefore work with projection-framed
spaces throughout, in two forms.

The scale-invariant form $\calX_{\Lambda,u}(K)$ consists of the
positive-thickness $C^{1,1}$ representatives of $K$ with $\Rop\le\Lambda$,
taken modulo orientation-preserving similarities whose rotational part fixes
$u$.  Scale invariance is what the analytic arguments need: the
time-dependent homothety in the proof of
\cref{prop:parametric-regularization} is a change of representative only
because homotheties are in the group.

The normalised form $\calY_{\Lambda,u}(K)$ imposes $\Thi\ge1$ and
$\Len\le\Lambda$ and divides only by isometries fixing $u$, which is the
model of \cite[Definition~6.1]{OzawaRFKS} with a marked direction.
\Cref{prop:direction-marked} makes the comparison precise: $\calY_{\Lambda,u}(K)$
is the direction-marked space over $Y_\Lambda(K)$, with fibre $S^2/\Gamma_\gamma$, where $\Gamma_\gamma$ is the rotational stabilizer --- the same device, one level down, as the pole-marked
bundle of \cite[Theorem~7.3]{OzawaRFKS}.  By \cref{thm:sublevel-compact} it is
compact Hausdorff, so connected merge levels are attained
(\cref{cor:attained-merge}).  Forgetting the scale gives a canonical
surjection to the scale-invariant classes.  It is continuous into their
$C^1$ model, not asserted continuous into their strong quotient or
$\star$ topology; see \cref{rem:sigma-caveat}.

The diagrammatic object is a pair of multigraphs.  The lifted multigraph
$\calG_{\Lambda,u}^{\operatorname{lift}}(K)$ records projection-fiber
components and bounded-ropelength Reidemeister movies.  Its diagram image
$\calH_{\Lambda,u}(K)$ is an increasing submultigraph of the classical
$S^2$-Reidemeister multigraph $G_S(K)$.  The separation between the two is
forced: fiber components merge as $\Lambda$ increases, so the lifted graphs
do not form an increasing sequence of subgraphs, whereas a recognition
statement needs a literal monotone filtration.  Component and merge
information therefore lives in $\calG^{\operatorname{lift}}$, and recognition
is formulated in $\calH$.

Both graphs come from one object (\cref{sec:cosheaf}).  Assigning to each
planar diagram the set of its strict fibre components, and to each
Reidemeister move the relation between components realized by movies, gives a
filtered cosheaf of sets on the planar Reidemeister multigraph.  The lifted
multigraph is its Grothendieck construction and the diagram image is its
support; the maps between Grothendieck constructions at different levels
merge vertices, while supports only grow.  The strict reconstruction theorem
becomes a descent statement: $\pi_0$ of the strict sublevel is the colimit of
the fibre components over the Reidemeister multigraph.  In this language one
can say exactly when a diagram image computes merges, and for which mirror
questions it is blind.

\subsection*{Finite recognition}

Barbensi and Celoria proved that $G_S(K)$ is a complete invariant up to
mirroring and that the information needed for their reconstruction is already
contained in a sufficiently large finite neighborhood \cite{BarbensiCeloria}.
For the filtered application we make the local comparison precise by
separating three notions: full rooted-ball occurrence in the completed
Reidemeister multigraph, visibility at a finite ropelength level, and
saturation of a visible occurrence.  An arbitrary injective subgraph copy is
too weak, and suppressing parallel edges discards information used by the
local reconstruction.  Ambient valences and inducedness provide finite
checks for saturation.  Diagram labels are forgotten when certificates are
compared, since a single labelled diagram would already determine the knot
type.

Every such finite ball has finite geometric visibility.  Each of its finitely
many diagrams has a positive-thickness spatial realization and each of its
finitely many edges has a standard spatial Reidemeister movie, so the maximum
ropelength over these finite choices is a visibility bound.  This gives
$L_{\operatorname{char},u}(K)<\infty$ for every knot type.  The argument can
be made effective through a universal computable function $A(n)$ of crossing
complexity, although the resulting bound is very coarse;
\cref{subsec:move-counts} explains why this quantity is independent of the
move-count bounds of \cite{Lackenby,CowardLackenby,HassLagarias} rather than
a consequence of them.

One combinatorial constraint survives when one asks for the stronger
\emph{coherent} realization, in which all edges of a pattern are realized
between one fixed spatial lift per vertex.  A movie whose projection stays
regular cannot change the planar isotopy class of the projection, so the
outer-face data must be chosen consistently along the pattern.  This
face-consistency is automatic for trees and for separated cube systems, and
it is precisely equivalent to the existence of a vertex-coherent lift.  It
plays no role in the monotone diagram-image filtration and hence none in
recognizability.

\subsection*{Results}

\begin{itemize}[leftmargin=*]
\item[\textbf{A.}] the normalised projection-framed sublevels are compact
Hausdorff, and are canonically direction-marked spaces over the normalised
sublevels of \cite{OzawaRFKS}
(\cref{thm:sublevel-compact,prop:direction-marked});
\item[\textbf{B.}] for strongly continuous parameterized paths, relative
smoothing holds with uniform reach control, giving right-relaxed relative smoothing
(\cref{lem:parametric-smoothing,cor:conj82-dim1});
\item[\textbf{C.}] strict-sublevel admissible components agree with path
components in the strong quotient and $\star$ models
(\cref{cor:strict-path-comparison}) and are computed by the
strict lifted multigraph
(\cref{thm:strict-component-reconstruction,cor:conj81-pi0});
\item[\textbf{D.}] connected merge levels are attained, and the right-relaxed
diagrammatic merge scale equals the geometric admissible merge scale
unconditionally; the three merge functions in play are compared in
\cref{subsec:three-merges}
(\cref{cor:attained-merge,lem:merge-comparison,thm:agreement-merge-scales});
\item[\textbf{E.}] a finite specified Reidemeister pattern admits a
vertex-coherent lift at one finite ropelength level if and only if it admits
a face-consistent planar labeling; in particular every tree-shaped pattern
does (\cref{thm:coherent-finite-lift});
\item[\textbf{F.}] every finite Reidemeister submultigraph is visible in
$\calH_{\Lambda,u}(K)$ at some finite level, and consequently every knot type
has finite recognition length, unconditionally, up to mirroring
(\cref{thm:finite-visibility,thm:finite-recognizability});
\item[\textbf{G.}] a universal computable crossing-complexity bound $A(n)$
controls the visibility of every diagram and move with at most $n$ crossings
(\cref{thm:effective-visibility});
\item[\textbf{H.}] the lifted multigraph is the Grothendieck construction of
the fibre cosheaf and the diagram image its support; the diagram image computes
merges if and only if no stalk is persistently separated, and the flat
coarsening is terminal among models that compute components
(\cref{prop:lifted-groth,thm:descent,thm:image-merges,prop:flat-terminal});
\item[\textbf{I.}] for amphichiral knots, mirror merges are invisible in the
spherical image at self-switch views and are stalk separations at self-switch
planar views; no planar type of the standard diagrams of $4_1$ and $6_3$ is
self-switch, and four seven-crossing diagrams of $6_3$ carry a crossing-number
barrier that rotations remove; symmetric hubs and finite certificates give
upper bounds for mirror merges, and sign separators give lower bounds
(\cref{prop:mirror-trichotomy,prop:no-self-switch,prop:63-hard,prop:symmetric-hub,prop:certificates,prop:sign-separator}).
\end{itemize}

\subsection*{What is not proved}

Relative smoothing extends to compact polyhedral parameter spaces in
\cref{thm:compact-parameter-smoothing}. Higher reconstruction additionally
requires relative genericity and coherent discriminant data; neither follows
from smoothing alone.

Exact point-set equality of components at an
individual closed critical level is neither needed nor asserted.  Finally, the effective bound $A(n)$ separates
logical effectivity from useful geometry only; finding small characterizing
radii or knot-sensitive visibility bounds is a substantially stronger
problem.

\begin{figure}[t]
\centering
\begin{tikzpicture}[
  >={Latex[length=4pt]},
  font=\scriptsize,
  head/.style={font=\tiny\scshape, text=figmid},
  obj/.style={draw=figink, fill=figpale, rounded corners=2.5pt, align=center,
              inner sep=2.5pt, text width=25.5mm, minimum height=12mm},
  res/.style={draw=figsoft, rounded corners=2.5pt, align=flush left, inner xsep=3pt,
              inner ysep=2.5pt, text width=32mm},
  tag/.style={circle, draw=figink, fill=white, inner sep=0.5pt,
              font=\tiny\bfseries, minimum size=8.5pt},
  lab/.style={font=\tiny, align=center, text=figink, inner sep=1pt}
]
\def\xa{0}\def\xb{4.45}\def\xc{8.9}
\node[head] at (\xa,0.98) {thick knots};
\node[head] at (\xb,0.98) {lifted multigraph};
\node[head] at (\xc,0.98) {diagram image};

\node[obj] (X) at (\xa,0)
  {$\calX_{\Lambda,u}(K)$, $\calY_{\Lambda,u}(K)$\\[1pt]
   projection-framed\\ sublevels};
\node[obj] (G) at (\xb,0)
  {$\calG^{\operatorname{lift}}_{\Lambda,u}(K)$, $\calG^{<,\operatorname{lift}}_{\Lambda,u}(K)$\\[1pt]
   vertices $(D,F)$,\\ lifted movies};
\node[obj] (H) at (\xc,0)
  {$\calH_{\Lambda,u}(K)\subset G_S(K)$\\[1pt]
   increasing in $\Lambda$};

\draw[->, line width=0.6pt] (X.east) -- coordinate[pos=0.5] (mXG) (G.west);
\node[lab, above=5pt] at (mXG) {projection--\\Cerf};
\node[tag] at (mXG) {B};
\node[lab, below=5pt] at (mXG) {smoothing,\\ transversality};
\draw[->, line width=0.6pt] ([yshift=7pt]G.east) -- node[lab, above=1pt] {forget fibers\\ $p_{\Lambda,u}$}
  ([yshift=7pt]H.west);
\draw[->, densely dashed, figmid, line width=0.5pt] ([yshift=-7pt]H.west) --
  coordinate[pos=0.5] (mHG) ([yshift=-7pt]G.east);
\node[tag] at (mHG) {E};
\node[lab, below=4pt, text=figmid] at (mHG) {coherent lift};

\node[res, anchor=north] (A) at (\xa+0.18,-1.3)
  {$\calY_{\Lambda,u}(K)$ compact;\\ direction-marked over\\ $Y_\Lambda(K)$, fibres $\Sph^2/\Gamma_\gamma$};
\node[res, anchor=north] (C) at (\xb,-1.3)
  {$\pi_0\calG^{<,\operatorname{lift}}_{\Lambda,u}(K)$\\
   $\quad\cong\pi_0^{\operatorname{adm}}\calX^{<}_{\Lambda,u}(K)$};
\node[res, anchor=north] (D) at ([yshift=-7pt]C.south)
  {$m^{\operatorname{diag},+}_u=m=\mu^{\operatorname{path}}_{K,u}$;\\
   connected merge levels\\ are attained};
\node[res, anchor=north] (F) at (\xc-0.18,-1.3)
  {finite $Q\subset\calH_{\Lambda_Q,u}(K)$;\\ $L_{\operatorname{char},u}(K)<\infty$};
\node[res, anchor=north] (Gb) at ([yshift=-7pt]F.south)
  {$L_u(K;\mathbb B_R(D))$\\ $\quad\le A(\crn(D)+2R)$};
\draw[->, line width=0.5pt] (X.south) -- (X.south |- A.north);
\draw[->, line width=0.5pt] (G.south) -- (C.north);
\draw[->, line width=0.5pt] (H.south) -- (H.south |- F.north);
\foreach \b/\t in {A/A, C/C, D/D, F/F, Gb/G}{
  \node[tag] at ([xshift=-1.5pt,yshift=1.5pt]\b.north west) {\t};
}
\draw[->, densely dashed, figmid, line width=0.5pt]
  (A.south -| X.south) -- ++(0,-0.62) coordinate (aa) -- node[lab, below, text=figmid, pos=0.5]
  {compactness} (aa -| D.west);
\end{tikzpicture}
\caption{Roadmap of the paper.  Projection--Cerf theory, made available by
relative smoothing and transversality, passes from the projection-framed
sublevels to the lifted multigraph; forgetting fibers gives the increasing
diagram image inside the classical $S^2$-Reidemeister multigraph.  Component
and merge information is read in the middle column and recognition in the
right one.  The circled letters refer to the results \textbf{A}--\textbf{G}
of the introduction; the lower dashed arrow records that compactness
of $\calY_{\Lambda,u}(K)$ is what makes connected merge levels attained.  The middle column is the Grothendieck construction of the fibre cosheaf and the right one its support (\cref{sec:cosheaf}).  A finite pattern in the image lifts coherently, with one
fiber component per vertex, exactly when it admits a face-consistent planar
labeling.}
\label{fig:main-pipeline}
\end{figure}

\subsection{Relation to the companion papers}
\label{subsec:companions}

This paper studies projection-framed geometry, regular fibres, and
Reidemeister transitions. The underlying geometric filtration is studied in
\cite{OzawaRFKS}; lattice move systems are treated in \cite{OzawaLattice}.
Definitions and proofs of the framed reconstruction results are given here.
No conjecture or problem from another paper is used as a theorem statement.

\subsection{Organization}

\Cref{sec:ropelength} defines the projection-framed ropelength spaces, proves
compactness, and identifies them as direction-marked spaces over the
normalised sublevels of \cite{OzawaRFKS}.
\Cref{sec:classical} formulates the finite rooted-ball certificates from
Barbensi--Celoria.  \Cref{sec:projection} reviews projection walls.
\Cref{sec:growth-graphs} contains the analytic core: relative parametric
smoothing, parametric regularization, and strict-sublevel component
reconstruction.  \Cref{sec:cosheaf} recasts these results as descent for
the fibre cosheaf and determines what a diagram image sees.
\Cref{sec:growth-invariants} records basic growth
quantities.  \Cref{sec:finite-recognition} proves finite visibility, its
effective crossing-complexity form, unconditional finite recognizability,
and the comparison with move-count bounds.  \Cref{sec:merge-trees} treats
merge persistence and attainment, and \cref{sec:examples} gives examples,
including mirror merges of $4_1$ and $6_3$.
\Cref{sec:compact-parameters} proves relative smoothing over compact
polyhedral parameter spaces, states the additional hypothesis needed for
higher reconstruction, and treats fixed planar fibres.
\Cref{sec:further} records quantitative problems.  The appendices contain
the complete proof of relative parametric smoothing, strong path
approximation, and the comparison of endpoint relations.  The finite
combinatorial checks of \cref{subsec:mirror} are reproduced by Python scripts,
using only the standard library, in the ancillary files
\texttt{verify\_local\_obstructions.py}, \texttt{verify\_move\_logs.py} and
\texttt{verify\_planar\_types.py}, with the data file \texttt{move\_logs.json}.


\section{Projection-framed ropelength sublevel spaces}
\label{sec:ropelength}

We first fix the geometric space on which projection is defined.  This
point is logically prior to the Reidemeister construction: quotienting by all
rotations and then fixing a direction $u$ does not produce a well-defined
projection map.

\begin{definition}[Thickness and ropelength]
Let $\gamma\subset\mathbb R^3$ be a $C^{1,1}$ embedded closed curve.  Its
thickness $\Thi(\gamma)$ is its reach in the sense of Federer
\cite{Federer}.  For positive-thickness curves, define
\[
  \Rop(\gamma)=\frac{\Len(\gamma)}{\Thi(\gamma)}.
\]
For a knot type $K$, set
\[
  \Rop(K)=\inf\{\Rop(\gamma)\mid \gamma\text{ represents }K\}.
\]
\end{definition}

For the standard thickness formula and existence of ropelength minimizers,
see \cite{LitherlandSimonDurumericRawdon,CantarellaKusnerSullivan}.  In
particular, the ideal level used below is nonempty.

Fix $u\in\mathbb S^2$.  Let $\Sim_u^+(\mathbb R^3)$ be the group of
orientation-preserving Euclidean similarities whose rotational part fixes
$u$.  Thus translations, positive homotheties, and rotations about the
$u$-axis are allowed, but rotations changing the projection direction are
not.

\begin{definition}[Projection-framed ropelength sublevel spaces]
Let $\mathcal E^{\mathrm{par}}(K)$ be the regular $C^{1,1}$ embeddings
$S^1=\mathbb R/\mathbb Z\to\mathbb R^3$ representing $K$, with the strong
$W^{2,\infty}$ norm topology.  Write $\Diff^{1,1}(S^1)$ for all
$C^{1,1}$ source diffeomorphisms and $\Diff^{1,1}_+(S^1)$ for its
orientation-preserving subgroup.  Unless otherwise stated, knots are
unoriented and we define
\[
  \mathcal E(K)=\mathcal E^{\mathrm{par}}(K)/\Diff^{1,1}(S^1)
\]
with the quotient topology; in the oriented variant one instead divides by
$\Diff^{1,1}_+(S^1)$.  This is a quotient of the fixed-parameter Banach
model, not a definition by normal graphs over nonsmooth curves.  Define the closed and
strict sublevels
\[
  \calX_{\Lambda,u}(K)
  =
  \{\gamma\in\mathcal E(K)\mid \Rop(\gamma)\le\Lambda\}
  /\Sim_u^+(\mathbb R^3),
\]
\[
  \calX^{<}_{\Lambda,u}(K)
  =
  \{\gamma\in\mathcal E(K)\mid \Rop(\gamma)<\Lambda\}
  /\Sim_u^+(\mathbb R^3).
\]
An admissible isotopy is represented by a jointly $C^{1,1}$ map
$\Gamma:S^1\times[0,1]\to\mathbb R^3$ whose time slices are
positive-thickness $C^{1,1}$ embeddings of knot type $K$ and for which
$|\partial_s\Gamma|$ is bounded away from zero.  Piecewise jointly
$C^{1,1}$ families are allowed after a finite subdivision of the time
interval.  It is $\Lambda$-admissible if
$\Rop(\Gamma_t)\le\Lambda$ for all $t$, and it is strictly
$\Lambda$-admissible if
\[
  \sup_{t\in[0,1]}\Rop(\Gamma_t)<\Lambda.
\]
Endpoints are compared after applying elements of
$\Sim_u^+(\mathbb R^3)$.  The resulting equivalence classes are called the
closed and strict projection-framed admissible components, respectively.
\end{definition}

\begin{remark}[Why strictness is uniform]
The uniform inequality in a strictly admissible isotopy is the analytic
slack needed for smoothing and transversality.  It is automatic whenever
ropelength is continuous along the chosen family and every time slice lies
in the strict sublevel, but it is included in the definition so that no
continuity assertion about reach is hidden in the notation.
\end{remark}

\begin{remark}[Strong paths and admissible isotopies]
\label{rem:strong-paths}
For a fixed parametrization, strong $W^{2,\infty}$ continuity (equivalently, fixed-parameter strong $C^{1,1}$ continuity) means continuity
in the Banach norm
$\|f\|_{2,\infty}=\|f\|_\infty+\|f'\|_\infty+
\|f''\|_{L^\infty}$; local boundedness of the tangent Lipschitz constant
alone is not this continuity condition.
By \cref{lem:strong-path-approx}, every such continuous path in a strict
ropelength sublevel is homotopic, relative to its endpoints and within that
sublevel, to an arbitrarily strongly close piecewise jointly $C^{1,1}$
path.  The admissible-isotopy definition is unchanged: joint $C^{1,1}$
regularity does not by itself imply strong continuity of the slices.
That approximation is in the fixed-parameter strong norm.  The separate
\cref{thm:uniform-slack-approx,cor:strict-path-comparison} below compare
endpoint connectivity also for quotient paths and for $C^1$ paths with a
uniform strict margin, without asserting a lift of the original path.
Pointwise strictness alone is not substituted for that uniform margin
in the $C^1$ topology.
\end{remark}

\begin{remark}[Why the parameterized quotient is specified]
\label{rem:normal-graph-limit}
A $C^{1,1}$ reference curve has only a Lipschitz normal field.  Normal
graphs over it need not give $C^{1,1}$ charts even for nearby translated
curves.  For example, join the line $(s,0)$ for $s\le0$ to the unit-speed
circle arc $(r\sin(s/r),r-r\cos(s/r))$ for $s\ge0$.
For its translate by $(0,\epsilon)$, $0<\epsilon<r$, the normal displacement
over the original curve has derivative $0$ from the line side and
$(-\epsilon/r,0)$ from the circle side at the join.  Thus that displacement
is not $C^1$.  The strong topology above is instead defined before taking
the source quotient.  Fixed source diffeomorphisms act by homeomorphisms,
so this quotient is well defined; joint continuity of the full
reparametrization action and path lifting are not assumed.
\end{remark}

\begin{figure}[t]
\centering
\begin{tikzpicture}
\draw[figsoft, densely dotted, line width=0.4pt] (0,1.900) -- (1.807,1.313);
\draw[figsoft, densely dotted, line width=0.4pt] (0,1.900) -- (0.587,0.093);
\draw[figsoft, densely dotted, line width=0.4pt] (0,1.900) -- (1.117,0.363);
\draw[figsoft, densely dotted, line width=0.4pt] (0,1.900) -- (1.537,0.783);
\draw[figsoft, densely dotted, line width=0.4pt] (0,1.900) -- (0,0);
\draw[knot, line width=1.0pt] (-3.600,0) -- (0,0) -- (0.040,0.000) -- (0.081,0.002) -- (0.121,0.004) -- (0.162,0.007) -- (0.202,0.011) -- (0.242,0.015) -- (0.282,0.021) -- (0.322,0.028) -- (0.362,0.035) -- (0.402,0.043) -- (0.441,0.052) -- (0.480,0.062) -- (0.519,0.072) -- (0.558,0.084) -- (0.597,0.096) -- (0.635,0.109) -- (0.673,0.123) -- (0.711,0.138) -- (0.748,0.153) -- (0.785,0.170) -- (0.822,0.187) -- (0.858,0.205) -- (0.894,0.223) -- (0.929,0.243) -- (0.965,0.263) -- (0.999,0.284) -- (1.033,0.306) -- (1.067,0.328) -- (1.100,0.351) -- (1.133,0.375) -- (1.165,0.399) -- (1.197,0.424) -- (1.228,0.450) -- (1.259,0.477) -- (1.289,0.504) -- (1.318,0.532) -- (1.347,0.560) -- (1.375,0.589) -- (1.403,0.619) -- (1.430,0.649) -- (1.456,0.680) -- (1.482,0.711) -- (1.507,0.743) -- (1.531,0.775) -- (1.555,0.808) -- (1.578,0.841) -- (1.600,0.875) -- (1.621,0.909) -- (1.642,0.944) -- (1.662,0.979) -- (1.681,1.015) -- (1.700,1.051) -- (1.717,1.087) -- (1.734,1.124) -- (1.750,1.161) -- (1.766,1.199) -- (1.780,1.236) -- (1.794,1.275) -- (1.807,1.313);
\draw[figink, line width=0.8pt, densely dashed] (-3.600,0.600) -- (0,0.600) -- (0.040,0.600) -- (0.081,0.602) -- (0.121,0.604) -- (0.162,0.607) -- (0.202,0.611) -- (0.242,0.615) -- (0.282,0.621) -- (0.322,0.628) -- (0.362,0.635) -- (0.402,0.643) -- (0.441,0.652) -- (0.480,0.662) -- (0.519,0.672) -- (0.558,0.684) -- (0.597,0.696) -- (0.635,0.709) -- (0.673,0.723) -- (0.711,0.738) -- (0.748,0.753) -- (0.785,0.770) -- (0.822,0.787) -- (0.858,0.805) -- (0.894,0.823) -- (0.929,0.843) -- (0.965,0.863) -- (0.999,0.884) -- (1.033,0.906) -- (1.067,0.928) -- (1.100,0.951) -- (1.133,0.975) -- (1.165,0.999) -- (1.197,1.024) -- (1.228,1.050) -- (1.259,1.077) -- (1.289,1.104) -- (1.318,1.132) -- (1.347,1.160) -- (1.375,1.189) -- (1.403,1.219) -- (1.430,1.249) -- (1.456,1.280) -- (1.482,1.311) -- (1.507,1.343) -- (1.531,1.375) -- (1.555,1.408) -- (1.578,1.441) -- (1.600,1.475) -- (1.621,1.509) -- (1.642,1.544) -- (1.662,1.579) -- (1.681,1.615) -- (1.700,1.651) -- (1.717,1.687) -- (1.734,1.724) -- (1.750,1.761) -- (1.766,1.799) -- (1.780,1.836) -- (1.794,1.875) -- (1.807,1.913);
\draw[wallRtwo, line width=0.7pt, -{Latex[length=3pt]}] (-3.000,0.000) -- (-3.000,0.600);
\draw[wallRtwo, line width=0.7pt, -{Latex[length=3pt]}] (-2.000,0.000) -- (-2.000,0.600);
\draw[wallRtwo, line width=0.7pt, -{Latex[length=3pt]}] (-1.000,0.000) -- (-1.000,0.600);
\draw[wallRtwo, line width=0.7pt, -{Latex[length=3pt]}] (0.000,0.000) -- (0.000,0.600);
\draw[wallRtwo, line width=0.7pt, -{Latex[length=3pt]}] (0.587,0.093) -- (0.408,0.644);
\draw[wallRtwo, line width=0.7pt, -{Latex[length=3pt]}] (1.117,0.363) -- (0.812,0.782);
\draw[wallRtwo, line width=0.7pt, -{Latex[length=3pt]}] (1.537,0.783) -- (1.201,1.028);
\fill[figink] (0,1.900) circle (1.2pt) node[left, font=\tiny] {centre $(0,r)$};
\fill[figink] (0,0) circle (1.3pt);
\node[font=\tiny, below] at (0,0) {join};
\node[font=\scriptsize, below] at (-3.4,0) {$c$};
\node[font=\scriptsize, above] at (-3.0,0.600) {$c+(0,\epsilon)$};
\node[font=\tiny, align=center, text=figink] at (-1.5,-0.5) {half-line: $n$ constant,\quad $\tfrac{d}{ds}(h\,n)=0$};
\node[font=\tiny, align=left, text=figink, anchor=west] at (2.25,0.55) {arc: $n$ turns at rate $1/r$,\\ $\tfrac{d}{ds}(h\,n)\to(-\epsilon/r,0)$ as $s\downarrow0$};
\end{tikzpicture}
\caption{The example of \cref{rem:normal-graph-limit}, with $r$ and
$\epsilon$ to scale.  The $C^{1,1}$ curve $c$ is a half-line joined to a circular
arc of radius~$r$; its translate $c+(0,\epsilon)$ is dashed.  The arrows are
the normal displacements $h\,n$ over $c$.  The normal field $n$ of $c$ is only
Lipschitz: it is constant along the half-line and turns along the arc, so the
displacement is not $C^1$ at the join although $h$ is.  This is why the strong
topology is defined on parametrized curves before the source quotient is
taken, and not by normal graphs.}
\label{fig:normal-graph}
\end{figure}

\begin{definition}[Two comparison topologies]
\label{def:comparison-topologies}
On oriented unparametrized curves write $\widehat x$ for a constant-speed
representative on $\mathbb R/\mathbb Z$ and put
\[
 \dOne(x,y)=\min_{\theta\in S^1}
    \|\widehat x-\widehat y(\,\cdot+\theta)\|_{C^1},\qquad
 \dStar(x,y)=\dOne(x,y)+|\Kmax(x)-\Kmax(y)|,
\]
where $\Kmax$ is the essential supremum of arclength curvature.
Reversal is included in the minimum in the unoriented theory.
Both are metrics.  For the similarity classes defining $\calX$, first
translate to arclength centroid zero and dilate to length one, and then
minimize also over $SO(2)_u$.  Denote the same sublevel sets with these
two metric topologies by $\calX^{C^1}_{\Lambda,u}(K)$ and
$\calX^\star_{\Lambda,u}(K)$, respectively; a superscript $<$ still
denotes the strict bound.  The unadorned $\calX$ retains the strong
quotient topology defined above.
The identity maps from the strong quotient to the $\star$ model and then
to the $C^1$ model are continuous, by
\cref{lem:comparison-prelim}.  No converse continuity is asserted.
\end{definition}

\subsection{The normalised model, and compactness}

The similarity quotient above is what the analytic arguments of
\cref{sec:growth-graphs} use: the time-dependent homothety
$\Gamma_t\mapsto\Len(\Gamma_t)^{-1}\Gamma_t$ in the proof of
\cref{prop:parametric-regularization} is an admissible change of
representative only because homotheties lie in $\Sim^+_u(\mathbb R^3)$.
Compactness, on the other hand, is a statement about the \emph{normalised}
model of \cite[Definition~6.1]{OzawaRFKS}.  We introduce both and compare
them.

\begin{definition}[Normalised projection-framed sublevels]
\label{def:normalised}
Represent configurations by constant-speed $C^{1,1}$ embeddings
$\gamma:S^1=\mathbb R/2\pi\mathbb Z\to\mathbb R^3$ with arclength centroid at
the origin, and let
\[
  W_\Lambda(K)=\bigl\{\gamma\ \text{of type}\ K\ :\
  \Thi(\gamma)\ge1,\ \Len(\gamma)\le\Lambda\bigr\}
\]
with the constant-speed $C^1$ topology.  Put
\[
  \calY_{\Lambda,u}(K)
  = W_\Lambda(K)\big/\bigl(\SO(2)_u\times\Ogp(2)\bigr),
  \qquad
  Y_\Lambda(K)
  = W_\Lambda(K)\big/\bigl(\SO(3)\times\Ogp(2)\bigr),
\]
where $\SO(2)_u\le\SO(3)$ is the stabiliser of $u$, the second factor acts by
rotation and reflection of the source circle, and $Y_\Lambda(K)$ is the
normalised sublevel of \cite[Definition~6.1]{OzawaRFKS}.
\end{definition}

Note $\Rop(\gamma)=\Len(\gamma)/\Thi(\gamma)\le\Lambda$ for
$\gamma\in W_\Lambda(K)$, and conversely every class with $\Rop\le\Lambda$ has
a representative in $W_\Lambda(K)$, obtained by the homothety of ratio
$\Thi(\gamma)^{-1}$.  Hence there is a canonical surjection
\[
  \sigma_u:\ \calY_{\Lambda,u}(K)\longrightarrow\calX_{\Lambda,u}(K),
\]
which forgets the choice of scale.  As a map to
$\calX^{C^1}_{\Lambda,u}(K)$ it is continuous: arclength centroid and
length normalization are continuous in $C^1$.  Continuity into the strong
quotient or $\star$ topology is not asserted.

\begin{theorem}[Compactness of the normalised projection-framed sublevels]
\label{thm:sublevel-compact}
For every knot type $K$ and every finite $\Lambda\ge\Rop(K)$, the spaces
$W_\Lambda(K)$ and $\calY_{\Lambda,u}(K)$ are compact and Hausdorff.
Consequently the coarser model $\calX^{C^1}_{\Lambda,u}(K)$ is compact
Hausdorff.  No compactness assertion for the strong quotient
$\calX_{\Lambda,u}(K)$ or its $\star$ model is made.
\end{theorem}

\begin{proof}
The first assertion is \cite[Theorem~6.5]{OzawaRFKS} together with one
change, which we record.  Let $\gamma\in W_\Lambda(K)$ and let
$v=\Len(\gamma)/2\pi$ be its constant speed.  A closed curve of length at most
$\Lambda$ has diameter at most $\Lambda/2$, so all centred representatives
lie in one fixed ball.  Thickness at least one forces $|\kappa|\le1$ almost
everywhere, whence $\Len(\gamma)\ge2\pi$ by Fenchel's theorem
\cite{Fenchel} (see \cref{prop:unknot-ideal}), so $1\le v\le\Lambda/2\pi$ and
\[
  |\gamma'|\equiv v\le\frac{\Lambda}{2\pi},
  \qquad
  \operatorname{Lip}(\gamma')\le v^{2}\|\kappa\|_\infty
  \le\Bigl(\frac{\Lambda}{2\pi}\Bigr)^{2}.
\]
So $W_\Lambda(K)$ is bounded and equicontinuous in $C^1$ and every sequence
has a $C^1$-convergent subsequence.  The condition $\Thi\ge1$ is closed under
uniform convergence, by the three-point characterisation of thickness
\cite{GonzalezMaddocks}; length is continuous in $C^1$; and curves of
thickness at least one that are sufficiently $C^1$-close are ambient isotopic
\cite[\S2]{CantarellaKusnerSullivan}.  Hence the limit lies in
$W_\Lambda(K)$, which is therefore compact Hausdorff.

For $\calY_{\Lambda,u}(K)$, the residual symmetry group is
$\SO(2)_u\times\Ogp(2)$, whereas \cite{OzawaRFKS} quotients by
$\SO(3)\times\Ogp(2)$.  The only property of the group used is compactness,
and $\SO(2)_u\times\Ogp(2)$ is a closed, hence compact, subgroup.  The orbit
space of a compact Hausdorff space by a compact group is compact Hausdorff.
The $C^1$ model is the metric orbit space of centred, length-one
representatives under a compact isometry group.  The continuous surjection
from $\calY_{\Lambda,u}(K)$ therefore makes it compact Hausdorff.
\end{proof}

\begin{proposition}[Direction-marked description]
\label{prop:direction-marked}
The map $\gamma\mapsto(\gamma,u)$ induces a homeomorphism
\[
  \calY_{\Lambda,u}(K)
  \;\xrightarrow{\ \cong\ }\;
  \calP_\Lambda(K):=
  \bigl(W_\Lambda(K)\times\Sph^2\bigr)
  \big/\bigl(\SO(3)\times\Ogp(2)\bigr),
\]
where $\SO(3)$ acts diagonally and tautologically on $\Sph^2$, and $\Ogp(2)$
acts trivially on $\Sph^2$.  Under this identification the forgetful map
\[
  q_u:\calY_{\Lambda,u}(K)\longrightarrow Y_\Lambda(K)
\]
is a closed surjection whose fibre over the class of $\gamma$ is
$\Sph^2/\Gamma_\gamma$, where $\Gamma_\gamma\le\SO(3)$ is the group of
rotations preserving the centred curve $\gamma$.
\end{proposition}

\begin{proof}
For a compact group $G$ acting on a compact Hausdorff space $X$ and a closed
subgroup $H\le G$, the map $x\mapsto(x,eH)$ induces a continuous bijection
$X/H\to(X\times G/H)/G$ for the diagonal action; both spaces are compact
Hausdorff, so it is a homeomorphism.  Apply this with $X=W_\Lambda(K)$,
$G=\SO(3)\times\Ogp(2)$ and $H=\SO(2)_u\times\Ogp(2)$, using
$G/H\cong\SO(3)/\SO(2)_u\cong\Sph^2$.  The description of the fibres of $q_u$
is the orbit--stabiliser statement for the $\SO(3)$-action on $\Sph^2$.
\end{proof}

\begin{remark}\label{rem:pole-analogue}
Direction marking should be distinguished from the pole-marked complement
bundle before taking ambient quotients. The latter is a genuine bundle by
isotopy extension. The direction-marked coarse quotient here has fibres
$S^2/\Gamma_\gamma$, and varying stabilizers preclude a general sphere-bundle
assertion after the ambient quotient.
\end{remark}

\begin{figure}[t]
\centering
\begin{tikzpicture}[>={Latex[length=4pt]}, font=\scriptsize]
\node[inner sep=0pt, anchor=south west] at (0,0)
  {\includegraphics[width=5.20cm]{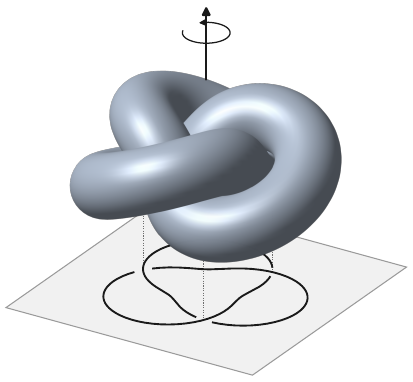}};
\node[panellabel] at (-0.1,4.854) {(a)};
\node[anchor=west] at (2.680,4.609) {$u$};
\node[anchor=east, font=\tiny] at (2.133,4.280) {$\SO(2)_u$};
\node[anchor=west] at (4.232,2.330) {$\gamma$};
\node[font=\tiny, anchor=east, text=figmid] at (5.005,1.438) {$u^\perp$};
\node[font=\tiny, anchor=west] at (3.805,0.845) {$\Pi_u(\gamma)$};
\begin{scope}[shift={(6.3,0.35)}]
\node[panellabel] at (-0.35,4.504) {(b)};
\def\R{0.52}
\draw[line width=0.6pt] (0,0) -- (4.95,0);
\node[font=\tiny, anchor=north] at (2.475,-0.36) {$q_u^{-1}(\text{round circle})=\Sph^2/\Ogp(2)\cong[0,\pi/2]$};
\foreach \k/\deg/\lab in {0/0/{0}, 1/30/{\pi/6}, 2/60/{\pi/3}, 3/90/{\pi/2}}{
  \pgfmathsetmacro{\xx}{1.65*\k}
  \pgfmathsetmacro{\cc}{cos(\deg)}
  \pgfmathsetmacro{\ss}{sin(\deg)}
  \draw (\xx,0.07) -- (\xx,-0.07) node[below, font=\tiny] {$\lab$};
  \begin{scope}[shift={(\xx,0.95)}]
    \ifnum\k<3
      \draw[knot] (0,0) ellipse ({\R} and {\R*\cc});
    \else
      \draw[figink, line width=0.45pt, double=white, double distance=0.6pt] (-\R,0) -- (\R,0);
      \fill[figink] (-\R,0) circle (0.9pt) (\R,0) circle (0.9pt);
    \fi
  \end{scope}
  \begin{scope}[shift={(\xx,2.55)}]
    \draw[figsoft, line width=0.5pt] (0,0) ellipse [x radius=\R, y radius=0.09, rotate=-\deg];
    \draw[knot] ({-\R*\cc},{\R*\ss}) -- ({\R*\cc},{-\R*\ss});
    \draw[figmid, densely dashed, line width=0.45pt] ({-0.62*\ss},{-0.62*\cc}) -- ({0.62*\ss},{0.62*\cc});
  \end{scope}
}
\begin{scope}[shift={(1.65,2.55)}]
  \draw[figmid, line width=0.4pt] (0,0.42) arc[start angle=90, end angle=60, radius=0.42];
  \node[font=\tiny] at (0.13,0.58) {$\varphi$};
\end{scope}
\node[font=\tiny, text=figmid] at (0.28,3.25) {$a$};
\draw[->, line width=0.6pt] (-0.62,2.05) -- (-0.62,3.05) node[above, font=\tiny] {$u$};
\node[font=\tiny, anchor=east] at (-0.42,0.95) {$\Pi_u$:};
\node[font=\tiny, text=figmid, anchor=west] at (5.5,0.95) {$\times2$};
\end{scope}
\end{tikzpicture}
\caption{Projection frames.  (a)~A projection-framed configuration: the thick
trefoil of \cref{fig:rotation-R1}, drawn with its normal tube of radius
$\Thi(\gamma)$, a marked direction $u$, and the regular diagram
$\Pi_u(\gamma)$ in $u^\perp$ (dotted lines project the three crossings).  Only
similarities whose rotational part fixes $u$ are divided out, as indicated by
the rotation about the $u$-axis, so the diagram is defined up to
orientation-preserving similarity of the projection plane.  (b)~The fibre of the
forgetful map $q_u$ of \cref{prop:direction-marked} over the round unit
circle, the only point of $Y_{2\pi}(U)$.  The circle is preserved by a
subgroup $\Ogp(2)\le\SO(3)$, so the fibre is $\Sph^2/\Ogp(2)\cong[0,\pi/2]$,
parametrized by the angle $\varphi$ between $u$ and the axis $a$ of the circle
(top row: side views).  The projections are ellipses with semi-axes $1$ and
$\cos\varphi$, regular and crossingless for $\varphi<\pi/2$; at $\varphi=\pi/2$ the curve collapses onto a
doubly covered segment, a non-regular projection that is not a
codimension-one Reidemeister wall (\cref{prop:unknot-ideal,rem:unknot-different}).}
\label{fig:framed-projection}
\end{figure}

\begin{remark}[What $\sigma_u$ does and does not give]
\label{rem:sigma-caveat}
The compactness argument uses the constant-speed $C^1$ topology.
It does not transfer compactness to a stronger topology on the same
similarity classes.  In particular, a $C^1$-continuous family with a
curvature bound need not be continuous in the strong quotient or the
$\star$ topology, where the curvature supremum is continuous.
Moreover, $\calY$ retains a scale whereas $\calX$ forgets it, and the
exact-thickness normalization need not be $C^1$-continuous because
thickness is only upper semicontinuous in $C^1$.
All uses of compactness in the merge theory are made in $\calY$.
The comparisons of strict admissible components in
\cref{cor:strict-path-comparison} are comparisons of endpoint relations,
not homeomorphisms between these spaces.
\end{remark}

\begin{definition}[Projection-framed ideal stratum]
The projection-framed ideal stratum is
\[
  I_u(K)=\calX_{\Rop(K),u}(K).
\]
Its path components are called ideal projection-framed components.
\end{definition}

The ideal stratum $I(K)$ of \cite{OzawaRFKS} is the unframed quotient,
obtained by allowing all rotations; by \cref{prop:direction-marked} the
present space is its direction-marked refinement, and $q_u$ is the forgetful
map.  All component and merge statements in this paper refer to the framed
spaces unless explicitly stated otherwise; no identification with the
unframed rotation quotient is silently assumed.

\begin{proposition}[Independence of the chosen direction]
\label{prop:direction-independence}
For $u,v\in\mathbb S^2$ and a rotation $A$ with $Au=v$, the assignment
$[\gamma]\mapsto[A\gamma]$ gives a ropelength-preserving homeomorphism
\[
  \calX_{\Lambda,u}(K)\cong\calX_{\Lambda,v}(K)
\]
that intertwines the corresponding diagrammatic constructions, up to
identifying the two projection spheres by an orientation-preserving
isometry.
Consequently, all isomorphism types and numerical invariants defined below
are independent of $u$, although the subscript is retained to display the
projection frame.
\end{proposition}

\begin{proof}
Conjugation by $A$ carries $\Sim_u^+(\mathbb R^3)$ to
$\Sim_v^+(\mathbb R^3)$ and preserves length, thickness, and ropelength.
Moreover, projection of $A\gamma$ in direction $v$ is identified with
projection of $\gamma$ in direction $u$ by the induced orientation-
preserving isometry of projection planes.  The inverse is induced by
$A^{-1}$.
\end{proof}

\subsection{Finite-scale normalization}

Set
\[
  \calX_{<\infty,u}(K)=\bigcup_{\Lambda<\infty}\calX_{\Lambda,u}(K),
  \qquad
  \calX_{\infty,u}(K):=\calX_{<\infty,u}(K).
\]
When used as a topological space, $\calX_{\infty,u}(K)$ carries the strong
quotient topology inherited from the full projection-framed $C^{1,1}$ knot
space, not an abstract direct-limit topology.  Since every point has finite
ropelength, one also has
\[
  \calX_{\infty,u}(K)
  =\bigcup_{\Lambda<\infty}\calX^{<}_{\Lambda,u}(K)
\]
as sets.
For a tame knot type this space is nonempty: a polygonal representative can
be smoothed to a positive-thickness $C^{1,1}$ representative and hence has
finite ropelength.  This elementary fact is only a normalization.  The
substantive finite information below is the finite scale at which a complete
rooted Reidemeister-ball certificate becomes visible.


\section{Classical Reidemeister multigraphs and finite certificates}
\label{sec:classical}

Let $K$ be an unoriented knot type in $S^3$, and let $\calD_S(K)$ be the
set of its diagrams on $S^2$, considered up to orientation-preserving
isotopy of $S^2$.

\begin{definition}[$S^2$-Reidemeister multigraph]
The $S^2$-Reidemeister multigraph $G_S(K)$ has vertex set $\calD_S(K)$.
Its edges are equivalence classes of single Reidemeister moves
\cite{Reidemeister}, where moves coinciding under diagram isotopy define the
same edge.  Distinct moves with
the same endpoints remain distinct parallel edges.
\end{definition}

Retaining parallel edges is essential: the local reconstruction in
Barbensi--Celoria uses edge multiplicities and valences.  The multigraph is
connected and locally finite \cite{BarbensiCeloria}.

\begin{theorem}[Barbensi--Celoria \cite{BarbensiCeloria}]
\label{thm:BC-complete}
The isomorphism type of $G_S(K)$ determines $K$ up to mirroring.
\end{theorem}

\begin{definition}[Full rooted ball]
Let $G$ be a locally finite multigraph, $D\in V(G)$, and $R\in\mathbb N$.
Set
\[
 V_R^G(D)=\{E\in V(G):d_G(D,E)\le R\},
 \qquad
 S_R^G(D)=G[V_R^G(D)],
\]
where $G[V_R^G(D)]$ is the induced submultigraph, with all parallel edges
retained, and the distinguished root is $D$.  When $G=G_S(K)$ is understood
we write simply $S_R(D)$.
\end{definition}

\begin{convention}[Full rooted-ball occurrence]
\label{conv:full-ball-occurrence}
An occurrence of $S_R^G(D)$ in a second locally finite multigraph $G'$ means
a root-preserving multigraph isomorphism
\[
  (S_R^G(D),D)\cong (S_R^{G'}(D'),D')
\]
for some $D'\in V(G')$.  Thus ``occurrence'' in this paper means occurrence
as the complete radius-$R$ rooted ball, not an arbitrary injective subgraph
copy.  This convention makes precise the full-ball verification used below;
it is not intended as a claim that Barbensi--Celoria separately defined the
word ``appear'' in this terminology.
\end{convention}

For $G=G_S(K)$ we retain, when needed for filtered verification, the
intrinsically recoverable Reidemeister type of each edge and the ambient
multigraph valence
\[
  \nu(E)=\val_{G_S(K)}(E).
\]

\begin{definition}[Rooted Reidemeister-ball certificate]
The decorated rooted ball $\mathbb B_R(D)$ is $S_R(D)$ together with its
Reidemeister edge types, parallel-edge multiplicities, and the ambient
valence labels
\[
  \nu_Q(E):=\val_{G_S(K)}(E)
  \qquad\bigl(d_{G_S(K)}(D,E)<R\bigr)
\]
on the interior vertices.  Boundary valences are not part of the
certificate.  The underlying characteristic object is the full rooted
multigraph ball; the interior labels are finite verification data used to
test whether a visible copy has omitted an edge leaving an interior vertex.
\end{definition}

\begin{figure}[t]
\centering
\begin{tikzpicture}[
  font=\scriptsize,
  v/.style={circle, fill=figink, inner sep=0pt, minimum size=3.3pt},
  root/.style={circle, draw=figink, fill=white, line width=0.8pt, inner sep=0pt, minimum size=5.5pt},
  nu/.style={font=\tiny, text=figmid, inner sep=0.5pt},
  stub/.style={figsoft, line width=0.5pt},
  extra/.style={figink, line width=1.1pt, densely dashed},
  ring/.style={figlight, line width=0.5pt}
]
\foreach \px/\mode in {0/0, 4.35/1, 8.7/2}{
\begin{scope}[shift={(\px,0)}]
  \draw[ring] (0,0) circle (0.78);
  \draw[ring] (0,0) circle (1.52);
  \coordinate (D) at (0,0);
  \coordinate (E1) at (150:0.78); \coordinate (E2) at (30:0.78); \coordinate (E3) at (270:0.78);
  \coordinate (F1) at (195:1.52); \coordinate (F2) at (90:1.52);
  \coordinate (F3) at (335:1.52); \coordinate (F4) at (240:1.52);
  \foreach \v/\a in {F1/170, F1/195, F1/220, F2/78, F2/102, F3/335, F4/240}{
    \draw[stub] (\v) -- ++(\a:0.3);
  }
  \draw[Rone] (D) -- (E1);
  \draw[Rtwo] (D) to[bend left=22] (E2);
  \draw[Rtwo] (D) to[bend right=22] (E2);
  \draw[Rthree] (D) -- (E3);
  \draw[Rtwo] (E1) -- (E2);
  \draw[Rtwo] (E1) -- (F1);
  \draw[Rone] (E1) -- (F2);
  \draw[Rthree] (E2) -- (F2);
  \draw[Rone] (E2) -- (F3);
  \draw[Rtwo] (E3) -- (F3);
  \draw[Rone] (E3) -- (F4);
  \draw[Rthree] (F3) -- (F4);
  \ifnum\mode=1
    \coordinate (X) at (205:1.02);
    \draw[extra] (E1) -- (X);
    \node[circle, draw=figink, fill=white, line width=0.9pt, inner sep=0pt, minimum size=3.6pt] at (X) {};
  \fi
  \ifnum\mode=2
    \draw[extra] (F1) to[bend left=18] (F2);
  \fi
  \foreach \v in {E1,E2,E3,F1,F2,F3,F4}{ \node[v] at (\v) {}; }
  \node[root] at (D) {};
  \node[font=\scriptsize, anchor=north west, inner sep=0.5pt] at ([xshift=3pt,yshift=-1pt]D) {$D$};
  \node[nu, anchor=north east] at ([xshift=-2pt,yshift=-1pt]D) {$4$};
  \node[nu, anchor=east] at ([xshift=-2.5pt]E1) {$4$};
  \node[nu, anchor=south west] at ([xshift=1pt,yshift=1pt]E2) {$5$};
  \node[nu, anchor=east] at ([xshift=-2pt]E3) {$3$};
\end{scope}
}
\node[anchor=north] at (0,-1.9) {(a) certificate $\mathbb B_2(D)$};
\node[anchor=north, align=center] at (4.35,-1.9) {(b) extra incident edge:\\ not a full ball};
\node[anchor=north, align=center] at (8.7,-1.9) {(c) hidden boundary edge:\\ visible, not saturated};
\end{tikzpicture}
\caption{Rooted Reidemeister-ball certificates (schematic, with small
valences).  (a)~A full radius-$2$ ball $\mathbb B_2(D)$: all edges between its
vertices are included, parallel edges are kept, edges are typed, and each
interior vertex carries its ambient valence $\nu$ (grey numbers).  At
vertices of distance less than $2$ the full ambient valence is realized
inside the ball.  The grey stubs at boundary vertices indicate allowed edges
leaving the ball; boundary valences are not part of the certificate.  (b)~An injective copy of the same graph in which a vertex at
distance one has an additional incident edge (dashed) is not an occurrence of
$\mathbb B_2(D)$: the full ball is larger and the valence label is violated.
(c)~All edges of the certificate are visible, but the completed ambient
graph contains an additional edge joining two boundary vertices (dashed).
The visible copy is therefore not the full induced radius-$2$ ball and is not
saturated in the sense of \cref{sec:finite-recognition}.  The line styles for
$R_1$, $R_2$, $R_3$ are those of \cref{fig:cerf-walls}.}
\label{fig:rooted-ball}
\end{figure}

\begin{corollary}[Finite local characterization, after Barbensi--Celoria]
\label{thm:BC-finite-local}
For every $D\in G_S(K)$ there exists $R=R(D)>0$ such that the finite
certificate $\mathbb B_R(D)$ is characterizing: if it occurs as a full
rooted ball in $G_S(K')$, then
\[
  K'=K\quad\text{or}\quad K'=\mirror K.
\]
\end{corollary}

\begin{proof}
Barbensi--Celoria's Corollary~5.4 says that, for every vertex diagram $D$, a
sufficiently large finite neighborhood is characterizing: the finite graph
specified there can only appear in the $S^2$-Reidemeister graph of the same
knot type, up to the unavoidable mirror ambiguity.  A full rooted-ball
occurrence in the sense of \cref{conv:full-ball-occurrence} is in particular
such an occurrence, so the stated implication follows.  Their
Theorem~3.23 shows that Reidemeister move types and crossing-number data are
recoverable from the local graph.  The ambient-valence decoration in
$\mathbb B_R(D)$ is added here only as finite verification data for the
filtered setting; it is not an extra hypothesis of the Barbensi--Celoria
theorem.
\end{proof}

\begin{definition}[Finite Reidemeister certificate]
A finite Reidemeister certificate is a finite rooted typed multigraph with
edge multiplicities and finite vertex decorations.  It is
\emph{characteristic for $K$} if every full rooted-ball occurrence in an
$S^2$-Reidemeister multigraph forces the knot type to be $K$ or $\mirror K$.
The certificates $\mathbb B_R(D)$ supplied by
\cref{thm:BC-finite-local} are called BC-characteristic certificates.
\end{definition}

\begin{corollary}[Existence]
\label{cor:finite-pattern-existence}
Every knot type admits a finite BC-characteristic certificate.
\end{corollary}

\subsection{Why arbitrary subgraph occurrence is insufficient}

Two distinctions will be used throughout the paper.

First, an arbitrary injective copy of a finite graph is too weak for local
reconstruction: extra incident edges in the ambient Reidemeister multigraph
can change precisely the valence and short-cycle data used in the proof.
Second, vertices in the geometric graph are internally represented by
actual diagrams, but the recognition certificate forgets those labels.  A
labeled diagram already determines its knot type, so retaining diagram
identity would make the word ``recognition'' tautological.  Recognition
below is therefore recognition from the finite abstract multigraph data and
its intrinsic finite decorations.

\begin{proposition}[No nontrivial complete-ball inclusion order]
\label{prop:no-global-order}
If a BC-characteristic certificate for $K$ occurs as a full rooted ball in
$G_S(K')$, then $K'=K$ or $K'=\mirror K$.  Hence any preorder on completed
Reidemeister multigraphs whose comparison maps preserve such complete
rooted-ball data collapses to equality up to mirroring.
\end{proposition}

\begin{proof}
This is exactly the characterizing property of the certificate.
\end{proof}

The meaningful object is therefore not inclusion of completed graphs, but
the scale at which complete finite certificates become geometrically
available.


\section{Generic projections, Reidemeister walls, and Cerf graphics}
\label{sec:projection}

Here a projection always means the ordinary operation of obtaining a regular diagram from a spatial knot.  Fix \(u\in\Sph^2\), and let
\[
  p_u:\R^3\to u^\perp
\]
be orthogonal projection.  The main point of this section is that changes in the projected Reidemeister graph occur when the projected curve crosses the discriminant of non-regular projections.  This is a Cerf-type viewpoint, modeled on Cerf's stratification theory \cite{Cerf} and analogous in spirit to tracking critical events in Morse theory \cite{Milnor}.  In a controlled generic one-parameter family, the standard codimension-one pieces met transversely are the local events corresponding to the Reidemeister moves; \cref{fig:cerf-walls} gives the schematic picture used below.  No global stratification theorem for the full constrained $C^{1,1}$ space is assumed here.

\subsection{Regular projections}

\begin{definition}[Regular projection]
For a representative \(\gamma\subset\R^3\), the projection \(p_u(\gamma)\) is regular if it has only finitely many transverse double points and has no triple point, tangency, or cusp-type singularity.  The projected image with over-under crossing information is denoted by \(\Pi_u(\gamma)\).
\end{definition}

\begin{convention}[Regularity of paths]
Projection-genericity statements are made for admissible paths that are sufficiently regular for the usual singularity analysis of planar projections.  For example, one may work with piecewise \(C^2\) thick embeddings and piecewise \(C^2\) one-parameter families, or with smooth representatives reached by the relative smoothing results below.  No density assertion in the strong \(C^{1,1}\) norm is used.
\end{convention}

\begin{definition}[Projection discriminant]
For a projection direction \(u\), the projection discriminant at level \(\Lambda\) is
\[
  \Delta_{\Lambda,u}(K)
  =
  \{x\in\calX_{\Lambda,u}(K)\mid \Pi_u(x)\text{ is not a regular diagram}\}.
\]
Its complement
\[
  \calX_{\Lambda,u}^{\operatorname{reg}}(K)
  =
  \calX_{\Lambda,u}(K)\setminus\Delta_{\Lambda,u}(K)
\]
is the regular projection locus.
\end{definition}

The connected regions of \(\calX_{\Lambda,u}^{\operatorname{reg}}(K)\) are domains on which the combinatorial type of the projected diagram is locally constant.  Thus the Reidemeister graph changes only when an admissible path meets \(\Delta_{\Lambda,u}(K)\).

\subsection{Reidemeister walls}

\begin{definition}[Reidemeister wall]
A point \(x\in \Delta_{\Lambda,u}(K)\) lies on a Reidemeister wall if a representative of \(x\) has projection with exactly one standard generic codimension-one singularity and is otherwise regular.  The wall is said to be:
\begin{enumerate}[label=(\roman*)]
  \item of type \(R1\) if the singularity is a cusp-type event producing or cancelling a small loop;
  \item of type \(R2\) if the singularity is a self-tangency of two projected branches producing or cancelling a pair of crossings;
  \item of type \(R3\) if the singularity is a transverse triple point producing a Reidemeister III transition.
\end{enumerate}
We denote the corresponding parts of the discriminant by
\[
  \Delta_{\Lambda,u}^{R1}(K),\qquad
  \Delta_{\Lambda,u}^{R2}(K),\qquad
  \Delta_{\Lambda,u}^{R3}(K).
\]
\end{definition}

\begin{figure}[t]
\centering
\begin{tikzpicture}[>={Latex[length=4pt]}, font=\scriptsize]
\node[panellabel] at (-0.62,1.2) {(a)};
\begin{scope}[shift={(0.000,0)}]
\draw[knot] (-0.694,0.157) .. controls (-0.3,-0.02) and (-0.16,-0.2) .. (0,-0.2) .. controls (0.16,-0.2) and (0.3,-0.02) .. (0.694,0.157);
\end{scope}
\begin{scope}[shift={(1.380,0)}]
\draw[knot] (-0.444,0.157) -- (-0.267,0.054) -- (-0.195,0.005) -- (-0.131,-0.042) -- (-0.079,-0.089) -- (-0.039,-0.131) -- (-0.012,-0.170) -- (-0.000,-0.202) -- (0.012,-0.170) -- (0.039,-0.131) -- (0.079,-0.089) -- (0.131,-0.042) -- (0.195,0.005) -- (0.267,0.054) -- (0.444,0.157);
\fill[wallRone] (0.000,-0.202) circle (1.5pt);
\end{scope}
\begin{scope}[shift={(2.760,0)}]
\draw[knot] (-0.194,0.157) -- (-0.114,0.100) -- (-0.049,0.047);
\draw[knot] (0.042,-0.053) -- (0.058,-0.081) -- (0.068,-0.108) -- (0.072,-0.132) -- (0.070,-0.154) -- (0.060,-0.173) -- (0.045,-0.188) -- (0.024,-0.198) -- (0.000,-0.202);
\draw[knot] (-0.000,-0.202) -- (-0.022,-0.199) -- (-0.042,-0.190) -- (-0.057,-0.178) -- (-0.067,-0.161) -- (-0.072,-0.142) -- (-0.071,-0.120) -- (-0.065,-0.096) -- (-0.053,-0.071) -- (-0.017,-0.019) -- (0.037,0.037) -- (0.108,0.096) -- (0.194,0.157);
\end{scope}
\begin{scope}[shift={(4.100,0)}]
\draw[knot] (-0.437,-0.290) -- (-0.384,-0.224) -- (-0.330,-0.165) -- (-0.275,-0.115) -- (-0.220,-0.073) -- (-0.165,-0.041) -- (-0.111,-0.019) -- (-0.056,-0.005) -- (-0.001,-0.000) -- (0.054,-0.004) -- (0.108,-0.018) -- (0.163,-0.040) -- (0.218,-0.072) -- (0.273,-0.113) -- (0.327,-0.163) -- (0.382,-0.222) -- (0.437,-0.290);
\draw[knot] (-0.437,0.402) -- (-0.384,0.336) -- (-0.330,0.277) -- (-0.275,0.227) -- (-0.220,0.185) -- (-0.165,0.153) -- (-0.111,0.131) -- (-0.056,0.117) -- (-0.001,0.112) -- (0.054,0.116) -- (0.108,0.130) -- (0.163,0.152) -- (0.218,0.184) -- (0.273,0.225) -- (0.327,0.275) -- (0.382,0.334) -- (0.437,0.402);
\end{scope}
\begin{scope}[shift={(5.480,0)}]
\draw[knot] (-0.437,-0.290) -- (-0.351,-0.187) -- (-0.308,-0.144) -- (-0.264,-0.106) -- (-0.222,-0.075) -- (-0.178,-0.048) -- (-0.135,-0.028) -- (-0.091,-0.013);
\draw[knot] (0.091,-0.013) -- (0.135,-0.028) -- (0.176,-0.047) -- (0.220,-0.073) -- (0.264,-0.106) -- (0.308,-0.144) -- (0.351,-0.187) -- (0.437,-0.290);
\draw[knot] (-0.437,0.290) -- (-0.384,0.224) -- (-0.330,0.165) -- (-0.275,0.115) -- (-0.220,0.073) -- (-0.165,0.041) -- (-0.111,0.019) -- (-0.056,0.005) -- (-0.001,0.000) -- (0.054,0.004) -- (0.108,0.018) -- (0.163,0.040) -- (0.218,0.072) -- (0.273,0.113) -- (0.327,0.163) -- (0.382,0.222) -- (0.437,0.290);
\fill[wallRtwo] (0.000,0.000) circle (1.5pt);
\end{scope}
\begin{scope}[shift={(6.860,0)}]
\draw[knot] (-0.437,-0.290) -- (-0.391,-0.232) -- (-0.345,-0.180) -- (-0.299,-0.136) -- (-0.253,-0.097);
\draw[knot] (-0.126,-0.024) -- (-0.065,-0.006) -- (-0.001,-0.000) -- (0.062,-0.006) -- (0.126,-0.024);
\draw[knot] (0.253,-0.097) -- (0.299,-0.136) -- (0.345,-0.180) -- (0.391,-0.232) -- (0.437,-0.290);
\draw[knot] (-0.437,0.178) -- (-0.384,0.112) -- (-0.330,0.053) -- (-0.275,0.003) -- (-0.220,-0.039) -- (-0.165,-0.071) -- (-0.111,-0.093) -- (-0.056,-0.107) -- (-0.001,-0.112) -- (0.054,-0.108) -- (0.108,-0.094) -- (0.163,-0.072) -- (0.218,-0.040) -- (0.273,0.001) -- (0.327,0.051) -- (0.382,0.110) -- (0.437,0.178);
\end{scope}
\begin{scope}[shift={(8.200,0)}]
\draw[knot] (0.224,-0.388) -- (0.111,-0.193);
\draw[knot] (0.044,-0.076) -- (0.035,-0.060);
\draw[knot] (-0.035,0.060) -- (-0.224,0.388);
\draw[knot] (-0.224,-0.388) -- (-0.111,-0.193);
\draw[knot] (-0.044,-0.076) -- (0.224,0.388);
\draw[knot] (-0.439,-0.134) -- (0.439,-0.134);
\end{scope}
\begin{scope}[shift={(9.580,0)}]
\draw[knot] (0.224,-0.388) -- (0.044,-0.076);
\draw[knot] (-0.044,0.076) -- (-0.224,0.388);
\draw[knot] (-0.224,-0.388) -- (-0.035,-0.060);
\draw[knot] (0.035,0.060) -- (0.224,0.388);
\draw[knot] (-0.439,0.000) -- (0.439,0.000);
\fill[wallRthree] (0.000,0.000) circle (1.5pt);
\end{scope}
\begin{scope}[shift={(10.960,0)}]
\draw[knot] (0.224,-0.388) -- (0.035,-0.060);
\draw[knot] (-0.035,0.060) -- (-0.044,0.076);
\draw[knot] (-0.111,0.193) -- (-0.224,0.388);
\draw[knot] (-0.224,-0.388) -- (0.044,0.076);
\draw[knot] (0.111,0.193) -- (0.224,0.388);
\draw[knot] (-0.439,0.134) -- (0.439,0.134);
\end{scope}

\foreach \k/\name/\what/\col in {0/R_1/cusp/wallRone, 1/R_2/self-tangency/wallRtwo, 2/R_3/triple point/wallRthree}{
  \begin{scope}[shift={(\k*4.1,0)}]
    \node[font=\scriptsize] at (1.38,0.95) {$\name$: \what};
    \draw[\col, line width=1.2pt] (0.3,0.74) -- (2.46,0.74);
    \foreach \j/\tt in {0/{t<0}, 1/{t=0}, 2/{t>0}}{
      \node[font=\tiny, text=figmid] at (\j*1.38,-0.7) {$\tt$};
    }
  \end{scope}
}
\begin{scope}[yshift=-5.0cm, xshift=-0.35cm]
\node[panellabel] at (-0.27,3.75) {(b)};
\fill[figpale, rounded corners=8pt] (0,0) rectangle (11.9,3.3);
\draw[figsoft, rounded corners=8pt] (0,0) rectangle (11.9,3.3);
\node[font=\tiny, text=figmid, anchor=south east] at (11.9,3.32)
  {a patch of $\calX^{<}_{\Lambda,u}(K)$, schematic};
\begin{scope}
\clip[rounded corners=8pt] (0,0) rectangle (11.9,3.3);
\draw[Rone, name path global=wa] (3.05,-0.1) .. controls (2.7,1.1) and (3.45,2.2) .. (3.1,3.4);
\draw[Rtwo, name path global=wb] (6.35,-0.1) .. controls (6.7,1.2) and (5.9,2.1) .. (6.25,3.4);
\draw[Rthree, name path global=wc] (9.35,-0.1) .. controls (9.0,1.0) and (9.7,2.3) .. (9.3,3.4);
\draw[Rone, name path global=wd] (4.4,3.4) .. controls (5.3,2.2) and (7.0,2.25) .. (8.1,3.4);
\end{scope}
\path[name intersections={of=wb and wd, by=cc}];
\draw[figink, line width=0.5pt, fill=white] (cc) circle (2.1pt);
\node[font=\tiny, anchor=north east, fill=figpale, inner sep=1pt] at ([xshift=-3pt,yshift=-3pt]cc) {$\Delta^{\ge2}_{\Lambda,u}(K)$};
\node at (1.2,2.55) {$D_0$};
\node at (4.3,1.85) {$D_1$};
\node at (7.75,1.95) {$D_2$};
\node at (10.8,2.45) {$D_3$};
\node[font=\tiny] at (5.25,3.1) {$D_1'$};
\node[font=\tiny] at (7.15,3.1) {$D_2'$};
\draw[figink, line width=1.05pt, name path=gp, ->]
  (0.35,0.75) .. controls (2.2,1.95) and (3.9,0.35) .. (5.6,0.95)
  .. controls (7.4,1.6) and (8.4,0.35) .. (10.2,0.9)
  .. controls (10.8,1.1) and (11.1,1.5) .. (11.55,1.7);
\path[name intersections={of=gp and wa, by=xa}];
\path[name intersections={of=gp and wb, by=xb}];
\path[name intersections={of=gp and wc, by=xc}];
\fill[wallRone] (xa) circle (1.8pt);
\fill[wallRtwo] (xb) circle (1.8pt);
\fill[wallRthree] (xc) circle (1.8pt);
\node[font=\tiny, anchor=south west] at (0.3,0.8) {$\gamma_t$};
\begin{scope}[yshift=-0.85cm]
\node[font=\tiny, text=figmid, anchor=west] at (-0.25,0.28) {in $\calH_{\Lambda,u}(K)$:};
\node[gvert, label={[font=\scriptsize]below:$D_0$}] (h0) at (1.5,0) {};
\node[gvert, label={[font=\scriptsize]below:$D_1$}] (h1) at (4.6,0) {};
\node[gvert, label={[font=\scriptsize]below:$D_2$}] (h2) at (7.8,0) {};
\node[gvert, label={[font=\scriptsize]below:$D_3$}] (h3) at (10.8,0) {};
\draw[Rone] (h0) -- node[font=\tiny, above, text=figink] {$R_1$} (h1);
\draw[Rtwo] (h1) -- node[font=\tiny, above, text=figink] {$R_2$} (h2);
\draw[Rthree] (h2) -- node[font=\tiny, above, text=figink] {$R_3$} (h3);
\end{scope}
\foreach \p in {xa,xb,xc}{
  \draw[figmid, densely dotted, line width=0.4pt, shorten <=2.5pt] (\p) -- (\p |- 0,-0.62);
}
\end{scope}
\end{tikzpicture}
\caption{Projection--Cerf bookkeeping.  (a)~The three codimension-one events
in a one-parameter family of projections, as crossed by a generic path at
$t=0$: a cusp creates or cancels a curl ($R_1$), a self-tangency creates or
cancels two crossings ($R_2$), and a transverse triple point interchanges three
branches ($R_3$).  Crossing information is indicated by gaps.  (b)~On a strict
sublevel the corresponding walls separate chambers on which the regular
diagram is constant; walls meet in the higher discriminant, which a generic
path avoids.  A projection-generic admissible path crosses the walls one at a
time and transversely, and each crossing contributes one Reidemeister edge of
the induced path in the diagram-image filtration $\calH_{\Lambda,u}(K)$
(bottom row).  The same line styles are used for $R_1$, $R_2$, $R_3$
throughout the paper.}
\label{fig:cerf-walls}
\end{figure}

\begin{remark}
The phrase ``cusp-type'' is used deliberately.  In the classical diagrammatic formulation, the Reidemeister \(R1\) move is the creation or cancellation of a curl.  In a one-parameter family of projected curves, the critical event may be represented locally by a cusp-type singularity or, equivalently, by crossing the codimension-one wall separating diagrams with and without the small loop.  The precise local normal form depends on the chosen regularity category, but the corresponding diagrammatic event is \(R1\).
\end{remark}

\begin{definition}[Higher discriminant]
The higher discriminant is the subset
\[
  \Delta_{\Lambda,u}^{\ge2}(K)
  =
  \Delta_{\Lambda,u}(K)\setminus
  \bigl(
  \Delta_{\Lambda,u}^{R1}(K)\cup
  \Delta_{\Lambda,u}^{R2}(K)\cup
  \Delta_{\Lambda,u}^{R3}(K)
  \bigr).
\]
It consists of non-generic events such as simultaneous Reidemeister events, higher-order tangencies, quadruple points, or tangencies occurring at a crossing.
\end{definition}

For a one-parameter generic path, the higher discriminant is avoided.  Thus the only wall crossings that remain are \(R1\), \(R2\), and \(R3\).

\begin{lemma}[Codimension of the standard projection events]
\label{lem:projection-codim}
For a smooth map $f:S^1\to\mathbb R^2$, the transverse-double-point stratum
has codimension $0$ in map space, whereas the nondegenerate cusp,
self-tangency, and transverse triple-point strata have codimension $1$.
More degenerate stationary points, simultaneous codimension-one events, and
quadruple-point phenomena lie in higher codimension.
\end{lemma}

\begin{proof}
The dimension count is most naturally made in the appropriate multijet
spaces; see \cite[Chapter~II]{GolubitskyGuillemin}.  A transverse double
point imposes the two scalar equations $f(s)=f(t)$ with two source
parameters $(s,t)$, hence codimension $2-2=0$.  A stationary point
$f'(s)=0$ imposes two equations with one source parameter, hence
codimension $2-1=1$; the usual nondegeneracy conditions selecting an
ordinary cusp are open conditions.  A self-tangency imposes
$f(s)=f(t)$ (two equations) and parallel tangent directions (one equation),
with two source parameters, hence codimension $(2+1)-2=1$.  A triple point
imposes $f(r)=f(s)=f(t)$, equivalently four scalar equations, with three
source parameters, hence codimension $4-3=1$.  Multijet transversality makes
these counts rigorous by requiring the corresponding jet conditions to be
transverse; every additional independent degeneracy raises the codimension.
\end{proof}

\begin{remark}[Transversality input]
The relative multijet transversality theorem used below is a standard smooth
transversality result; we use it as an external input in the form that permits
a perturbation to be fixed on endpoint collars.  See
\cite[Chapter~II]{GolubitskyGuillemin}.  The knot-specific point proved here
is that the perturbation can be made while retaining the strict ropelength
budget, by the smoothing and reach estimates of \cref{sec:growth-graphs}.
\end{remark}

\subsection{Local wall-crossing models}

The following observation is the local mechanism behind the graph-growth construction.

\begin{observation}[Local wall-crossing model]
Let \(\gamma_t\), \(t\in(-\eps,\eps)\), be a sufficiently regular path in \(\calX_{\Lambda,u}(K)\).  Suppose that \(\gamma_0\) crosses one Reidemeister wall transversely and avoids the higher discriminant.  Then for sufficiently small \(t>0\), the two regular diagrams
\[
  \Pi_u(\gamma_{-t})
  \quad\text{and}\quad
  \Pi_u(\gamma_t)
\]
differ by exactly one Reidemeister move of the type of the wall crossed.
\end{observation}

\begin{proof}
This is the standard local classification of codimension-one singularities in a generic one-parameter family of plane projections of an embedded curve; compare the classical treatments of plane-curve perestroikas and fronts \cite{ArnoldPlaneCurves,PolyakCurvesFronts}.  A cusp-type event gives the creation or cancellation of a curl, a self-tangency gives the creation or cancellation of two crossings, and a triple point gives the interchange of three local branches.  These are exactly the three Reidemeister moves.
\end{proof}

\begin{definition}[Projection-generic path]
An admissible path \(\gamma_t\in\calX_{\Lambda,u}(K)\) is projection-generic with respect to \(u\) if:
\begin{enumerate}[label=(\roman*)]
  \item \(\gamma_t\) is contained in the regular projection locus except for finitely many parameters;
  \item at each exceptional parameter, \(\gamma_t\) crosses exactly one Reidemeister wall transversely;
  \item the path is disjoint from the higher discriminant.
\end{enumerate}
\end{definition}

\begin{observation}
A projection-generic admissible path determines a finite sequence of regular
diagrams separated by Reidemeister moves.  On a strict ropelength sublevel,
the required general-position perturbation is available because compact
paths have uniform slack; this is proved in
\cref{prop:parametric-regularization}.  At a closed boundary level one allows
an arbitrarily small increase of the level, as in
\cref{cor:epsilon-genericity}.  The finite-dimensional and polygonal
statements below remain useful concrete bookkeeping models.
\end{observation}

\begin{definition}[Liftable Reidemeister transition]
A Reidemeister transition between projected diagrams is \(\Lambda\)-liftable if it is realized as the projection of a projection-generic admissible path segment in \(\calX_{\Lambda,u}(K)\).  Equivalently, it is a transverse crossing of one of the Reidemeister walls inside the ropelength sublevel space.
\end{definition}

\subsection{Cerf graphic of projected thick knots}

The discriminant can be organized into a Cerf-type picture by recording the ropelength level at which a wall is met.

\begin{definition}[Ropelength--projection Cerf graphic]
Fix a projection direction \(u\).  The ropelength--projection Cerf graphic is the subset
\[
  \mathcal C_u(K)
  =
  \{(\Lambda,x)\mid x\in\Delta_{\Lambda,u}(K),\ \Rop(x)=\Lambda\}
\]
viewed over the ropelength axis.  In any controlled stratification for which the standard codimension-one Reidemeister strata are defined, their images are called the \(R1\)-, \(R2\)-, and \(R3\)-branches of the graphic.
\end{definition}

This graphic records the boundary levels at which projected singular events
are met as the ropelength bound is relaxed.  The associated diagram and edge
birth scales are defined intrinsically from the monotone image filtration in
\cref{sec:growth-invariants}.  Vertices are born when new diagram types become
realizable by thick representatives of bounded ropelength, and edges are born
when the corresponding Reidemeister wall crossing has a bounded-ropelength
spatial realization.

\subsection{Three types of graph-growth events}

The wall-crossing viewpoint separates three kinds of changes in the filtered graph.

\begin{definition}[Graph-growth event]
A graph-growth event at ropelength level \(\Lambda_c\) is one of the following:
\begin{enumerate}[label=(\roman*)]
  \item a \emph{vertex birth}, where a new projected diagram first appears at level \(\Lambda_c\);
  \item an \emph{edge birth}, where a new liftable Reidemeister transition first appears at level \(\Lambda_c\);
  \item a \emph{component merge}, where two previously distinct accessible ranges become connected after adding vertices and edges visible at level \(\Lambda_c\).
\end{enumerate}
\end{definition}

\begin{remark}
An edge birth is controlled directly by a transverse crossing of a Reidemeister wall.  A vertex birth is a boundary phenomenon for the ropelength sublevel set: a new regular-projection region touches the boundary \(\Rop=\Lambda_c\).  A component merge is the \(H_0\)-level consequence of vertex and edge births.  Thus the diagrammatic merge tree records only a coarse shadow of the richer wall-crossing growth process.
\end{remark}


\section{Lifted graphs and their diagram-image filtration}
\label{sec:growth-graphs}

We now separate two objects that serve different purposes.  The lifted graph
retains projection-fiber components and is suited to component persistence.
Its diagram image is a monotone submultigraph of $G_S(K)$ and is suited to
finite recognition.

\begin{definition}[Projection fiber]
For a regular diagram $D$, define
\[
  F_{\Lambda,u}(D)
  =
  \{x\in\calX_{\Lambda,u}(K)\mid \Pi_u(x)=D\}.
\]
Its path components are called fiber components over $D$.
Here diagrams are taken up to planar isotopy, with their crossing data;
the regular-projection locus carrying a fixed such diagram is $C^1$-open.
\end{definition}

\begin{definition}[Lifted ropelength-filtered Reidemeister multigraph]
The multigraph $\calG_{\Lambda,u}^{\operatorname{lift}}(K)$ has vertices
$(D,F)$, where $F$ is a path component of $F_{\Lambda,u}(D)$.  For each
classical Reidemeister edge $e:D\to D'$, there is one lifted edge
$(e,F,F')$ precisely when some $\Lambda$-admissible projection-generic movie
realizes $e$ from a point of $F$ to a point of $F'$.  Different classical
moves with the same endpoint diagrams remain distinct, but different movies
realizing the same triple $(e,F,F')$ do not create additional parallel
edges.
\end{definition}

There is a natural multigraph morphism
\[
  p_{\Lambda,u}:\calG_{\Lambda,u}^{\operatorname{lift}}(K)
  \longrightarrow G_S(K),
\]
sending $(D,F)$ to $D$ and a lifted movie to its classical Reidemeister
edge.

\begin{definition}[Diagram-image filtration]
\label{def:diagram-image}
The ropelength-filtered Reidemeister multigraph is the image
\[
  \calH_{\Lambda,u}(K)
  =\operatorname{im}(p_{\Lambda,u})\subset G_S(K).
\]
Thus a diagram is a vertex of $\calH_{\Lambda,u}(K)$ when it has a
representative of ropelength at most $\Lambda$, and a classical
Reidemeister edge belongs to $\calH_{\Lambda,u}(K)$ when that specific move
has some liftable spatial realization of ropelength at most $\Lambda$.
Each vertex is additionally decorated by its ambient valence
$\nu(D)=\val_{G_S(K)}(D)$.
\end{definition}

For $\Lambda\le\Lambda'$, every movie available at level $\Lambda$ remains
available at level $\Lambda'$.  Hence
\[
  \calH_{\Lambda,u}(K)\subseteq\calH_{\Lambda',u}(K)
\]
as decorated submultigraphs.  In contrast, the maps on lifted fiber
components need not be injective because components can merge.  This is the
reason finite recognition is defined using $\calH$, while merge persistence
is defined using $\calG^{\operatorname{lift}}$.

\begin{figure}[t]
\centering
\begin{tikzpicture}[
  >={Latex[length=4pt]},
  font=\scriptsize,
  lv/.style={circle, fill=figink, inner sep=0pt, minimum size=3.8pt},
  le/.style={figink, line width=0.75pt},
  vl/.style={font=\tiny, inner sep=1pt, fill=figpale, fill opacity=0.85, text opacity=1},
  dl/.style={font=\tiny, text=figmid},
  hv/.style={circle, fill=figink, inner sep=0pt, minimum size=3.4pt}
]
\def\horseshoe{(0.2,0.2) -- (5.4,0.2) -- (5.4,3.8) -- (3.8,3.8) -- (3.8,1.0) -- (1.8,1.0) -- (1.8,3.8) -- (0.2,3.8) -- cycle}
\def\bigger{(0.2,0.2) -- (5.4,0.2) -- (5.4,3.8) -- (3.8,3.8) -- (3.8,2.55) -- (1.8,2.55) -- (1.8,3.8) -- (0.2,3.8) -- cycle}
\def\walls{%
  \draw[Rtwo] (0,1.3) -- (5.6,1.3);
  \draw[Rthree] (0,2.8) -- (5.6,2.8);
  \draw[Rone] (3.25,1.95) circle (0.3);}
\begin{scope}
\node[font=\scriptsize] at (2.8,4.15) {level $\Lambda$: $\calX_{\Lambda,u}(K)$ and $\calG^{\operatorname{lift}}_{\Lambda,u}(K)$};
\fill[figpale, rounded corners=6pt] \horseshoe;
\begin{scope}
  \clip[rounded corners=6pt] \horseshoe;
  \walls
\end{scope}
\draw[figsoft, rounded corners=6pt, line width=0.5pt] \horseshoe;
\node[dl] at (0.55,0.45) {$D_1$};
\node[dl] at (0.5,1.55) {$D_2$};
\node[dl] at (5.1,1.55) {$D_2$};
\node[dl] at (0.5,3.05) {$D_3$};
\node[dl] at (5.1,3.05) {$D_3$};
\coordinate (a1) at (2.8,0.64);
\coordinate (a2) at (1.05,2.05);  \coordinate (b2) at (4.55,2.05);
\coordinate (a3) at (1.05,3.4);   \coordinate (b3) at (4.55,3.4);
\draw[le] (a1) .. controls (1.7,0.64) and (1.05,0.85) .. (a2);
\draw[le] (a1) .. controls (3.9,0.64) and (4.55,0.85) .. (b2);
\draw[le] (a2) -- (a3);
\draw[le] (b2) -- (b3);
\fill[wallRtwo] (1.179,1.3) circle (1.6pt);
\fill[wallRtwo] (4.421,1.3) circle (1.6pt);
\fill[wallRthree] (1.05,2.8) circle (1.6pt);
\fill[wallRthree] (4.55,2.8) circle (1.6pt);
\foreach \p in {a1,a2,b2,a3,b3}{\node[lv] at (\p) {};}
\node[vl, anchor=north] at ([yshift=-1.5pt]a1) {$(D_1,F_1)$};
\node[vl, anchor=west] at ([xshift=2pt]a2) {$(D_2,F_2)$};
\node[vl, anchor=east] at ([xshift=-2pt]b2) {$(D_2,F_2')$};
\node[vl, anchor=west] at ([xshift=2pt]a3) {$(D_3,F_3)$};
\node[vl, anchor=east] at ([xshift=-2pt]b3) {$(D_3,F_3')$};
\end{scope}
\begin{scope}[xshift=6.5cm]
\node[font=\scriptsize] at (2.8,4.15) {level $\Lambda'>\Lambda$: $\calX_{\Lambda',u}(K)$ and $\calG^{\operatorname{lift}}_{\Lambda',u}(K)$};
\fill[figpale, rounded corners=6pt] \bigger;
\begin{scope}
  \clip[rounded corners=6pt] \bigger;
  \walls
\end{scope}
\draw[figsoft, rounded corners=6pt, line width=0.5pt] \bigger;
\draw[figmid, densely dotted, line width=0.45pt] (1.8,2.55) -- (1.8,1.0) -- (3.8,1.0) -- (3.8,2.55);
\node[dl] at (0.55,0.45) {$D_1$};
\node[dl] at (0.5,1.55) {$D_2$};
\node[dl] at (5.1,1.55) {$D_2$};
\node[dl] at (0.5,3.05) {$D_3$};
\node[dl] at (5.1,3.05) {$D_3$};
\coordinate (c1) at (2.8,0.64);
\coordinate (c2) at (2.35,1.85);
\coordinate (c3) at (1.05,3.4);  \coordinate (d3) at (4.55,3.4);
\coordinate (c4) at (3.25,1.95);
\draw[le] (c1) -- (c2);
\draw[le] (c2) .. controls (1.4,1.85) and (1.05,2.1) .. (c3);
\draw[le] (c2) .. controls (3.2,1.35) and (4.55,1.35) .. (d3);
\draw[le] (c2) -- (c4);
\fill[wallRtwo] (2.555,1.3) circle (1.6pt);
\fill[wallRthree] (1.085,2.8) circle (1.6pt);
\fill[wallRthree] (4.504,2.8) circle (1.6pt);
\fill[wallRone] (2.952,1.917) circle (1.6pt);
\foreach \p in {c1,c2,c3,d3,c4}{\node[lv] at (\p) {};}
\node[vl, anchor=north] at ([yshift=-1.5pt]c1) {$(D_1,F_1)$};
\node[vl, anchor=north east] at ([xshift=-1pt,yshift=-2pt]c2) {$(D_2,\overline{F}_2)$};
\node[vl, anchor=west] at ([xshift=2pt]c3) {$(D_3,F_3)$};
\node[vl, anchor=east] at ([xshift=-2pt]d3) {$(D_3,F_3')$};
\node[vl, anchor=south, fill=none] at ([yshift=9pt]c4) {$(D_4,F_4)$};
\end{scope}
\draw[->, figmid, line width=0.6pt] (5.5,2.05) -- node[font=\tiny, above, align=center, text=figink] {merge,\\ not injective} (6.95,2.05);
\begin{scope}[yshift=-1.05cm]
\node[hv, label={[font=\tiny]below:$D_1$}] (h1) at (1.5,0) {};
\node[hv, label={[font=\tiny]below:$D_2$}] (h2) at (2.8,0) {};
\node[hv, label={[font=\tiny]below:$D_3$}] (h3) at (4.1,0) {};
\draw[Rtwo] (h1) -- (h2);
\draw[Rthree] (h2) -- (h3);
\node[font=\scriptsize] at (2.8,-0.62) {$\calH_{\Lambda,u}(K)$};
\node at (5.75,-0.1) {$\subseteq$};
\node[hv, label={[font=\tiny]below:$D_1$}] (k1) at (8.0,0) {};
\node[hv, label={[font=\tiny]below:$D_2$}] (k2) at (9.3,0) {};
\node[hv, label={[font=\tiny]below:$D_3$}] (k3) at (10.6,0) {};
\node[hv, label={[font=\tiny]right:$D_4$}] (k4) at (10.0,0.5) {};
\draw[Rtwo] (k1) -- (k2);
\draw[Rthree] (k2) -- (k3);
\draw[Rone] (k2) -- (k4);
\node[font=\scriptsize] at (9.3,-0.62) {$\calH_{\Lambda',u}(K)\subset G_S(K)$};
\end{scope}
\draw[->, figmid, line width=0.5pt] (2.8,0.12) -- node[font=\tiny, right, text=figink] {$p_{\Lambda,u}$} (2.8,-0.72);
\draw[->, figmid, line width=0.5pt] (9.3,0.12) -- node[font=\tiny, left, text=figink] {$p_{\Lambda',u}$} (9.3,-0.72);
\end{tikzpicture}
\caption{Why component information and recognition live in different
graphs (schematic).  Shaded: a sublevel at level $\Lambda$ and a larger one at
$\Lambda'>\Lambda$ (the dotted line marks the former), divided by Reidemeister
walls into regions of constant diagram $D_1,\dots,D_4$.  Dots and arcs form
the lifted multigraphs: a vertex is a pair $(D,F)$ with $F$ a path component
of the fiber $F_{\Lambda,u}(D)$, and an edge is realized by a movie crossing
one wall.  At level $\Lambda$ the fiber over $D_2$ has two components, which
merge at level $\Lambda'$; so the lifted multigraphs do not form an increasing
family, and their two lifted edges over the move $D_1\to D_2$ become one.
Their diagram images do increase, $\calH_{\Lambda,u}(K)\subseteq
\calH_{\Lambda',u}(K)$ (bottom), and here acquire the diagram $D_4$ and an $R_1$
edge at level $\Lambda'$.}
\label{fig:lifted-vs-image}
\end{figure}

\begin{remark}[Finite witnesses, not necessarily a finite level graph]
The adjective ``finite'' refers to the recognizing certificates.  No
finiteness of $\calH_{\Lambda,u}(K)$ is asserted: a ropelength bound alone
does not supply a proved uniform bound on the complexity of one fixed-direction
projection, and compactness of a geometric sublevel would not by itself imply
that only finitely many diagram types occur.  All counting statistics below are therefore taken on specified
finite windows.
\end{remark}

\begin{definition}[Strict lifted graph]
For $\Lambda>\Rop(K)$, define
$\calG^{<,\operatorname{lift}}_{\Lambda,u}(K)$ exactly as
$\calG^{\operatorname{lift}}_{\Lambda,u}(K)$, but replace the closed
sublevel $\calX_{\Lambda,u}(K)$ by the strict sublevel
$\calX^{<}_{\Lambda,u}(K)$, use strict admissible components of the fibers,
and require every lifted movie to be strictly $\Lambda$-admissible.
\end{definition}

\begin{definition}[Unfiltered lifted graph]
\label{def:unfiltered-lifted}
For a regular diagram $D$, put
\[
  F_{\infty,u}(D)
  =\{x\in\calX_{\infty,u}(K)\mid \Pi_u(x)=D\}.
\]
The unfiltered lifted Reidemeister multigraph
$\calG^{\operatorname{lift}}_{\infty,u}(K)$ has vertices $(D,F)$, where
$F$ is an admissible component of $F_{\infty,u}(D)$ (that is, an equivalence
class under admissible isotopies contained in that fiber), and has a lifted edge
$(e,F,F')$ precisely when an admissible projection-generic movie, with no
ropelength bound imposed, realizes the classical Reidemeister edge
$e:D\to D'$ from a point of $F$ to a point of $F'$.
\end{definition}

\begin{remark}
The unfiltered lifted graph is defined directly, rather than as the union of
the finite-level lifted graphs.  The latter do not form an increasing family
of subgraphs, because fiber components may merge as $\Lambda$ increases.
There are nevertheless natural merging maps on their component sets.
\end{remark}

\begin{lemma}[Uniform thickness on a compact admissible family]
\label{lem:uniform-thickness-family}
Let $\Gamma:S^1\times[0,1]\to\mathbb R^3$ be an admissible isotopy whose
time slices have uniformly Lipschitz tangent.  Then
\[
  \inf_{t\in[0,1]}\Thi(\Gamma_t)>0.
\]
\end{lemma}

\begin{proof}
Use the thickness formula as the minimum of the curvature radius and half the
doubly-critical self-distance
\cite{LitherlandSimonDurumericRawdon,CantarellaKusnerSullivan}.  The uniform
Lipschitz bound on the tangent gives a positive uniform curvature radius.  If
the doubly-critical self-distance had a sequence tending to zero, compactness
would give parameter pairs and times converging to a limiting pair.  A
limiting pair with distinct parameters would be a self-intersection of the
limiting time slice.  A pair converging to the diagonal is excluded by the
uniform local graph estimate supplied by the tangent-Lipschitz bound.  Both
alternatives contradict embeddedness, so the doubly-critical self-distance
also has a positive uniform lower bound.
\end{proof}

\begin{theorem}[$C^1$ path approximation with uniform slack]
\label{thm:uniform-slack-approx}
Let $x:[0,1]\to\mathcal E(K)$ be continuous in $\dOne$ and suppose
$M=\sup_t\Rop(x_t)<\Lambda$.  For every $\epsilon>0$ and
$M'\in(M,\Lambda)$ there is a finite piecewise-affine path
$a:[0,1]\to W^{2,\infty}(S^1,\mathbb R^3)$ of regular embeddings of type
$K$, whose endpoint images are $x_0,x_1$, such that
\[
 \sup_t\dOne([a_t],x_t)<\epsilon,
 \qquad \sup_t\Rop(a_t)<M'.
\]
The path $a$ is strongly continuous and piecewise jointly $C^{1,1}$,
with spatial derivative bounded away from zero.  Its image path is also
$\dStar$-continuous.  The same statement holds for the normalized framed
metrics of \cref{def:comparison-topologies}, with endpoints interpreted
modulo the stated symmetries.
Only $C^1$ approximation is asserted; neither approximation in $\dStar$
nor a homotopy from the original quotient path is asserted.
\end{theorem}

\begin{proof}
The quantitative constant-speed interpolation estimate and the complete
compactness argument are in \cref{app:uniform-slack}.
\end{proof}

\begin{proposition}[Strong bridges for reparametrizations]
\label{prop:reparam-bridge}
If $\gamma\in\mathcal E^{\mathrm{par}}(K)$ has $\Rop(\gamma)<\Lambda$
and $\phi$ is an orientation-preserving $C^{1,1}$ source diffeomorphism,
then $\gamma$ and $\gamma\circ\phi$ can be joined inside the same strict
sublevel by a finite piecewise-affine, strongly continuous path.  The
unparametrized images of this path can be kept arbitrarily $\dOne$-close
to $[\gamma]$.
\end{proposition}

\begin{proof}
The proof in \cref{app:reparam-bridge} interpolates finitely many samples
of a family of changes of parameter.  That family need not itself be
strongly continuous.
\end{proof}

\begin{corollary}[Comparison of strict path relations]
\label{cor:strict-path-comparison}
For two oriented unparametrized curves of ropelength less than $\Lambda$,
the following endpoint relations are equivalent:
\begin{enumerate}[label=\textup{(\roman*)}]
\item some (equivalently, any) orientation-compatible parameterized
representatives are joined by a strong $W^{2,\infty}$ path with
$\Rop<\Lambda$;
\item they are joined by a continuous path in the strong quotient
strict sublevel;
\item they are joined by a $\dStar$-continuous path in the strict sublevel;
\item they are joined by a $\dOne$-continuous path with
$\sup_t\Rop<\Lambda$;
\item they are joined by a strictly $\Lambda$-admissible isotopy.
\end{enumerate}
These relations also agree with existence of a finite piecewise-affine
strong path in the strict sublevel.  The same equivalence holds after
allowing the endpoint similarities in $\Sim_u^+$, and in each strict
regular-projection fiber.  In the unoriented theory endpoint reversal is
also allowed; no assertion that every oriented representative joins its
reverse is made.
In particular, naturally in $\Lambda$,
\[
 \pi_0^{\mathrm{adm}}(\calX^<_{\Lambda,u}(K))
 \cong\pi_0(\calX^<_{\Lambda,u}(K))
 \cong\pi_0(\calX^{\star,<}_{\Lambda,u}(K)),
\]
where $\pi_0$ denotes path components.  Arbitrary pointwise-strict paths
in the $C^1$ model are not included unless they have a uniform strict margin.
The strong quotient strict sublevel is locally path connected, so its path
components also equal its connected components.  This last assertion is
not an assertion about the closed $C^1$ sublevels used in the merge theory.
\end{corollary}

\begin{proof}
See \cref{app:strict-relations}.  The argument compares endpoint
connectivity and does not identify the underlying topologies or path spaces.
\end{proof}

\begin{lemma}[Piecewise jointly $C^{1,1}$ approximation]
\label{lem:strong-path-approx}
Let $B=W^{2,\infty}(S^1,\mathbb R^3)$ with its strong norm and let
$U_\Lambda(K)\subset B$ consist of regular embeddings of type $K$ with
$\Rop<\Lambda$.  If $\gamma:[0,1]\to U_\Lambda(K)$ is strongly continuous,
then, for every $\epsilon>0$, there is a strongly continuous path
$\widetilde\gamma$ with the same endpoints such that
\[
 \sup_t\|\widetilde\gamma_t-\gamma_t\|_{2,\infty}<\epsilon,
 \qquad \sup_t\Rop(\widetilde\gamma_t)<\Lambda,
\]
and $(s,t)\mapsto\widetilde\gamma_t(s)$ is jointly $C^{1,1}$ on each
closed interval of a finite time subdivision.  The two paths are homotopic
rel endpoints through strongly continuous paths in $U_\Lambda(K)$.
If the original family is jointly smooth near prescribed closed endpoint
collars $A$, the approximation and homotopy can be fixed on a neighborhood
of $A$.
\end{lemma}

\begin{proof}
See \cref{app:strong-path-approx}.  The construction is polygonal
interpolation in the time variable; it does not mollify the curve variable
or impose a constant-speed parametrization.
\end{proof}

\begin{lemma}[Relative parametric smoothing with uniform reach control]
\label{lem:parametric-smoothing}
Let $\Gamma:S^1\times[0,1]\to\mathbb R^3$ be a continuous family of
positive-thickness $C^{1,1}$ embedded closed curves, made constant on small
endpoint collars.  Assume that $\Gamma$ and $\partial_s\Gamma$ are continuous
on the parameter cylinder and that $|\partial_s\Gamma|$ is bounded away from
zero.  Put
\[
 r_\Gamma=\inf_{t\in[0,1]}\Thi(\Gamma_t)>0.
\]
In particular, the hypotheses hold for an admissible isotopy and for its
continuous time-dependent similarity normalizations.
For every $\epsilon>0$ and every $0<\delta<r_\Gamma$ there is an isotopy
$\Gamma^{\mathrm{sm}}$, fixed on smaller endpoint collars and smooth away
from those collars, such that
\[
 \sup_{t\in[0,1]}
 \bigl|\Len(\Gamma^{\mathrm{sm}}_t)-\Len(\Gamma_t)\bigr|<\epsilon,
 \qquad
 \inf_{t\in[0,1]}\Thi(\Gamma^{\mathrm{sm}}_t)
 \ge r_\Gamma-\delta.
\]
After reparametrizing the time slices, the approximation can be chosen
arbitrarily small in the uniform slice-wise $C^1$ topology.
\end{lemma}

\begin{proof}
A complete curve-specific proof, including the relative endpoint statement
and the uniform control of both the curvature radius and the doubly-critical
self-distance, is given in \cref{app:parametric-smoothing-proof}.
\end{proof}

\begin{corollary}[Relative one-parameter smoothing]
\label{cor:conj82-dim1}
Let $P=[0,1]$ and let $A\subset P$ be a union of closed endpoint collars.
Let $F:P\to W^{2,\infty}(S^1,\mathbb R^3)$ be a strongly continuous
family of regular positive-thickness embeddings of type $K$ with
$\Rop(F_t)\le\Lambda$, jointly smooth on a neighborhood of $S^1\times A$.
Then for every $\varepsilon>0$ there is a strongly continuous homotopy rel
$A$, through embeddings with ropelength less than $\Lambda+\varepsilon$,
from $F$ to a continuous family of smooth embeddings.  The endpoint family
can be chosen continuous also in the $C^\infty$ topology.

In degree zero this has the following precise right-relaxed meaning.
Every strong $C^{1,1}$ representative at level $\Lambda$ can be joined to
a smooth representative below $\Lambda+\varepsilon$; any two smooth
representatives joined by a strong $C^{1,1}$ path at level $\Lambda$ can
be joined by a path of smooth representatives below $\Lambda+\varepsilon$.
This does not assert a bijection on components at a fixed closed level,
or a comparison for all paths in a weaker $C^1$ topology.
\end{corollary}

\begin{proof}
Use the strict budget $\Lambda+\varepsilon/2$ for $F$.
Apply \cref{lem:strong-path-approx} relative to a neighborhood of $A$ to
obtain a piecewise jointly $C^{1,1}$ path $G$ and a strong homotopy from
$F$ to $G$ in this sublevel.  Apply \cref{lem:strong-relative-smoothing}
to $G$ with budget $\Lambda+\varepsilon$ and concatenate the homotopies.
The second lemma uses a fixed convolution width and a linear homotopy;
it does not assume that convolution converges in the strong norm.
For the degree-zero statement, use a constant family for a single
representative.  A strong path with smooth endpoints can first be made
constant on small endpoint collars by a continuous time reparametrization,
without changing its ropelength bound, and then the relative assertion
applies.
\end{proof}

\begin{remark}[Scope of smoothing]\label{rem:conj82-scope}
The preceding result treats strong parameterized paths. Its compact-parameter
extension is \cref{thm:compact-parameter-smoothing}. Neither result supplies
path lifting from a coarse quotient. Letting the convolution width tend to
zero is not a strong $W^{2,\infty}$ homotopy; the proofs keep that width fixed.
\end{remark}

\begin{proposition}[Parametric regularization with ropelength slack]
\label{prop:parametric-regularization}
Let $\Gamma$ be a strictly $\Lambda$-admissible isotopy with regular
projected endpoints.  Then, relative to the endpoints, $\Gamma$ can be
replaced by a projection-generic strictly $\Lambda$-admissible isotopy.
In particular, the replacement has only finitely many singular projection
times, each giving exactly one move of type $R1$, $R2$, or $R3$.
\end{proposition}

\begin{proof}
Set
\[
 M=\sup_{t\in[0,1]}\Rop(\Gamma_t)<\Lambda.
\]
After a harmless reparametrization in time, make the family constant on
small endpoint collars.  Since the length $L(t)=\Len(\Gamma_t)$ is positive
and continuous, the time-dependent homothety
\[
 \widehat\Gamma_t=L(t)^{-1}\Gamma_t
\]
represents the same path in the projection-framed quotient and has
\[
 \Len(\widehat\Gamma_t)=1,
 \qquad
 \Thi(\widehat\Gamma_t)=\Rop(\Gamma_t)^{-1}\ge M^{-1}.
\]
The length-normalized family satisfies the hypotheses of
\cref{lem:parametric-smoothing}: use normalized arclength on each time slice,
as in the proof in \cref{app:parametric-smoothing-proof}.  Choose
$\epsilon>0$ and $0<\delta<M^{-1}$ so small that
\[
 \frac{1+\epsilon}{M^{-1}-\delta}<\Lambda.
\]
By \cref{lem:parametric-smoothing}, there is a relative smoothing of
$\widehat\Gamma$ satisfying
\[
 \Len(\Gamma^{\mathrm{sm}}_t)<1+\epsilon,
 \qquad
 \Thi(\Gamma^{\mathrm{sm}}_t)\ge M^{-1}-\delta.
\]
Consequently
\[
 \Rop(\Gamma^{\mathrm{sm}}_t)
 \le \frac{1+\epsilon}{M^{-1}-\delta}<\Lambda
\]
uniformly in $t$.

Apply relative multijet transversality to the smooth part of
$\Gamma^{\mathrm{sm}}$; see \cite[Chapter~II]{GolubitskyGuillemin}.  The perturbation may be taken arbitrarily small in
$C^2$ and fixed on the endpoint collars.  On a compact smooth family,
curvature radius and doubly-critical self-distance remain bounded below
under sufficiently small $C^2$ perturbations; this follows directly from
the thickness formula by the same compactness argument used in
\cref{app:parametric-smoothing-proof}.  Hence the strict ropelength
inequality is preserved.  By \cref{lem:projection-codim}, the codimension count excludes the
higher projection discriminant and leaves only isolated cusp-type,
self-tangency, and triple-point events.  Their local models are respectively
the three Reidemeister moves; compare \cite{QueffelecReidemeister}.  The
exceptional set is a closed discrete subset of the compact parameter
interval and hence is finite.
\end{proof}

\begin{corollary}[Arbitrarily small relaxation of a closed-level path]
\label{cor:epsilon-genericity}
Let $x$ be a $C^1$-continuous path of unparametrized curves, possibly
modulo the framed similarities, with $\sup_t\Rop(x_t)\le\Lambda$ and
regular projected endpoints.  This includes every $\Lambda$-admissible
isotopy.  For every $\varepsilon>0$, its endpoints can be joined by a
projection-generic strictly $(\Lambda+\varepsilon)$-admissible isotopy.
For a given parameterized path, its specified endpoint parametrizations
can be retained, with the usual orientation convention.
\end{corollary}

\begin{proof}
Apply \cref{thm:uniform-slack-approx} with a budget strictly between
$\Lambda$ and $\Lambda+\varepsilon$.  Use \cref{prop:reparam-bridge}
at the ends, and ambient similarity paths when needed, to match the
prescribed endpoint representatives.  All these paths can be chosen
strictly below $\Lambda+\varepsilon$.  Their finite concatenation is
admissible; \cref{prop:parametric-regularization} then applies.
\end{proof}

\begin{remark}[What is and is not proved at a closed critical level]
The preceding corollary does not assert that every path contained in the
closed set $\Rop\le\Lambda$ can be made generic while remaining in that same
closed set.  Such an assertion can fail for a general constrained functional
at a critical boundary level and is unnecessary here.  The strict-level
reconstruction and the arbitrarily small relaxation are exactly what is
needed for persistence infima and merge scales.
\end{remark}

\begin{figure}[t]
\centering

\caption{Slack.  (a)~A $C^1$ path with $M=\sup_t\Rop(x_t)<\Lambda$ (black) is
replaced, as in \cref{thm:uniform-slack-approx}, by a strongly continuous
piecewise-affine path $a_t$ (grey) with the same endpoint classes, uniformly
$C^1$-close to $x_t$ and with $\sup_t\Rop(a_t)<M'$ for any prescribed
$M'\in(M,\Lambda)$.  (b)~At the closed level, where $\sup_t\Rop(x_t)=\Lambda$
may be attained, right relaxation is used instead: for every $\varepsilon>0$
the endpoints are joined by a projection-generic, strictly
$(\Lambda+\varepsilon)$-admissible isotopy (grey;
\cref{cor:epsilon-genericity}).  Its finitely many singular times (ticks on
the $t$-axis, in the line styles of \cref{fig:cerf-walls}) are single
Reidemeister moves (\cref{prop:parametric-regularization}).  The profiles are
illustrative.}
\label{fig:slack}
\end{figure}

\subsection{Finite-dimensional and polygonal bookkeeping}

\begin{proposition}[Finite-dimensional projection--Cerf bookkeeping]
\label{thm:finite-dimensional-cerf}
Let $B$ be a compact finite-dimensional family in
$\calX_{\Lambda,u}(K)$.  Suppose its projection discriminant has a locally
finite Whitney stratification and that a path in $B$ is transverse to the
standard codimension-one $R1$, $R2$, and $R3$ strata and avoids every other
stratum.  Then its projected diagrams change at finitely many times, by one
Reidemeister move at each such time.
\end{proposition}

\begin{proof}
Transversality and local finiteness give a finite set of intersections on a
compact parameter interval.  The stipulated local model of each crossed
stratum gives the corresponding Reidemeister move.
\end{proof}

\begin{definition}[Polygonal thickness]
For an embedded closed polygon $P$ with fixed cyclic combinatorics, let
\[
  \Thi_{\operatorname{poly}}(P)
  =\min\{\operatorname{MinRad}(P),
  \operatorname{dcsd}_{\operatorname{poly}}(P)/2\}
\]
be Rawdon's polygonal thickness \cite{Rawdon}.
\end{definition}

\begin{lemma}[Semialgebraic polygonal sublevels]
\label{lem:polygonal-thickness-semialgebraic}
For a fixed number of vertices and fixed cyclic combinatorics, the conditions
$\Thi_{\operatorname{poly}}(P)\ge1$, $\Len(P)\le\Lambda$, nonzero edge
lengths, and embeddedness define a semialgebraic set in vertex coordinates.
\end{lemma}

\begin{proof}
There are finitely many vertex and edge-pair conditions.  The minimum-radius
conditions are algebraic after introducing edge-length variables and clearing
denominators on the nonzero-edge locus.  Critical-distance and disjointness
conditions are first-order semialgebraic conditions in finitely many edge
parameters, and quantifier elimination gives a semialgebraic set in the
vertex coordinates.
\end{proof}

\begin{proposition}[Polygonal projection bookkeeping]
\label{thm:polygonal-cerf}
A compact semialgebraic polygonal sublevel has a finite semialgebraic Whitney
stratification of its projection discriminant.  Any semialgebraic path
transverse to the strata representing the standard PL Reidemeister events
and avoiding all remaining strata has finitely many singular times and
undergoes one PL Reidemeister move at each such time.
\end{proposition}

\begin{proof}
Semialgebraic sets admit finite Whitney stratifications.  A transverse
semialgebraic path has finite intersection with the relevant strata on a
compact interval.  The conclusion follows from the assumed standard local
model at each crossed stratum.  Notice that this proposition does not claim,
without an additional general-position argument, that every codimension-one
stratum in an arbitrary constrained polygon family is automatically a
Reidemeister stratum.
\end{proof}

\begin{theorem}[Strict-sublevel component reconstruction]
\label{thm:strict-component-reconstruction}
For every $\Lambda>\Rop(K)$,
\[
  \pi_0\bigl(\calG^{<,\operatorname{lift}}_{\Lambda,u}(K)\bigr)
  \cong
  \pi_0^{\operatorname{adm}}\bigl(\calX^{<}_{\Lambda,u}(K)\bigr),
\]
where the right-hand side denotes strict admissible components.
By \cref{cor:strict-path-comparison}, it can equivalently be written
$\pi_0(\calX^<_{\Lambda,u}(K))$ in the strong quotient topology or
$\pi_0(\calX^{\star,<}_{\Lambda,u}(K))$.
\end{theorem}

\begin{proof}
Every strict admissible component meets the regular projection locus: first
smooth one representative within the component, using
\cref{lem:parametric-smoothing} for a constant family together with the
ropelength slack, and then apply a small generic spatial rotation, joined to
the identity through rotations; for a smooth curve almost every direction
gives a regular projection, and rotations change neither length nor
thickness.  By \cref{prop:parametric-regularization}, every
strict admissible isotopy with regular endpoints can then be replaced,
without leaving the strict sublevel, by a projection-generic one.  Such an isotopy gives a finite sequence of lifted
vertices and lifted Reidemeister edges.  Conversely, every lifted edge is
represented by a strict admissible path segment.  For two consecutive
edges incident to $(D_i,F_i)$, the terminal point of the first movie and
the initial point of the second need not coincide.  They lie in the same
strict admissible fiber component $F_i$, so insert a strict admissible
path inside $F_i$ joining them.  Join prescribed representatives at the
first and last vertices in the same way.  There are only finitely many
segments, so their concatenation is piecewise jointly $C^{1,1}$ and its
ropelength supremum is still strictly less than $\Lambda$.
Hence the two connectedness relations agree.
\end{proof}

\begin{corollary}[Component reconstruction after attaching higher cells]
\label{cor:conj81-pi0}
For every $\Lambda>\Rop(K)$, any filtered lifted Reidemeister complex
$\calR^{\operatorname{lift}}_{\Lambda,u}(K)$ whose one-skeleton is
$\calG^{<,\operatorname{lift}}_{\Lambda,u}(K)$ satisfies
\[
  \pi_0\bigl(\lvert\calR^{\operatorname{lift}}_{\Lambda,u}(K)\rvert\bigr)
  \cong
  \pi_0^{\operatorname{adm}}\bigl(\calX^{<}_{\Lambda,u}(K)\bigr).
\]
This is a component statement at every strict level; it asserts no higher weak equivalence.
\end{corollary}

\begin{proof}
Attaching cells of dimension at least two does not change $\pi_0$, so
$\pi_0(\lvert\calR^{\operatorname{lift}}_{\Lambda,u}(K)\rvert)
=\pi_0(\calG^{<,\operatorname{lift}}_{\Lambda,u}(K))$.  Now apply
\cref{thm:strict-component-reconstruction}.
\end{proof}

\begin{corollary}[Unfiltered component reconstruction]
\label{cor:unfiltered-component-reconstruction}
There is a natural bijection
\[
  \pi_0\bigl(\calG^{\operatorname{lift}}_{\infty,u}(K)\bigr)
  \cong
  \pi_0^{\operatorname{adm}}\bigl(\calX_{\infty,u}(K)\bigr).
\]
Moreover, both sides consist of a single element.
\end{corollary}

\begin{proof}
Every admissible isotopy has finite maximal ropelength.  Indeed, after the
finite subdivision in the definition of admissibility, joint $C^{1,1}$
regularity gives a uniform upper bound for the lengths of the time slices,
while \cref{lem:uniform-thickness-family} gives a positive uniform lower
bound for their thicknesses.  Thus the whole isotopy is strictly
$\Lambda$-admissible for some finite $\Lambda$.  The same observation applies
to every finite concatenation of admissible fiber paths and lifted movies.
Consequently the proof of \cref{thm:strict-component-reconstruction} applies
at a sufficiently large finite level and gives the stated bijection.

It remains to see that the geometric side is connected.  By
\cref{cor:conj82-dim1}, every $C^{1,1}$ representative can be joined, after
allowing some finite ropelength budget, to a smooth representative of the
same knot type.  Any two smooth representatives of $K$ are joined by a
smooth isotopy.  On the compact parameter interval this smooth family has
uniformly bounded length and, by \cref{lem:uniform-thickness-family},
uniformly positive thickness.  Hence it is an admissible isotopy of finite
maximal ropelength.  Therefore
$\pi_0^{\operatorname{adm}}(\calX_{\infty,u}(K))$ is a singleton.
\end{proof}

\begin{proposition}[Compact-parameter recovery at the unfiltered level]
\label{prop:unfiltered-homotopy-colimit}
Fix a basepoint $x_0\in\calX_{\infty,u}(K)$.  For every $n\ge1$, the
inclusions of strict sublevels induce a natural isomorphism
\[
  \varinjlim_{\Lambda>\Rop(x_0)}
  \pi_n\bigl(\calX^{<}_{\Lambda,u}(K),x_0\bigr)
  \xrightarrow{\ \cong\ }
  \pi_n\bigl(\calX_{\infty,u}(K),x_0\bigr).
\]
For $n=0$ the analogous statement holds for path-component sets.
\end{proposition}

\begin{proof}
By \cref{lem:comparison-prelim}, ropelength is continuous in the strong
quotient topology.  Hence, for every continuous map
$f:S^n\to\calX_{\infty,u}(K)$, the compactness of $S^n$ gives
\[
  M_f:=\max_{z\in S^n}\Rop(f(z))<\infty.
\]
Choosing $\Lambda>M_f$ places the whole image of $f$ in
$\calX^{<}_{\Lambda,u}(K)$.  Thus every unfiltered homotopy class is
represented at a finite strict level.

Likewise, if two maps become homotopic in $\calX_{\infty,u}(K)$, a homotopy
$H:S^n\times[0,1]\to\calX_{\infty,u}(K)$ has compact image, so
$\Rop\circ H$ has a finite maximum.  The entire homotopy therefore lies in
one sufficiently large strict sublevel.  This proves injectivity as well as
surjectivity of the displayed map.  The same compactness argument with
points and paths proves the statement for $n=0$.
\end{proof}

\begin{remark}[Higher homotopy and the Reidemeister complex]
\label{rem:unfiltered-higher-pi}
\Cref{prop:unfiltered-homotopy-colimit} is a statement only about the
geometric knot spaces.  It does not extend the Reidemeister comparison of
\cref{cor:conj81-pi0} to positive degrees, and in particular it does not
assert that
\(\pi_n(\calX_{\infty,u}(K))=0\) for $n\ge1$.  Such groups may carry
nontrivial topology of the knot space.

A positive-degree analogue of \cref{cor:conj81-pi0}, that is, a higher reconstruction theorem, remains open; see \cref{sec:compact-parameters}.
To compare $\pi_n$ with a full lifted Reidemeister complex one would need
parameterized smoothing and projection-generic regularization for the
relevant $n$-parameter families and their homotopies, hence through
parameter dimension $n+1$, together with the corresponding higher-cell
relations coming from the projection discriminant.  Slice-wise smoothing
over compact parameter spaces is supplied by
\cref{thm:compact-parameter-smoothing}; the missing inputs are joint
regularity and projection-generic regularization for families of dimension
at least two, which is the relative-genericity hypothesis of
\cref{sec:compact-parameters}.  In degree one, for example, the required
homotopies are already two-parameter.
\end{remark}

\begin{remark}[No finite stabilization is implied]
\label{rem:unfiltered-no-stabilization}
The fact that the unfiltered component set is a singleton does not imply the
existence of a finite $\Lambda_0$ for which
$\pi_0(\calG^{<,\operatorname{lift}}_{\Lambda,u}(K))$ is a singleton for
every $\Lambda\ge\Lambda_0$.  Unfiltered connectivity says only that each
particular finite collection of points and connecting paths is contained in
some finite ropelength level.  It gives no uniform bound over all components
that may appear at larger levels.
\end{remark}

\begin{definition}[Right-relaxed closed-level equivalence]
For regular points $x,y\in\calX_{\Lambda,u}(K)$, write
$x\sim_{\Lambda+}y$ if, for every $\varepsilon>0$, their images lie in the
same strict admissible component of
$\calX^{<}_{\Lambda+\varepsilon,u}(K)$.  By
\cref{thm:strict-component-reconstruction}, this is equivalent to requiring
connectivity of the corresponding lifted vertices in
$\calG^{<,\operatorname{lift}}_{\Lambda+\varepsilon,u}(K)$ for every
$\varepsilon>0$.
\end{definition}

\section{The fibre cosheaf}
\label{sec:cosheaf}

The lifted multigraph, its diagram image, and the reconstruction theorem
\cref{thm:strict-component-reconstruction} have a common source.  Over each
planar diagram one has the set of fibre components, and over each Reidemeister
move the relation given by the movies that realize it.  This is a cosheaf of
sets on a Reidemeister multigraph, in the cellular sense of \cite{Curry}.  The
lifted multigraph is its Grothendieck construction, its zeroth homology is the
set of strict admissible components, and the diagram image is its support.
This section develops the necessary combinatorics in
\cref{subsec:rel-cosheaves,subsec:filtered-cosheaves,subsec:flat-coarsening},
identifies the geometric instance in \cref{subsec:fibre-cosheaf}, and reads
off exactly what a diagram image can see in \cref{subsec:image-merges}.

\subsection{Relational cosheaves on a multigraph}
\label{subsec:rel-cosheaves}

Let $G$ be a multigraph; loops and parallel edges are allowed.  Orient each
edge $e$ arbitrarily, with endpoints $\partial_0e$ and $\partial_1e$.  The
\emph{cell category} $\operatorname{Cell}(G)$ has objects
$\Vtx(G)\sqcup\Edge(G)$ and, besides identities, two morphisms
$e\to\partial_0e$ and $e\to\partial_1e$ for each edge; they are distinct even
when $e$ is a loop.

\begin{definition}[Relational cellular cosheaf]
\label{def:rel-cosheaf}
A \emph{cellular cosheaf of sets} on $G$ is a functor
$P:\operatorname{Cell}(G)\to\mathbf{Set}$.  It assigns a \emph{stalk} $P(D)$
to each vertex, a \emph{costalk} $P(e)$ to each edge, and \emph{extension
maps} $P(e)\to P(\partial_ie)$.  It is \emph{relational} if for every edge
the map $P(e)\to P(\partial_0e)\times P(\partial_1e)$ is injective; we then
identify $P(e)$ with a relation
$R_P(e)\subset P(\partial_0e)\times P(\partial_1e)$.  A morphism of cosheaves
is a natural transformation.  Reversing the orientation of $e$ replaces
$R_P(e)$ by its transpose, and nothing below depends on the orientations.
\end{definition}

\begin{definition}[Zeroth homology, Grothendieck construction, support]
\label{def:groth}
The \emph{zeroth homology} of $P$ is the colimit
\[
  H_0(G;P)=\operatorname*{colim}_{\operatorname{Cell}(G)}P
  =\Bigl(\,\bigsqcup_{D\in\Vtx(G)}P(D)\Bigr)\Big/\!\approx,
\]
where $\approx$ is the equivalence relation generated by $a\approx b$
whenever $a$ and $b$ are the two extensions of one element of a costalk; the
class of $a$ is written $[a]$.  The \emph{Grothendieck construction}
$\int_GP$ is the multigraph with one vertex $(D,a)$ for each $D\in\Vtx(G)$
and $a\in P(D)$, and one edge $(e,c)$ for each $e\in\Edge(G)$ and
$c\in P(e)$, joining the two extensions of $c$.  The projection
$\pi_P:\int_GP\to G$ forgets $a$ and $c$.  The \emph{support}
$\operatorname{supp}P\subset G$ consists of the vertices with
$P(D)\ne\varnothing$ and the edges with $P(e)\ne\varnothing$; it is a
subgraph because an element of a costalk extends to both endpoint stalks.
\end{definition}

\begin{lemma}[Basic properties]
\label{lem:groth-basic}
\leavevmode
\begin{enumerate}[label=\textup{(\roman*)}]
\item There is a natural bijection $\pi_0\bigl(\int_GP\bigr)\cong H_0(G;P)$.
\item $\pi_P$ maps $\int_GP$ onto $\operatorname{supp}P$, surjectively on
vertices and on edges.  Hence it induces a surjection
$\pi_0\bigl(\int_GP\bigr)\to\pi_0(\operatorname{supp}P)$.
\item Replacing every costalk by its image relation gives a relational
cosheaf $P^{\mathrm{rel}}$; $\int_GP^{\mathrm{rel}}$ is obtained from
$\int_GP$ by identifying parallel edges with the same endpoints over the same
edge of $G$, and $H_0(G;P^{\mathrm{rel}})=H_0(G;P)$.
\end{enumerate}
\end{lemma}

\begin{proof}
(i) The vertex set of $\int_GP$ is $\bigsqcup_DP(D)$, and its edges join
exactly the generating pairs of $\approx$.  (ii) Every vertex and every edge
of the support has a preimage, so every component of the support is hit.
(iii) The generating pairs of $\approx$ depend only on the image relations.
\end{proof}

Part (iii) is the reason why the lifted multigraph has one lifted edge for
each triple $(e,F,F')$, however many movies realize it: a span-valued
cosheaf remembering the movies would carry the same zeroth homology.

\begin{proposition}[When the support computes components]
\label{prop:support-criterion}
The following are equivalent.
\begin{enumerate}[label=\textup{(\alph*)}]
\item The surjection $\pi_0\bigl(\int_GP\bigr)\to\pi_0(\operatorname{supp}P)$
is a bijection.
\item For every vertex $D$, the image of $P(D)$ in $H_0(G;P)$ has at most one
element.
\end{enumerate}
\end{proposition}

\begin{proof}
(b)$\Rightarrow$(a).  Let $(D,a)$ and $(D',a')$ map into one component of the
support, and choose a path $D=D_0,e_1,D_1,\dots,e_k,D_k=D'$ in it.  Each
$P(e_i)$ contains some $c_i$, with extensions $b_i\in P(D_{i-1})$ and
$b_i'\in P(D_i)$.  Then $[b_i]=[b'_i]$, and by (b) $[a]=[b_1]$,
$[b'_i]=[b_{i+1}]$ and $[b'_k]=[a']$.  Hence $[a]=[a']$.
(a)$\Rightarrow$(b).  Two elements of $P(D)$ give vertices over the same
vertex $D$, hence over one component of the support, so they have the same
class by (a).
\end{proof}

\begin{figure}[t]
\centering
\begin{tikzpicture}[x=1cm,y=1cm,font=\small]
 \foreach \x/\n/\k in {0/D_1/1,3.2/D_2/2,6.4/D_3/3}{
   \node[gvert] (v\k) at (\x,0) {};
   \node[below=2pt] at (\x,0) {$\n$};
   \draw[figsoft,densely dotted] (\x,0.25)--(\x,1.15);
 }
 \draw[figink,line width=0.7pt] (v1)--node[below]{$e_1$}(v2)--node[below]{$e_2$}(v3);
 \node[anchor=east] at (-0.5,0) {$\operatorname{supp}P$};
 \node[gvert] (a) at (0,1.9) {};   \node[left=2pt] at (a) {$a$};
 \node[gvert] (b) at (3.2,1.5) {}; \node[above left=-1pt] at (b) {$b$};
 \node[gvert] (bb) at (3.2,2.5) {};\node[above left=-1pt] at (bb) {$b'$};
 \node[gvert] (c) at (6.4,1.5) {}; \node[right=2pt] at (c) {$c$};
 \node[gvert] (cc) at (6.4,2.5) {};\node[right=2pt] at (cc) {$c'$};
 \draw[wallRtwo,line width=0.9pt] (a)--(b)--(c);
 \draw[wallRone,line width=0.9pt] (bb)--(cc);
 \node[anchor=east] at (-0.5,2) {$\int_GP$};
 \node[align=left,anchor=west] at (7.6,2.0) {two components};
 \node[align=left,anchor=west] at (7.6,0) {one component};
 \draw[->,figmid] (7.35,1.4)--node[right]{$\pi_P$}(7.35,0.45);
\end{tikzpicture}
\caption{A relational cosheaf on a path with stalks $\{a\}$, $\{b,b'\}$,
$\{c,c'\}$ and costalks $R(e_1)=\{(a,b)\}$, $R(e_2)=\{(b,c),(b',c')\}$.  Its
Grothendieck construction has two components and its support has one.
Condition~(b) of \cref{prop:support-criterion} fails at $D_2$ and $D_3$.  The
flat coarsening of \cref{prop:flat-terminal} has stalks with one, two and two
elements.  The same picture, with fibre components as stalk elements, is
\cref{fig:lifted-vs-image}.}
\label{fig:groth}
\end{figure}

\subsection{Filtrations and merge scales}
\label{subsec:filtered-cosheaves}

A \emph{filtered} relational cosheaf is a family $(P_\Lambda)_{\Lambda\in I}$,
indexed by an interval $I\subset\R$, of relational cosheaves on $G$ with
morphisms $\iota_{\Lambda\Lambda'}:P_\Lambda\to P_{\Lambda'}$ for
$\Lambda\le\Lambda'$, satisfying $\iota_{\Lambda\Lambda}=\id$ and
$\iota_{\Lambda'\Lambda''}\iota_{\Lambda\Lambda'}=\iota_{\Lambda\Lambda''}$.
The Grothendieck constructions then form a family of graph morphisms
$\int_GP_\Lambda\to\int_GP_{\Lambda'}$ over $G$, in general not injective,
whereas the supports form an increasing family of subgraphs.  This is the
structural reason, announced in the introduction, why merge information
lives in the lifted graphs and finite recognition in the diagram image.

For $a\in P_{\Lambda_a}(D)$ and $b\in P_{\Lambda_b}(D')$ put
$\Lambda_{ab}=\max\{\Lambda_a,\Lambda_b\}$ and
\begin{align*}
  m^P(a,b)&=\inf\{\Lambda\ge\Lambda_{ab}:[\iota a]=[\iota b]
  \text{ in }H_0(G;P_\Lambda)\},\\
  m^{\operatorname{supp}}(a,b)&=\inf\{\Lambda\ge\Lambda_{ab}:D,D'
  \text{ lie in one component of }\operatorname{supp}P_\Lambda\}.
\end{align*}
By \cref{lem:groth-basic}(ii), $m^{\operatorname{supp}}\le m^P$.  We say that
$P$ has a \emph{persistent stalk separation} if there are $\Lambda<\Lambda'$
in $I$, a vertex $D$ and $a,a'\in P_\Lambda(D)$ whose images in
$H_0(G;P_{\Lambda'})$ are distinct.

\begin{theorem}[When the support computes merges]
\label{thm:support-merges}
The equality $m^{\operatorname{supp}}(a,b)=m^P(a,b)$ holds for all pairs of
stalk elements if and only if $P$ has no persistent stalk separation.
\end{theorem}

\begin{proof}
If $a,a'\in P_\Lambda(D)$ stay distinct in $H_0(G;P_{\Lambda'})$, then
$m^{\operatorname{supp}}(a,a')=\Lambda<\Lambda'\le m^P(a,a')$.  Conversely,
suppose $m^{\operatorname{supp}}(a,b)<m^P(a,b)$, and choose
$\Lambda<\Lambda'$ in $I$ strictly between them such that $D$ and $D'$ lie in
one component of $\operatorname{supp}P_\Lambda$; this is possible because the
supports increase.  Lift a path in the support edge by edge as in the proof of
\cref{prop:support-criterion}.  Consecutive lifts meet a common stalk of
$P_\Lambda$, and in the absence of persistent stalk separation their images
in $H_0(G;P_{\Lambda'})$ agree.  Hence $[\iota a]=[\iota b]$ in
$H_0(G;P_{\Lambda'})$, contradicting $\Lambda'<m^P(a,b)$.
\end{proof}

The second level $\Lambda'$ is a one-sided slack of the same kind as the
right relaxation of \cref{cor:epsilon-genericity}.

\subsection{The flat coarsening}
\label{subsec:flat-coarsening}

A \emph{coarsening} of a relational cosheaf $P$ is a relational cosheaf $Q$
with a morphism $r:P\to Q$ that is surjective on every stalk and costalk.  The
\emph{flat coarsening} $P^\flat$ has as stalk $P^\flat(D)$ the image of $P(D)$
in $H_0(G;P)$ and as costalk the image relation of $P(e)$; let
$\kappa:P\to P^\flat$ be the quotient.  The \emph{trivial coarsening}
$\mathbf 1_{\operatorname{supp}P}$ has one-point stalks and costalks on the
support and empty ones elsewhere; its Grothendieck construction is
$\operatorname{supp}P$.

\begin{proposition}[The flat coarsening is terminal]
\label{prop:flat-terminal}
\leavevmode
\begin{enumerate}[label=\textup{(\roman*)}]
\item $\kappa$ induces a bijection $H_0(G;P)\to H_0(G;P^\flat)$.
\item If $r:P\to Q$ is a coarsening inducing a bijection on $H_0$, there is a
unique morphism $s:Q\to P^\flat$ with $s\circ r=\kappa$.
\item Every coarsening maps to $\mathbf 1_{\operatorname{supp}P}$, and
$P^\flat=\mathbf 1_{\operatorname{supp}P}$ exactly when condition~\textup{(b)}
of \cref{prop:support-criterion} holds.
\end{enumerate}
\end{proposition}

\begin{proof}
(i) An element of $P^\flat(D)$ is a class $[a]$ met by $P(D)$, and a costalk
of $P^\flat$ relates only equal classes.  Hence
$[\,[a]\text{ at }D\,]\mapsto[a]$ is a well-defined map
$H_0(G;P^\flat)\to H_0(G;P)$.  It is inverse to the map induced by $\kappa$,
since a chain of generating pairs in $P$ maps to a chain in $P^\flat$.
(ii) For $y\in Q(D)$ choose $a\in P(D)$ with $r(a)=y$ and put $s(y)=[a]$.  If
also $r(a')=y$, then $a,a'$ have the same image in $H_0(G;Q)$, hence the same
class in $H_0(G;P)$, because $H_0(r)$ is bijective.  On costalks put
$s(z)=\kappa(c)$ for a lift $c$ of $z$.  Since $P^\flat$ is relational,
$\kappa(c)$ is determined by its extensions, which are $s$ of the extensions
of $z$.  Naturality is clear, and uniqueness follows from surjectivity of $r$.
(iii) The trivial coarsening is terminal among cosheaves with support
$\operatorname{supp}P$, and $P^\flat$ has one-point stalks exactly under (b);
its costalks are then relations between one-point sets.
\end{proof}

Every relational cosheaf thus carries a canonical chain of coarsenings
\[
  P\xrightarrow{\ \kappa\ }P^\flat\longrightarrow\mathbf 1_{\operatorname{supp}P}.
\]
The first arrow forgets everything that does not affect components, and by
(ii) $P^\flat$ is the coarsest model indexed by the vertices of $G$ that still
computes $H_0$.  The second arrow forgets exactly the stalk separations.

\paragraph{Pushforward and involutions.}
For a graph morphism $\varphi:G'\to G$ and a cosheaf $P$ on $G'$, the
pushforward $\varphi_*P$ has stalks $\bigsqcup_{\varphi(D')=D}P(D')$ and
costalks $\bigsqcup_{\varphi(e')=e}P(e')$.  Then $\int_G\varphi_*P=\int_{G'}P$,
so $H_0(G;\varphi_*P)=H_0(G';P)$, and
$\operatorname{supp}\varphi_*P=\varphi(\operatorname{supp}P)$.  A persistent
stalk separation of $P$ is one of $\varphi_*P$, but not conversely.  An
\emph{involution} of $P$ covering a graph involution $D\mapsto D^*$ of $G$
consists of compatible bijections $P(D)\to P(D^*)$ and $P(e)\to P(e^*)$
squaring to the identity; it induces involutions of $\int_GP$ and of
$H_0(G;P)$.

\subsection{The fibre cosheaf of a knot}
\label{subsec:fibre-cosheaf}

Let $G_\infty(K)$ be the \emph{planar Reidemeister multigraph}.  Its vertices
are the planar types of diagrams of $K$ in $u^\perp$, that is, diagrams up to
orientation-preserving similarity of the plane, or equivalently spherical
diagrams with a distinguished face.  Its edges are the classes of Reidemeister
moves supported in a disk of the plane.  Forgetting the face at infinity gives
a graph morphism $\varphi:G_\infty(K)\to G_S(K)$.  The projection fibres of
\cref{sec:growth-graphs} are taken over planar types, and by
\cref{lem:planar-type} the planar type is constant on each fibre component.

\begin{definition}[Fibre cosheaf]
\label{def:fibre-cosheaf}
For $\Lambda>\Rop(K)$ and a planar type $p$, let $P_\Lambda(p)$ be the set of
strict admissible components of
\[
  F^<_{\Lambda,u}(p)=\{x\in\calX^<_{\Lambda,u}(K)\mid\Pi_u(x)\text{ is
  regular of planar type }p\}.
\]
For a planar move $\varepsilon:p\to p'$, let $P_\Lambda(\varepsilon)$ be the
set of pairs $(F,F')$ such that some strictly $\Lambda$-admissible
projection-generic movie realizes $\varepsilon$ from a point of $F$ to a point
of $F'$.  For $\Lambda<\Lambda'$ the inclusion of strict sublevels induces
$\iota_{\Lambda\Lambda'}$.  The \emph{strict diagram image} is
$\calH^<_{\Lambda,u}(K)=p_{\Lambda,u}\bigl(\calG^{<,\operatorname{lift}}_{\Lambda,u}(K)\bigr)\subset G_S(K)$.
\end{definition}

Since a strict admissible component at level $\Lambda$ lies in one at level
$\Lambda'$, and a strictly $\Lambda$-admissible movie is strictly
$\Lambda'$-admissible, $(P_\Lambda)_{\Lambda>\Rop(K)}$ is a filtered
relational cosheaf on $G_\infty(K)$.  Moreover
$\calH^<_{\Lambda,u}(K)\subset\calH_{\Lambda,u}(K)\subset\calH^<_{\Lambda',u}(K)$
for $\Lambda<\Lambda'$, so the strict and closed diagram images have the same
birth scales.

\begin{proposition}[The lifted graph is a Grothendieck construction]
\label{prop:lifted-groth}
For every $\Lambda>\Rop(K)$:
\begin{enumerate}[label=\textup{(\roman*)}]
\item the Grothendieck construction over $G_S(K)$ of the relational image of
$\varphi_*P_\Lambda$ is $\calG^{<,\operatorname{lift}}_{\Lambda,u}(K)$, and its
projection is $p_{\Lambda,u}$;
\item $\operatorname{supp}\varphi_*P_\Lambda=\varphi(\operatorname{supp}P_\Lambda)=\calH^<_{\Lambda,u}(K)$;
\item $\iota$ induces the merging maps of lifted graphs of
\cref{sec:growth-graphs} and the inclusions of strict diagram images.
\end{enumerate}
The same holds without a ropelength bound: the unfiltered fibre cosheaf,
defined with admissible components of $F_{\infty,u}(p)$ and admissible
movies, has Grothendieck construction $\calG^{\operatorname{lift}}_{\infty,u}(K)$.
\end{proposition}

\begin{proof}
A lifted vertex $(D,F)$ has $F$ inside one planar type, by
\cref{lem:planar-type}, so it is an element of a stalk of $\varphi_*P_\Lambda$.
A lifted edge $(e,F,F')$ is an element of the image relation over the
classical edge $e$, and passing to the image relation matches the convention
that different movies realizing the same triple give one lifted edge
(\cref{lem:groth-basic}(iii)).  The remaining statements are restatements of
the definitions.
\end{proof}

\begin{theorem}[Descent]
\label{thm:descent}
For every $\Lambda>\Rop(K)$ there are bijections, natural in $\Lambda$,
\[
  H_0\bigl(G_\infty(K);P_\Lambda\bigr)
  =H_0\bigl(G_S(K);\varphi_*P_\Lambda\bigr)
  \cong\pi_0^{\operatorname{adm}}\bigl(\calX^{<}_{\Lambda,u}(K)\bigr),
\]
sending the class of a fibre component to the strict admissible component
containing it.  Similarly, $H_0$ of the unfiltered fibre cosheaf is a single
point.
\end{theorem}

\begin{proof}
Combine \cref{lem:groth-basic}(i), the pushforward identity and
\cref{prop:lifted-groth} with \cref{thm:strict-component-reconstruction}, and
in the unfiltered case with \cref{cor:unfiltered-component-reconstruction}.
\end{proof}

\Cref{thm:descent} says that $\pi_0$ of a strict sublevel is the colimit, over
the Reidemeister multigraph, of $\pi_0$ of the regular fibres glued along
single walls.  This is a degree-zero descent statement for the decomposition
of the sublevel into fibres and walls, which is what justifies the word
cosheaf.  A space-valued version, with the fibres as stalks and spaces of
movies as costalks, would be a homotopy cosheaf, and higher descent for it is
the higher reconstruction problem of \cref{sec:compact-parameters} and
\cref{prob:two-cells}.

\subsection{What a diagram image sees}
\label{subsec:image-merges}

For a regular configuration $x$ and $\Lambda>\Rop(x)$ let
$a_\Lambda(x)\in P_\Lambda(\Pi_ux)$ be the strict admissible fibre component
containing $x$; then $\iota_{\Lambda\Lambda'}a_\Lambda(x)=a_{\Lambda'}(x)$.
For regular $x,y$ put $\Lambda_{xy}=\max\{\Rop(x),\Rop(y)\}$ and
\begin{align*}
  m_u(x,y)&=\inf\{\Lambda>\Lambda_{xy}:[a_\Lambda(x)]=[a_\Lambda(y)]
  \text{ in }H_0(G_\infty(K);P_\Lambda)\},\\
  m^{\calH_\infty}_u(x,y)&=\inf\{\Lambda>\Lambda_{xy}:\Pi_ux\sim\Pi_uy
  \text{ in }\operatorname{supp}P_\Lambda\},\\
  m^{\calH}_u(x,y)&=\inf\{\Lambda>\Lambda_{xy}:\varphi(\Pi_ux)\sim
  \varphi(\Pi_uy)\text{ in }\calH^<_{\Lambda,u}(K)\},
\end{align*}
where $\sim$ means lying in one connected component.
By \cref{thm:descent}, $m_u(x,y)$ is the infimum of the strict levels at which
$x$ and $y$ are joined by a strictly admissible isotopy.  By
\cref{cor:epsilon-genericity}, exactly as in the proof of
\cref{thm:agreement-merge-scales}, it is also the infimum of the closed levels
at which they are joined by an admissible isotopy.

\begin{corollary}[Diagram images give lower bounds]
\label{cor:image-lower-bounds}
For all regular $x,y$,
\[
  m^{\calH}_u(x,y)\le m^{\calH_\infty}_u(x,y)\le m_u(x,y).
\]
\end{corollary}

\begin{proof}
The second inequality is \cref{lem:groth-basic}(ii).  The first holds because
$\varphi$ maps components of $\operatorname{supp}P_\Lambda$ into components of
$\varphi(\operatorname{supp}P_\Lambda)$.
\end{proof}

\begin{theorem}[When a diagram image computes merges]
\label{thm:image-merges}
\leavevmode
\begin{enumerate}[label=\textup{(\roman*)}]
\item $m^{\calH_\infty}_u=m_u$ for all pairs of regular configurations if and
only if the fibre cosheaf has no persistent stalk separation over
$G_\infty(K)$.
\item $m^{\calH}_u=m_u$ for all pairs if and only if $\varphi_*P$ has no
persistent stalk separation over $G_S(K)$; this implies the condition
in~\textup{(i)}.
\end{enumerate}
\end{theorem}

\begin{proof}
Every stalk element is $a_\Lambda(x)$ for some regular $x$.  With this, the
proof of \cref{thm:support-merges} applies verbatim to $P$ and to $\varphi_*P$.
The last statement holds because each stalk of $P$ lies in a stalk of
$\varphi_*P$.
\end{proof}

By \cref{prop:flat-terminal}, the minimal model indexed by diagrams that
computes merges is the flat coarsening $P^\flat_\Lambda$, whose stalk at $p$ is
the set of strict admissible components of the whole sublevel that meet the
fibre of $p$.  The planar and spherical diagram images are further quotients
of it, and \cref{thm:image-merges} identifies exactly what they forget.  If
every planar fibre $F^<_{\Lambda,u}(p)$ is a single strict admissible
component whenever it is nonempty, all stalks of $P$ have at most one element
and the planar diagram image computes merges.  This is the role of the
fibre-barrier question of \cref{subsec:fibre-barriers}.  Over $G_S(K)$ the
stalk of a spherical diagram is the disjoint union over its planar types, so
one-point spherical stalks are not to be expected; the planar multigraph is
the natural base for the question.

The fibre cosheaf is not computable as it stands; computations use finite
certified parts of it.  A \emph{certificate} at level $\Lambda$ consists of a
finite set $S$ of explicit regular configurations of ropelength less than
$\Lambda$; for some planar moves, finitely many explicit strictly
$\Lambda$-admissible projection-generic movies realizing them between points
of fibre components that contain elements of $S$, together with explicit fibre
paths from those points to elements of $S$; and finitely many \emph{rotation
witnesses}, that is, pairs of elements of $S$ on one $\SO(3)$-orbit.  Let
$Q_\Lambda(p)$ be the set of elements of $S$ with view $p$ modulo the recorded
fibre paths, and $Q_\Lambda(\varepsilon)$ the relation given by the recorded
movies.  Then $Q_\Lambda$ is a relational cosheaf on $G_\infty(K)$ with a
morphism $Q_\Lambda\to P_\Lambda$.  Let $H_0^{\mathrm{cert}}$ be the quotient
of $\bigsqcup_pQ_\Lambda(p)$ by the equivalence relation generated by the
costalks and the rotation witnesses.

\begin{proposition}[Soundness of certificates]
\label{prop:certificates}
The morphism $Q_\Lambda\to P_\Lambda$ and the rotation witnesses induce a map
$H_0^{\mathrm{cert}}\to H_0(G_\infty(K);P_\Lambda)$.  Hence if two elements
$x,y\in S$ have equal classes in $H_0^{\mathrm{cert}}$, they are joined by an
explicit strictly $\Lambda$-admissible isotopy, and $m_u(x,y)\le\Lambda$.
\end{proposition}

\begin{proof}
A morphism of cosheaves induces a map of zeroth homologies.  A rotation
witness relates two points joined by a rotation path, which is admissible at
constant ropelength by \cref{lem:rotation-free}, so their stalk elements have
equal classes by \cref{thm:descent}.  A chain of generating pairs is realized
by concatenating the recorded movies, fibre paths and rotation paths.
\end{proof}

Disconnection of a certificate gives no lower bound, because the map
$H_0^{\mathrm{cert}}\to H_0(G_\infty(K);P_\Lambda)$ need not be injective.  A
lower bound needs the full cosheaf, or an argument of a different kind such
as \cref{prop:sign-separator}.


\section{Growth invariants}
\label{sec:growth-invariants}

The monotone diagram filtration $\calH_{\Lambda,u}(K)$ and the lifted
persistence system carry complementary numerical data.

\begin{definition}[Diagram and edge birth scales]
For $D\in G_S(K)$ and a classical Reidemeister edge $e$ of $G_S(K)$, set
\[
  \beta_u(D)=\inf\{\Lambda\mid D\in V(\calH_{\Lambda,u}(K))\},
\]
\[
  \beta_u(e)=\inf\{\Lambda\mid e\in E(\calH_{\Lambda,u}(K))\}.
\]
These are well defined as extended real numbers because $\calH$ is
monotone.  They avoid the ambiguity of assigning a birth time to a
level-dependent projection-fiber component.
\end{definition}

\begin{definition}[Regular ideal bases]
Let $A_{\Lambda,u}(K)$ be the set of vertices of
$\calG_{\Lambda,u}^{\operatorname{lift}}(K)$ whose fiber component contains
the image of a regular-projection point of $I_u(K)$.  For
$\Lambda>\Rop(K)$, define $A^{<}_{\Lambda,u}(K)$ analogously in
$\calG^{<,\operatorname{lift}}_{\Lambda,u}(K)$.  Every ideal component
contains regular points, because a generic spatial rotation preserves
ropelength and gives a regular projection.  By
\cref{cor:epsilon-genericity,thm:strict-component-reconstruction}, regular
points in one ideal admissible component determine vertices in one component
of the strict lifted graph at every level strictly above the ideal level.
\end{definition}

\begin{definition}[Reidemeister radius and diameter]
Define
\[
  \rad_{\Lambda,u}(K)
  =\sup_{v}d_{\calG_{\Lambda,u}^{\operatorname{lift}}}
  \bigl(A_{\Lambda,u}(K),v\bigr)
  \in\mathbb N\cup\{\infty\},
\]
and let $\diam_{\Lambda,u}(K)$ be the extended graph diameter of
$\calG_{\Lambda,u}^{\operatorname{lift}}(K)$.  Distances between different
components are $\infty$.
\end{definition}

The diameter need not be monotone because newly born edges can create
shortcuts.  The birth scales in $\calH$, by contrast, are attached to a
literal increasing filtration.

\begin{definition}[Crossing profile of a finite window]
For a finite set $W\subset V(\calH_{\Lambda,u}(K))$, define
\[
  C_{W,u}(n;K)=\#\{D\in W\mid \crn(D)=n\}.
\]
This is a statistic of the specified window; no finiteness of the whole
sublevel graph is assumed.
\end{definition}


\section{Finite recognition length}
\label{sec:finite-recognition}

Finite recognition is formulated in the monotone diagram-image filtration
$\calH_{\Lambda,u}(K)$.  Diagram labels used to construct $\calH$ are
forgotten when certificates are compared.

\begin{definition}[Visibility]
Let $Q=\mathbb B_R(D)$ be a BC-characteristic certificate in $G_S(K)$.  We
say that $Q$ is \emph{visible at level $\Lambda$} if its underlying finite
typed multigraph, with all specified parallel edges, is contained in
$\calH_{\Lambda,u}(K)$.  Equivalently, every vertex and every specified
Reidemeister edge of $Q$ has a bounded-ropelength realization by level
$\Lambda$.  Visibility is therefore a scale condition; it does not by itself
assert that a copy seen in another filtered graph is a complete ambient
radius ball.
\end{definition}

\begin{definition}[Saturated occurrence]
\label{def:saturated-occurrence}
Let $G'=G_S(K')$ and let
\[
 \iota:Q\hookrightarrow\calH_{\Lambda,u}(K')\subset G'
\]
be a rooted injective typed multigraph embedding preserving parallel-edge
multiplicities.  We call this visible embedding \emph{saturated} if
\[
  \iota(Q)=S_R^{G'}(\iota(D))
\]
as rooted multigraphs.  Thus saturation is a completeness condition relative
to the completed Reidemeister multigraph, whereas visibility only records
what has appeared by the finite ropelength level.
\end{definition}

\begin{lemma}[Finite saturation criterion]
\label{lem:saturation-criterion}
A visible embedding $\iota:Q\hookrightarrow\calH_{\Lambda,u}(K')$ is
saturated if and only if both of the following finite conditions hold:
\begin{enumerate}[label=\textup{(\roman*)}]
\item every interior ambient-valence label is preserved,
\[
  \nu_Q(E)=\val_{G'}(\iota(E))
  \qquad\bigl(d_Q(D,E)<R\bigr);
\]
\item the occurrence is induced on its image vertex set in the completed
ambient multigraph,
\[
  G'[\iota(V(Q))]=\iota(Q).
\]
\end{enumerate}
In particular, condition \textup{(i)} saturates every star based at a vertex
of distance $<R$, while condition \textup{(ii)} rules out additional edges
joining boundary vertices.
\end{lemma}

\begin{proof}
The conditions are immediate for a full rooted-ball occurrence.  Conversely,
let $E$ have distance $<R$ from the root in $Q$.  Since $Q$ itself is a full
radius-$R$ ball, every ambient edge incident to $E$ already lies in $Q$;
hence $\deg_Q(E)=\nu_Q(E)$.  Valence preservation therefore leaves no
additional ambient edge incident to $\iota(E)$.  It follows that no extra
vertex of $G'$ can lie at distance at most $R$ from $\iota(D)$, because a
shortest path to such a vertex would leave the image through a vertex of
distance $<R$.  The only remaining possible discrepancy is an additional
ambient edge whose two endpoints both lie on the boundary sphere of radius
$R$; condition \textup{(ii)} excludes exactly these edges.  Thus the image is
the full radius-$R$ rooted ball.
\end{proof}

\begin{definition}[Pattern visibility and finite recognition length]
For a BC-characteristic certificate $Q$ for $K$, set
\[
  L_u(K;Q)=
  \inf\{\Lambda\ge\Rop(K)\mid Q\subset\calH_{\Lambda,u}(K)\}.
\]
Define
\[
  L_{\operatorname{char},u}(K)=\inf_Q L_u(K;Q),
\]
where $Q$ ranges over all BC-characteristic certificates for $K$.
\end{definition}

Because $\calH$ is monotone, visibility persists at every larger level.  The
infimum need not be attained, so ``recognition length'' means a threshold
value rather than necessarily a first realized level.

\begin{lemma}[Finite spatial realization of a Reidemeister edge]
\label{lem:finite-edge-realization}
Every classical Reidemeister edge $e:D\to D'$ has a smooth spatial movie
whose projection is regular except at one standard singular time realizing
$e$.  The movie has finite maximal ropelength.
\end{lemma}

\begin{proof}
Choose planar representatives of $D$ and $D'$ exhibiting the prescribed
local move in a disk.  Lift crossings by small, mutually separated height
functions and realize the standard Reidemeister movie in a small spatial
ball; planar isotopies before and after the move are lifted through regular
projections.  Choose the local spatial model with uniform strand separation and
bounded curvature, and keep the complementary part of the curve fixed with a
positive tubular clearance from the move ball.  The resulting compact movie
then has a common embedded normal-tube radius and a uniform length bound.
Consequently its thickness is bounded below by a positive constant and its
ropelength has a finite maximum.
\end{proof}

\begin{theorem}[Finite visibility theorem]
\label{thm:finite-visibility}
Let $Q$ be any finite submultigraph of $G_S(K)$, with specified
Reidemeister edges.  Then there exists $\Lambda_Q<\infty$ such that
\[
  Q\subset\calH_{\Lambda_Q,u}(K).
\]
\end{theorem}

\begin{proof}
Choose one positive-thickness spatial representative for every vertex of
$Q$.  For each of the finitely many edges, apply
\cref{lem:finite-edge-realization}.  Let $\Lambda_Q$ be the maximum of the
ropelengths of the chosen vertex representatives and of all curves occurring
in the finitely many edge movies.  Every vertex and every specified edge of
$Q$ then belongs to $\calH_{\Lambda_Q,u}(K)$.
\end{proof}

\begin{corollary}[Appearance versus eventual visible saturation]
\label{cor:appearance-visible-saturated}
Let $Q=\mathbb B_R(D)$ be a finite rooted-ball certificate and let
$G'=G_S(K')$.  The following are equivalent:
\begin{enumerate}[label=\textup{(\roman*)}]
\item $Q$ occurs as a full rooted ball in $G'$;
\item for some finite $\Lambda$, the underlying pattern $Q$ has a visible
saturated embedding in $\calH_{\Lambda,u}(K')$.
\end{enumerate}
For a prescribed level $\Lambda$, only the implication
\textup{(ii)}$\Rightarrow$\textup{(i)} is automatic; the converse may require
increasing the ropelength level.
\end{corollary}

\begin{proof}
A visible saturated embedding is a full rooted-ball occurrence by
definition.  Conversely, suppose $Q$ occurs as a full rooted ball in $G'$.
Its image is a finite submultigraph of $G'$, so
\cref{thm:finite-visibility} makes every one of its vertices and specified
edges visible at some finite level.  The occurrence remains the same full
ambient rooted ball, hence is saturated at that level.
\end{proof}

\begin{lemma}[Planar type is constant along wall-free movies]
\label{lem:planar-type}
A regular projected curve in $u^\perp$ has a planar combinatorial type,
namely its equivalence class under orientation-preserving ambient isotopy
of the plane.  Under the one-point compactification, the planar type
refines the spherical diagram class by the choice of the face containing
$\infty$.  Along a continuous family of curves with regular projections,
the planar type is locally constant.  Consequently every fiber component of
$F_{\Lambda,u}(D)$ has a single well-defined planar type refining $D$, and
a movie whose projection remains regular cannot change this type.
\end{lemma}

\begin{proof}
A regular plane diagram has a $C^1$-neighborhood consisting of regular
diagrams ambient isotopic to it in the plane, so the planar type is locally
constant along continuous regular families, hence constant on path
components of fibers and along wall-free movies.  For the refinement
statement, two planar pictures of the same spherical diagram are planar
ambient isotopic exactly when their distinguished outer faces correspond.
\end{proof}

\begin{figure}[t]
\centering
\begin{tikzpicture}[font=\scriptsize]
\shade[inner color=white, outer color=figlight] (1.450,0) circle (1.28);
\draw[figsoft, line width=0.5pt] (1.450,0) circle (1.28);
\begin{scope}[shift={(1.450,0.000)}, scale=0.780]
\draw[knot] (-0.620,0.118) -- (-0.549,0.279) -- (-0.498,0.384) -- (-0.441,0.483) -- (-0.381,0.571) -- (-0.312,0.656) -- (-0.237,0.733) -- (-0.155,0.802) -- (-0.069,0.861) -- (0.026,0.913) -- (0.126,0.955) -- (0.230,0.987) -- (0.337,1.008) -- (0.445,1.017) -- (0.553,1.015) -- (0.660,1.002) -- (0.766,0.977) -- (0.870,0.941) -- (0.972,0.892) -- (1.067,0.833) -- (1.156,0.762) -- (1.237,0.682) -- (1.309,0.594) -- (1.369,0.499) -- (1.418,0.399) -- (1.457,0.293) -- (1.483,0.186) -- (1.498,0.071) -- (1.499,-0.040) -- (1.489,-0.151) -- (1.466,-0.259) -- (1.431,-0.367) -- (1.385,-0.469) -- (1.345,-0.539) -- (1.300,-0.605) -- (1.249,-0.669) -- (1.194,-0.727) -- (1.133,-0.782) -- (1.070,-0.830) -- (1.003,-0.874) -- (0.935,-0.911) -- (0.863,-0.943) -- (0.786,-0.971) -- (0.708,-0.992) -- (0.629,-1.007) -- (0.550,-1.015) -- (0.470,-1.018) -- (0.389,-1.014) -- (0.310,-1.004) -- (0.228,-0.987) -- (0.147,-0.963) -- (0.070,-0.933) -- (-0.007,-0.896) -- (-0.081,-0.853) -- (-0.151,-0.805) -- (-0.217,-0.751) -- (-0.279,-0.691) -- (-0.331,-0.634) -- (-0.380,-0.572) -- (-0.469,-0.436) -- (-0.539,-0.301) -- (-0.671,-0.008) -- (-0.708,0.061) -- (-0.745,0.121) -- (-0.796,0.186) -- (-0.850,0.241) -- (-0.910,0.285) -- (-0.973,0.319) -- (-1.037,0.341) -- (-1.102,0.352) -- (-1.168,0.353) -- (-1.232,0.342) -- (-1.292,0.320) -- (-1.349,0.288) -- (-1.397,0.248) -- (-1.438,0.198) -- (-1.470,0.142) -- (-1.490,0.082) -- (-1.500,0.017) -- (-1.497,-0.046) -- (-1.483,-0.109) -- (-1.457,-0.167) -- (-1.421,-0.221) -- (-1.375,-0.268) -- (-1.340,-0.294) -- (-1.302,-0.316) -- (-1.262,-0.332) -- (-1.219,-0.345) -- (-1.176,-0.352) -- (-1.132,-0.354) -- (-1.088,-0.351) -- (-1.043,-0.343) -- (-0.957,-0.312) -- (-0.878,-0.263) -- (-0.805,-0.197) -- (-0.740,-0.113);
\end{scope}
\node at (0.588,-0.004) {$A$};
\node at (1.801,-0.004) {$B$};
\node at (1.100,-0.92) {$C$};
\node[anchor=north] at (1.450,-1.42) {(a) $D$ on $S^2$};
\begin{scope}[shift={(5.350,0.000)}, scale=1.000]
\draw[knot] (-0.620,0.118) -- (-0.549,0.279) -- (-0.498,0.384) -- (-0.441,0.483) -- (-0.381,0.571) -- (-0.312,0.656) -- (-0.237,0.733) -- (-0.155,0.802) -- (-0.069,0.861) -- (0.026,0.913) -- (0.126,0.955) -- (0.230,0.987) -- (0.337,1.008) -- (0.445,1.017) -- (0.553,1.015) -- (0.660,1.002) -- (0.766,0.977) -- (0.870,0.941) -- (0.972,0.892) -- (1.067,0.833) -- (1.156,0.762) -- (1.237,0.682) -- (1.309,0.594) -- (1.369,0.499) -- (1.418,0.399) -- (1.457,0.293) -- (1.483,0.186) -- (1.498,0.071) -- (1.499,-0.040) -- (1.489,-0.151) -- (1.466,-0.259) -- (1.431,-0.367) -- (1.385,-0.469) -- (1.345,-0.539) -- (1.300,-0.605) -- (1.249,-0.669) -- (1.194,-0.727) -- (1.133,-0.782) -- (1.070,-0.830) -- (1.003,-0.874) -- (0.935,-0.911) -- (0.863,-0.943) -- (0.786,-0.971) -- (0.708,-0.992) -- (0.629,-1.007) -- (0.550,-1.015) -- (0.470,-1.018) -- (0.389,-1.014) -- (0.310,-1.004) -- (0.228,-0.987) -- (0.147,-0.963) -- (0.070,-0.933) -- (-0.007,-0.896) -- (-0.081,-0.853) -- (-0.151,-0.805) -- (-0.217,-0.751) -- (-0.279,-0.691) -- (-0.331,-0.634) -- (-0.380,-0.572) -- (-0.469,-0.436) -- (-0.539,-0.301) -- (-0.671,-0.008) -- (-0.708,0.061) -- (-0.745,0.121) -- (-0.796,0.186) -- (-0.850,0.241) -- (-0.910,0.285) -- (-0.973,0.319) -- (-1.037,0.341) -- (-1.102,0.352) -- (-1.168,0.353) -- (-1.232,0.342) -- (-1.292,0.320) -- (-1.349,0.288) -- (-1.397,0.248) -- (-1.438,0.198) -- (-1.470,0.142) -- (-1.490,0.082) -- (-1.500,0.017) -- (-1.497,-0.046) -- (-1.483,-0.109) -- (-1.457,-0.167) -- (-1.421,-0.221) -- (-1.375,-0.268) -- (-1.340,-0.294) -- (-1.302,-0.316) -- (-1.262,-0.332) -- (-1.219,-0.345) -- (-1.176,-0.352) -- (-1.132,-0.354) -- (-1.088,-0.351) -- (-1.043,-0.343) -- (-0.957,-0.312) -- (-0.878,-0.263) -- (-0.805,-0.197) -- (-0.740,-0.113);
\end{scope}
\node at (4.244,-0.005) {$A$};
\node at (5.800,-0.005) {$B$};
\node at (4.100,-0.9) {$C$};
\node[anchor=north] at (5.350,-1.42) {(b) planar type $p$: $\infty\in C$};
\node at (7.900,0) {\large $\neq$};
\begin{scope}[shift={(10.050,0.000)}, scale=1.000]
\draw[knot] (-1.120,-0.122) -- (-1.140,-0.210) -- (-1.149,-0.300) -- (-1.148,-0.392) -- (-1.137,-0.484) -- (-1.115,-0.574) -- (-1.084,-0.664) -- (-1.043,-0.750) -- (-0.992,-0.833) -- (-0.933,-0.911) -- (-0.867,-0.982) -- (-0.793,-1.048) -- (-0.713,-1.106) -- (-0.627,-1.158) -- (-0.536,-1.200) -- (-0.441,-1.235) -- (-0.344,-1.261) -- (-0.246,-1.278) -- (-0.148,-1.286) -- (-0.048,-1.286) -- (0.052,-1.278) -- (0.152,-1.261) -- (0.249,-1.236) -- (0.344,-1.202) -- (0.437,-1.161) -- (0.527,-1.112) -- (0.612,-1.056) -- (0.694,-0.993) -- (0.769,-0.924) -- (0.838,-0.848) -- (0.902,-0.767) -- (0.959,-0.681) -- (1.010,-0.591) -- (1.065,-0.464) -- (1.107,-0.333) -- (1.135,-0.197) -- (1.149,-0.060) -- (1.148,0.077) -- (1.132,0.215) -- (1.103,0.349) -- (1.059,0.481) -- (1.001,0.607) -- (0.931,0.726) -- (0.849,0.836) -- (0.756,0.936) -- (0.653,1.025) -- (0.541,1.104) -- (0.422,1.169) -- (0.297,1.220) -- (0.196,1.250) -- (0.095,1.272) -- (-0.008,1.284) -- (-0.112,1.287) -- (-0.215,1.282) -- (-0.316,1.267) -- (-0.416,1.243) -- (-0.512,1.210) -- (-0.605,1.169) -- (-0.693,1.119) -- (-0.775,1.062) -- (-0.850,0.998) -- (-0.920,0.926) -- (-0.980,0.850) -- (-1.033,0.767) -- (-1.076,0.682) -- (-1.114,0.578) -- (-1.139,0.472) -- (-1.149,0.366) -- (-1.146,0.259) -- (-1.129,0.156) -- (-1.098,0.058) -- (-1.055,-0.033) -- (-0.999,-0.115) -- (-0.929,-0.189) -- (-0.850,-0.248) -- (-0.764,-0.291) -- (-0.674,-0.317) -- (-0.582,-0.325) -- (-0.494,-0.313) -- (-0.452,-0.301) -- (-0.412,-0.283) -- (-0.375,-0.262) -- (-0.341,-0.237) -- (-0.288,-0.179) -- (-0.251,-0.112) -- (-0.232,-0.038) -- (-0.232,0.038) -- (-0.251,0.112) -- (-0.288,0.179) -- (-0.341,0.236) -- (-0.407,0.281) -- (-0.479,0.309) -- (-0.557,0.323) -- (-0.638,0.322) -- (-0.720,0.307) -- (-0.799,0.276) -- (-0.874,0.232) -- (-0.943,0.176) -- (-1.003,0.110);
\end{scope}
\node at (9.450,-0.005) {$A$};
\node at (10.483,-0.068) {$C$};
\node at (11.300,1.05) {$B$};
\node[anchor=north] at (10.050,-1.42) {(c) planar type $p'$: $\infty\in B$};
\end{tikzpicture}
\caption{Two planar types of one spherical diagram.  (a)~The one-crossing
unknot diagram $D$ on $S^2$ has three faces: two monogons $A$, $B$ and a face
$C$ bounded by both edges.  (b), (c)~The planar pictures of $D$ with $\infty$
in $C$ and with $\infty$ in $B$; picture~(c) is the image of picture~(b)
under an orientation-preserving M\"obius transformation, which carries the
crossing information along.  They are not planar ambient isotopic, since the
unbounded face is bounded by two edges in~(b) and by one edge in~(c).
Equivalently, (b) arises from a round planar circle by an $R_1$ move adding a
curl on the outside, and (c) by one adding it on the inside, both at an affine
site.  By \cref{lem:planar-type}, no thick movie with everywhere-regular
projection joins lifts of~(b) to lifts of~(c); this planar-class rigidity is
the source of the face-consistency condition of
\cref{def:face-consistent}.}
\label{fig:planar-types}
\end{figure}

\begin{definition}[Face-consistent planar labeling]
\label{def:face-consistent}
Let $Q$ be a finite submultigraph of $G_S(K)$ with specified Reidemeister
edges.  A \emph{planar labeling} of $Q$ assigns to each vertex $D$ a planar
type $p_D$ refining $D$.  A specified edge $e:D\to D'$ is \emph{compatible}
with the labeling if there is a one-parameter family of planar pictures, of
type $p_D$ before and of type $p_{D'}$ after, regular except for a single
wall crossing realizing $e$ at a site in the affine plane.  The labeling is
\emph{face-consistent} if every specified edge of $Q$ is compatible.  Every
tree-shaped $Q$ admits a face-consistent labeling: choose any planar type
at the root and propagate along the edges; each move, performed at an
affine site, determines an admissible planar type of the new endpoint,
because the face containing $\infty$ persists through a move supported away
from it.
\end{definition}

\begin{remark}[The monodromy behind the labeling condition]
\label{rem:face-monodromy}
By \cref{lem:planar-type}, a coherent lift with one fiber component per
vertex forces one planar type per vertex, and every edge movie must respect
these types on both sides.  Around a cycle of specified moves, the induced
tracking of the outer face may return to a different face of the initial
diagram, in which case no labeling makes all edges of that cycle
compatible.  This planar-class monodromy is the precise residue of the
endpoint-coherence problem.  It is a finite combinatorial condition on
$Q$, checkable face by face; it does not involve the geometry of thick
curves, and it plays no role in the diagram-image filtration
$\calH_{\Lambda,u}(K)$ or in finite recognizability.
\end{remark}

\begin{theorem}[Vertex-coherent finite lifting]
\label{thm:coherent-finite-lift}
Let $Q$ be a finite submultigraph of $G_S(K)$ with specified Reidemeister
edges, admitting a face-consistent planar labeling $(p_D)_{D\in V(Q)}$; by
\cref{def:face-consistent} this includes every tree-shaped $Q$.  There is a
finite level $\Lambda_Q^{\operatorname{coh}}$ and, for
each vertex $D$ of $Q$, one lifted vertex $(D,F_D)$ of
$\calG_{\Lambda_Q^{\operatorname{coh}},u}^{\operatorname{lift}}(K)$ such
that every specified edge $e:D\to D'$ of $Q$ lifts to an edge
\[
  (e,F_D,F_{D'}).
\]
Thus the whole finite pattern has a coherent lift using one fiber component
at each shared vertex.  Conversely, if $Q$ admits a vertex-coherent lift at
some finite level, then the planar types of the chosen fiber components
form a face-consistent labeling of $Q$.
\end{theorem}

\begin{proof}
For each vertex $D$ of $Q$, choose a smooth regular planar representative
$d_D:S^1\to u^\perp$ of the labeled planar type $p_D$, and a smooth height
function $h_D:S^1\to\mathbb R$
which realizes the prescribed over--under information.  The spatial curve
\[
  x_D(s)=d_D(s)+h_D(s)u
\]
is an embedded positive-thickness representative of $K$.

Fix a specified edge $e:D\to D'$.  Starting at $x_D$, lift a planar ambient
isotopy which puts $d_D$ into a representative adapted to the chosen local
Reidemeister move; by compatibility of $e$ with the labeling this can be
done within the planar type $p_D$, with the move site in the affine plane.  Perform a standard smooth spatial $R1$, $R2$, or $R3$
movie in a small ball, keeping the complement fixed and all strands
uniformly separated.  By compatibility, the terminal projection has planar type $p_{D'}$, so a
planar ambient isotopy carries it to the fixed representative $d_{D'}$;
lift this isotopy as well.  At this stage the
curve has the form
\[
  d_{D'}(s)+\widehat h_e(s)u
\]
for a height function $\widehat h_e$ realizing the same crossing information
as $h_{D'}$.  The straight interpolation
\[
  h_{e,t}=(1-t)\widehat h_e+t h_{D'}
\]
preserves every strict over--under inequality.  Since the planar projection
is fixed and regular, the curves
$d_{D'}+h_{e,t}u$ remain embedded and give a regular-projection path ending
exactly at $x_{D'}$.

The complete movie from $x_D$ to $x_{D'}$ has one Reidemeister singular time
and is otherwise regular.  It is a compact smooth family of embedded curves;
hence it has a positive common thickness lower bound and a finite common
length upper bound.  Its maximal ropelength is finite.  Taking the maximum
over the finitely many vertices and specified edges of $Q$ gives
$\Lambda_Q^{\operatorname{coh}}$ and the required fiber components $F_D$.

For the converse, suppose a vertex-coherent lift is given.  By
\cref{lem:planar-type}, each chosen fiber component $F_D$ has a single
planar type $p_D$, and the projected movie of each lifted edge is regular
except for one affine wall crossing realizing that edge, starting in type
$p_D$ and ending in type $p_{D'}$.  These types form a face-consistent
labeling, so the labeling hypothesis is precisely the obstruction and not
an artifact of the construction.
\end{proof}

\begin{theorem}[Effective crossing-complexity visibility]
\label{thm:effective-visibility}
There is a universal computable nondecreasing function
$A:\mathbb N\to(0,\infty)$ with the following property.  Every spherical
knot diagram with at most $n$ crossings, and every classical Reidemeister
edge whose two endpoint diagrams have at most $n$ crossings, is visible in
the appropriate diagram-image filtration by level $A(n)$.  Consequently,
for every rooted radius-$R$ ball centered at a $c$-crossing diagram,
\[
  L_u\bigl(K;\mathbb B_R(D)\bigr)\le A(c+2R).
\]
\end{theorem}

\begin{proof}
Up to orientation-preserving isotopy of $S^2$, there are only finitely many
knot diagrams with at most $n$ crossings.  Each has finite valence in the
Reidemeister multigraph, so there are only finitely many Reidemeister edges
whose endpoint diagrams both have at most $n$ crossings.

This finite list is effective.  For each diagram, a planar graph-drawing
algorithm gives a rational polygonal diagram; assigning separated rational
heights at crossings gives an embedded spatial polygon.  For each listed
edge, use a fixed rational polygonal template for the corresponding local
$R1$, $R2$, or $R3$ movie, with the complement kept outside a rational move
box.  After a uniform, explicitly controlled rounding of corners, the
resulting smooth curves and movies have computable length upper bounds,
curvature bounds, and positive separation bounds.  These data give a
certified positive lower bound for thickness.  Taking the maximum of the
finitely many certified ropelength bounds defines $A(n)$; replacing it by
$\max_{j\le n}A(j)$ makes it nondecreasing.

Every Reidemeister move changes crossing number by at most two.  Hence every
vertex in the radius-$R$ ball about a $c$-crossing diagram has at most
$c+2R$ crossings, and the asserted certificate bound follows.
\end{proof}

\begin{remark}
The function $A$ is deliberately universal and very large.  Its role is to
separate logical effectivity from useful geometry.  Finding polynomial,
near-linear, or knot-sensitive visibility bounds is a substantially stronger
problem.
\end{remark}

\subsection{Comparison with move-count bounds}
\label{subsec:move-counts}

Effective statements about Reidemeister moves already exist, and it is worth
saying precisely why they neither imply nor are implied by
\cref{thm:effective-visibility}.

Hass and Lagarias proved that a diagram of the unknot with $n$ crossings can
be reduced to the trivial diagram in at most $2^{10^{11}n}$ Reidemeister
moves \cite{HassLagarias}; Lackenby improved this to the polynomial bound
$(236n)^{11}$ \cite{Lackenby}, and Coward and Lackenby gave a bound for two
diagrams of an arbitrary link \cite{CowardLackenby}.  These are bounds on the
\emph{number of moves} in a diagrammatic sequence.  The quantity $A(n)$ is a
bound on the \emph{ropelength budget} at which the diagrams and moves become
geometrically available.

The two quantities are logically independent.  A bound on the number of moves
says nothing about the thickness of any spatial realisation of the
intermediate diagrams, and a bound on the ropelength at which a move is
visible says nothing about how many moves a reduction needs.  Converting one
into the other would require a per-move ropelength increment bound, that is,
a bound on
\[
  \sup\bigl\{\Lambda_{D'}-\Lambda_D\ :\ D\to D'\ \text{a single Reidemeister
  move}\bigr\},
\]
and no such bound is known; the difficulty is precisely that a move performed
near an existing strand must be given room, so the increment depends on the
ambient configuration and not only on the local model.  This is the same
phenomenon that prevents the heuristic decomposition ``core length plus kink
length'' in \cref{sec:examples}.

The recursion-theoretic side of the picture should also be recorded.
Nabutovsky showed that natural geometric quantities on moduli of thick knots
and hypersurfaces are governed by non-recursive functions \cite{Nabutovsky},
so the claim that $A$ is computable is a genuine assertion rather than a
formality.  What \cref{thm:effective-visibility} does is bound visibility
uniformly over the finite set of diagrams of bounded crossing number; it does
not produce a recursive bound for $L_{\operatorname{char},u}(K)$ in terms of
any classical invariant of $K$, because the characterizing radius $R(D)$
enters, and \cite{BarbensiCeloria} gives a computable choice of $R(D)$ only
away from periodic knots.  Decidability of knot recognition itself is of
course classical \cite{Haken,Hemion,Matveev}, and is not improved here;
see also \cite{Kuperberg} for the complexity-theoretic setting.  Finally,
lower bounds on $L_{\operatorname{char},u}$ would have to pass through
ropelength lower bounds, for which the crossing-number bounds of
Buck--Simon \cite{BuckSimon} and Cantarella--Kusner--Sullivan
\cite{CantarellaKusnerSullivan} are the natural input.

The diagram-image formulation remains convenient because it is monotone in
$\Lambda$.  \Cref{thm:coherent-finite-lift} shows that the stronger
endpoint-coherent realization is available at some finite level for every
pattern admitting a face-consistent planar labeling, with no analytic
hypothesis; the labeling condition is a finite combinatorial matter,
equivalent to coherent liftability, and is not needed for recognizability.

\begin{theorem}[Unconditional finite recognizability]
\label{thm:finite-recognizability}
For every knot type $K$ and every projection direction $u$,
\[
  L_{\operatorname{char},u}(K)<\infty.
\]
If a BC-characteristic certificate visible for $K$ occurs as a saturated
certificate in the filtered diagram graph of another knot type $K'$, then
$K'=K$ or $K'=\mirror K$.
\end{theorem}

\begin{proof}
Choose a diagram $D$ of $K$.  By \cref{thm:BC-finite-local}, some finite
ball $Q=\mathbb B_R(D)$ is BC-characteristic.  By
\cref{thm:finite-visibility}, $Q$ is contained in
$\calH_{\Lambda_Q,u}(K)$ for a finite $\Lambda_Q$.  Hence
$L_{\operatorname{char},u}(K)\le\Lambda_Q<\infty$.

For the second assertion, saturation makes the occurrence a full radius-$R$
rooted ball in $G_S(K')$; equivalently one may verify this by the finite
criterion of \cref{lem:saturation-criterion}.
\Cref{thm:BC-finite-local} then gives the conclusion.  If $K$ is
non-periodic, Barbensi--Celoria's computable choice of $R(D)$ together with
\cref{thm:effective-visibility} yields the explicit recursive upper bound
$A(\crn(D)+2R(D))$.
\end{proof}

\begin{corollary}[Coherent finite recognition length]
For a BC-characteristic certificate $Q$, let
$L_u^{\operatorname{coh}}(K;Q)$ be the infimum of the levels at which $Q$
has a vertex-coherent lift as in \cref{thm:coherent-finite-lift}, and set
\[
  L_{\operatorname{char},u}^{\operatorname{coh}}(K)
  =\inf_Q L_u^{\operatorname{coh}}(K;Q).
\]
If some BC-characteristic certificate for $K$ admits a face-consistent
planar labeling, then
\[
  L_{\operatorname{char},u}^{\operatorname{coh}}(K)<\infty.
\]
\end{corollary}

\begin{proof}
Apply \cref{thm:coherent-finite-lift} to such a certificate.
\end{proof}

Whether every knot type admits a face-consistently labeled characteristic
certificate is a purely combinatorial question about $G_S(K)$; it is posed
in \cref{sec:further}.  Unconditional finite recognizability
(\cref{thm:finite-recognizability}) is independent of it.

\subsection{Quantitative coherent patterns from separated local moves}
\label{subsec:cubes}

The visibility levels supplied by \cref{thm:effective-visibility} are
universal but astronomically coarse, and \cref{thm:coherent-finite-lift}
gives finiteness with no useful bound.  The following commutation theorem
complements both: from purely local data it produces coherently lifted
patterns of exponential size at an explicit, additive ropelength level.

\begin{definition}[Separated system of local moves]
\label{def:separated-system}
Let $\gamma_0$ be a $C^{1,1}$ embedded representative of $K$ with
$\Thi(\gamma_0)\ge1$ whose projection in direction $u$ is regular, with
diagram $D_0$.  A \emph{separated system of local moves} on $\gamma_0$ is a
finite collection $(\Gamma_i,B_i,\tau_i)$, $1\le i\le k$, where each
$B_i\subset\R^3$ is a closed ball, each $\tau_i\in\{R1,R2,R3\}$, and each
$\Gamma_i(t)$, $0\le t\le1$, is a movie of embedded $C^{1,1}$ curves of
thickness at least $1$ starting at $\gamma_0$, stationary outside $B_i$,
whose projected movie crosses exactly one Reidemeister wall, transversely
and of type $\tau_i$, and is otherwise regular, such that
$\operatorname{dist}(B_i,B_j)\ge2$ for all $i\ne j$ and the disks
$p_u(B_1),\dots,p_u(B_k)\subset u^\perp$ are pairwise disjoint.  For
$S\subseteq\{1,\dots,k\}$, the \emph{composite state} $\gamma_S$ agrees
with $\Gamma_i(1)$ inside $B_i$ for $i\in S$ and with $\gamma_0$ elsewhere,
with diagram $D_S=\Pi_u(\gamma_S)$; the \emph{composite movie} performing
move $i\notin S$ from $\gamma_S$ replaces $\Gamma_i(1)$ by $\Gamma_i(t)$
inside $B_i$.
\end{definition}

\begin{theorem}[Cube patterns from separated systems]
\label{thm:separated-cubes}
With the notation of \cref{def:separated-system}, set
\[
  \Lambda_{\operatorname{cube}}
  =\Len(\gamma_0)+\sum_{i=1}^{k}\delta_i,
  \qquad
  \delta_i=\max\Bigl\{0,\;\sup_{t\in[0,1]}\Len(\Gamma_i(t))-\Len(\gamma_0)\Bigr\}.
\]
Then every composite state and every intermediate composite curve is an
embedded $C^{1,1}$ representative of $K$ with thickness at least $1$ and
length, hence ropelength, at most $\Lambda_{\operatorname{cube}}$, and
every composite movie is an admissible projection-generic movie in
$\calX_{\Lambda_{\operatorname{cube}},u}(K)$ crossing exactly one wall,
transversely and of the expected type.  Consequently all diagrams $D_S$
and all transitions $D_S\to D_{S\cup\{i\}}$ are visible in
$\calH_{\Lambda_{\operatorname{cube}},u}(K)$; and if the $2^{k}$ diagrams
$D_S$ are pairwise distinct, the typed $k$-cube they span in $G_S(K)$ has a
vertex-coherent lift at level $\Lambda_{\operatorname{cube}}$, with the
fiber component of $\gamma_S$ at each vertex.
\end{theorem}

\begin{proof}
By spatial separation, no point of $\R^3$ lies within distance less than
$1$ of two distinct balls, so every point of a composite curve has a
neighborhood meeting at most one ball, on which the composite coincides,
with the same local parametrization, with a single \emph{reference curve}
$\gamma_0$, $\Gamma_j(1)$ for $j\in S$, or $\Gamma_i(t)$.  Hence each
composite curve is $C^{1,1}$ with curvature radius at least $1$ almost
everywhere.  Let $p\ne q$ be points of a composite curve.  If $p\in B_i$
and $q\in B_j$ with $i\ne j$, then $|p-q|\ge2$.  Otherwise $p$ and $q$ lie
in a common region $B_l\cup(\R^3\setminus\bigcup_jB_j)$, on which the
composite coincides with one reference curve; so $p$ and $q$ are distinct
points of an embedded curve of thickness at least $1$, and if $(p,q)$ is a
doubly critical pair of the composite, it is a doubly critical pair of that
reference curve, of length at least $2$.  Hence each composite curve is
embedded with thickness at least $1$.  Since the movies are stationary
outside their balls, its length is at most $\Lambda_{\operatorname{cube}}$,
and therefore its ropelength is at most $\Lambda_{\operatorname{cube}}$.
Performing the moves one at a time through the composite movies shows that
all composite curves represent $K$.

For the projected movies, fix $S$ and $i\notin S$ and let $x\in u^\perp$.
By projected separation, if $x\in p_u(B_i)$, the strands over a
neighborhood of $x$ are exactly those of the reference movie
$\Gamma_i(t)$, with the same heights along $u$; if $x\in p_u(B_j)$ with
$j\in S$, they are those of $\Gamma_j(1)$, independent of $t$; outside all
disks they are those of $\gamma_0$.  Since only the strands in $B_i$ move,
every singular event of the projected composite movie occurs over
$p_u(B_i)$ and coincides with the single transverse wall crossing of
$\Gamma_i$, of type $\tau_i$; elsewhere the projection is locally identical
to a regular diagram.  Hence each composite movie is projection-generic
with exactly one wall crossing, through curves of
$\calX_{\Lambda_{\operatorname{cube}},u}(K)$.

Visibility of all vertices and transitions in
$\calH_{\Lambda_{\operatorname{cube}},u}(K)$ follows.  If the $2^{k}$
diagrams are pairwise distinct, assign to the cube vertex $S$ the fiber
component of $\gamma_S$ at level $\Lambda_{\operatorname{cube}}$; each
composite movie realizes the corresponding cube edge between exactly these
components, and every occurrence of a vertex as an endpoint of an edge uses
the same curve $\gamma_S$.  This is a vertex-coherent lift.
\end{proof}

\begin{remark}[Collapse and automatic face-consistency]
\label{rem:cube-collapse}
Identical local moves at interchangeable positions can make different
composite diagrams agree, in which case the cube collapses in $G_S(K)$ and
only subpatterns with pairwise distinct vertex diagrams are lifted
injectively; mixing local moves with pairwise distinguishable diagrammatic
effects avoids the collapse.  The face-consistency condition of
\cref{def:face-consistent} is automatic here, as the converse part of
\cref{thm:coherent-finite-lift} shows: all moves are performed at affine
sites in pairwise disjoint disks of one fixed picture, so the outer face is
tracked trivially around every cycle of the cube.
\end{remark}

\begin{figure}[t]
\centering
\begin{tikzpicture}
\begin{scope}[shift={(0,0)}]
\node[panellabel] at (-2.190,1.443) {(a)};
\fill[figlight] (-1.330,0.760) -- (1.330,0.760) arc[start angle=90, end angle=-90, radius=0.760] -- (-1.330,-0.760) arc[start angle=270, end angle=90, radius=0.760] -- cycle;
\draw[figsoft, line width=0.3pt] (-1.330,0) -- (1.330,0);
\draw[knot] (-1.330,0.380) -- (1.330,0.380) -- (1.382,0.376) -- (1.419,0.369) -- (1.455,0.359) -- (1.490,0.345) -- (1.523,0.327) -- (1.555,0.307) -- (1.584,0.283) -- (1.610,0.256) -- (1.634,0.228) -- (1.655,0.196) -- (1.673,0.163) -- (1.688,0.129) -- (1.698,0.093) -- (1.706,0.056) -- (1.710,0.019) -- (1.710,-0.019) -- (1.706,-0.056) -- (1.690,-0.122) -- (1.666,-0.177) -- (1.630,-0.234) -- (1.589,-0.278) -- (1.542,-0.315) -- (1.490,-0.345) -- (1.434,-0.366) -- (1.367,-0.378) -- (1.330,-0.380) -- (-1.330,-0.380) -- (-1.382,-0.376) -- (-1.448,-0.361) -- (-1.483,-0.348) -- (-1.517,-0.331) -- (-1.548,-0.311) -- (-1.578,-0.288) -- (-1.625,-0.239) -- (-1.663,-0.183) -- (-1.679,-0.150) -- (-1.692,-0.115) -- (-1.702,-0.078) -- (-1.708,-0.041) -- (-1.709,0.026) -- (-1.698,0.093) -- (-1.688,0.129) -- (-1.673,0.163) -- (-1.655,0.196) -- (-1.634,0.228) -- (-1.589,0.278) -- (-1.561,0.302) -- (-1.530,0.323) -- (-1.469,0.354) -- (-1.405,0.373) -- (-1.367,0.378) -- (-1.330,0.380) -- cycle;
\draw[figink, densely dashed, line width=0.45pt] (-0.760,0.456) circle (0.342);
\node[font=\scriptsize] at (-0.760,0.998) {$B_1$};
\draw[figink, densely dashed, line width=0.45pt] (0.760,0.456) circle (0.342);
\node[font=\scriptsize] at (0.760,0.998) {$B_2$};
\draw[{Bar[width=3pt]}-{Bar[width=3pt]}, line width=0.4pt] (-0.418,0.456) -- node[font=\tiny, above=0.5pt] {$\ge2$} (0.418,0.456);
\node[font=\scriptsize] at (0,-0.600) {$\gamma_0$};
\node[font=\tiny, text=figmid, align=center] at (0,-0.960) {tube of radius $\Thi(\gamma_0)=1$};
\end{scope}
\node[panellabel] at (3.650,1.443) {(b)};
\draw[knot, line width=0.8pt] (4.200,0.650) -- (6.300,0.650) -- (6.341,0.647) -- (6.399,0.633) -- (6.453,0.608) -- (6.500,0.573) -- (6.540,0.530) -- (6.571,0.479) -- (6.591,0.423) -- (6.600,0.365) -- (6.597,0.306) -- (6.584,0.254) -- (6.566,0.210) -- (6.537,0.166) -- (6.505,0.131) -- (6.468,0.101) -- (6.426,0.078) -- (6.376,0.060) -- (6.330,0.051) -- (4.200,0.050) -- (4.159,0.053) -- (4.101,0.067) -- (4.053,0.089) -- (4.004,0.123) -- (3.967,0.161) -- (3.934,0.210) -- (3.914,0.260) -- (3.902,0.317) -- (3.901,0.371) -- (3.909,0.423) -- (3.929,0.479) -- (3.960,0.530) -- (3.995,0.569) -- (4.037,0.602) -- (4.090,0.629) -- (4.147,0.645) -- (4.200,0.650) -- cycle;
\node[font=\scriptsize, anchor=north] at (5.250,0.030) {$D_\emptyset$};
\draw[knot, line width=0.8pt] (8.466,0.848) -- (8.488,0.887) -- (8.502,0.925) -- (8.508,0.964) -- (8.505,0.999) -- (8.493,1.032) -- (8.473,1.060) -- (8.445,1.079) -- (8.414,1.088) -- (8.382,1.088) -- (8.351,1.077) -- (8.325,1.057) -- (8.306,1.030) -- (8.295,0.999) -- (8.292,0.964) -- (8.298,0.925) -- (8.312,0.887) -- (8.332,0.851) -- (8.358,0.818) -- (8.390,0.787) -- (8.428,0.757) -- (8.471,0.730) -- (8.515,0.708) -- (8.562,0.689) -- (8.613,0.673) -- (8.700,0.656) -- (8.786,0.650) -- (10.062,0.650) -- (10.115,0.643) -- (10.165,0.627) -- (10.197,0.611) -- (10.232,0.588) -- (10.263,0.561) -- (10.287,0.534) -- (10.310,0.500) -- (10.326,0.468) -- (10.339,0.429) -- (10.347,0.394) -- (10.350,0.353) -- (10.348,0.317) -- (10.342,0.282) -- (10.330,0.243) -- (10.316,0.210) -- (10.294,0.175) -- (10.267,0.143) -- (10.241,0.119) -- (10.203,0.092) -- (10.154,0.069) -- (10.103,0.055) -- (10.056,0.050) -- (7.950,0.050) -- (7.909,0.053) -- (7.862,0.063) -- (7.829,0.075) -- (7.797,0.092) -- (7.768,0.112) -- (7.741,0.135) -- (7.717,0.161) -- (7.696,0.190) -- (7.679,0.221) -- (7.666,0.254) -- (7.656,0.288) -- (7.651,0.323) -- (7.650,0.359) -- (7.653,0.394) -- (7.661,0.429) -- (7.672,0.463) -- (7.684,0.490) -- (7.703,0.520) -- (7.733,0.557) -- (7.768,0.588) -- (7.813,0.617) -- (7.857,0.635) -- (7.885,0.643) -- (7.920,0.649) -- (8.050,0.651) -- (8.139,0.663) -- (8.187,0.673) -- (8.233,0.687) -- (8.277,0.704) -- (8.318,0.724);
\node[font=\scriptsize, anchor=north] at (9.000,0.030) {$D_{\{1\}}$};
\node[font=\tiny] at (8.184,0.950) {$+$};
\draw[knot, line width=0.8pt] (5.932,-1.326) -- (5.973,-1.346) -- (6.017,-1.363) -- (6.063,-1.377) -- (6.111,-1.387) -- (6.200,-1.399) -- (6.330,-1.401) -- (6.365,-1.407) -- (6.393,-1.415) -- (6.437,-1.433) -- (6.477,-1.458) -- (6.513,-1.489) -- (6.547,-1.530) -- (6.566,-1.560) -- (6.580,-1.593) -- (6.591,-1.627) -- (6.598,-1.662) -- (6.600,-1.697) -- (6.598,-1.733) -- (6.592,-1.768) -- (6.582,-1.802) -- (6.568,-1.834) -- (6.550,-1.865) -- (6.529,-1.894) -- (6.505,-1.919) -- (6.477,-1.942) -- (6.453,-1.958) -- (6.421,-1.975) -- (6.388,-1.987) -- (6.341,-1.997) -- (6.300,-2.000) -- (4.194,-2.000) -- (4.147,-1.995) -- (4.096,-1.981) -- (4.053,-1.961) -- (4.009,-1.931) -- (3.979,-1.902) -- (3.956,-1.875) -- (3.934,-1.840) -- (3.920,-1.807) -- (3.908,-1.768) -- (3.902,-1.733) -- (3.900,-1.697) -- (3.903,-1.656) -- (3.911,-1.621) -- (3.924,-1.582) -- (3.940,-1.550) -- (3.963,-1.516) -- (3.987,-1.489) -- (4.018,-1.462) -- (4.047,-1.442) -- (4.085,-1.423) -- (4.135,-1.407) -- (4.188,-1.400) -- (5.464,-1.400) -- (5.550,-1.394) -- (5.637,-1.377) -- (5.688,-1.361) -- (5.735,-1.342) -- (5.779,-1.320) -- (5.822,-1.293) -- (5.860,-1.263) -- (5.892,-1.232) -- (5.918,-1.199) -- (5.938,-1.163) -- (5.952,-1.125) -- (5.958,-1.086) -- (5.955,-1.051) -- (5.944,-1.020) -- (5.925,-0.993) -- (5.899,-0.973) -- (5.868,-0.962) -- (5.836,-0.962) -- (5.805,-0.971) -- (5.777,-0.990) -- (5.757,-1.018) -- (5.745,-1.051) -- (5.742,-1.086) -- (5.748,-1.125) -- (5.762,-1.163) -- (5.784,-1.202);
\node[font=\scriptsize, anchor=north] at (5.250,-2.020) {$D_{\{2\}}$};
\node[font=\tiny] at (6.066,-1.100) {$-$};
\draw[knot, line width=0.8pt] (9.682,-1.326) -- (9.723,-1.346) -- (9.767,-1.363) -- (9.813,-1.377) -- (9.856,-1.386) -- (9.945,-1.398) -- (10.080,-1.401) -- (10.115,-1.407) -- (10.154,-1.419) -- (10.203,-1.442) -- (10.246,-1.473) -- (10.283,-1.511) -- (10.313,-1.555) -- (10.328,-1.587) -- (10.341,-1.627) -- (10.348,-1.667) -- (10.350,-1.703) -- (10.347,-1.744) -- (10.339,-1.779) -- (10.326,-1.818) -- (10.310,-1.850) -- (10.287,-1.884) -- (10.263,-1.911) -- (10.232,-1.938) -- (10.203,-1.958) -- (10.165,-1.977) -- (10.126,-1.990) -- (10.085,-1.998) -- (10.050,-2.000) -- (7.950,-2.000) -- (7.897,-1.995) -- (7.840,-1.979) -- (7.813,-1.967) -- (7.782,-1.949) -- (7.754,-1.927) -- (7.729,-1.902) -- (7.710,-1.880) -- (7.690,-1.850) -- (7.674,-1.818) -- (7.662,-1.785) -- (7.655,-1.756) -- (7.651,-1.721) -- (7.650,-1.685) -- (7.654,-1.650) -- (7.662,-1.615) -- (7.674,-1.582) -- (7.690,-1.550) -- (7.706,-1.525) -- (7.737,-1.489) -- (7.773,-1.458) -- (7.813,-1.433) -- (7.857,-1.415) -- (7.885,-1.407) -- (7.920,-1.401) -- (8.050,-1.399) -- (8.139,-1.387) -- (8.187,-1.377) -- (8.233,-1.363) -- (8.277,-1.346) -- (8.318,-1.326);
\draw[knot, line width=0.8pt] (8.466,-1.202) -- (8.490,-1.160) -- (8.504,-1.120) -- (8.508,-1.081) -- (8.504,-1.044) -- (8.489,-1.011) -- (8.478,-0.996) -- (8.465,-0.984) -- (8.435,-0.967) -- (8.401,-0.961) -- (8.369,-0.966) -- (8.340,-0.980) -- (8.315,-1.004) -- (8.300,-1.033) -- (8.292,-1.067) -- (8.293,-1.104) -- (8.304,-1.144) -- (8.322,-1.183) -- (8.345,-1.216) -- (8.373,-1.248) -- (8.408,-1.278) -- (8.449,-1.307) -- (8.494,-1.332) -- (8.540,-1.353) -- (8.594,-1.371) -- (8.646,-1.385) -- (8.710,-1.395) -- (8.776,-1.400) -- (9.240,-1.399) -- (9.305,-1.393) -- (9.368,-1.381) -- (9.429,-1.364) -- (9.485,-1.342) -- (9.525,-1.322) -- (9.565,-1.298) -- (9.598,-1.273) -- (9.629,-1.245) -- (9.655,-1.216) -- (9.676,-1.185) -- (9.692,-1.155) -- (9.703,-1.122) -- (9.708,-1.089) -- (9.707,-1.057) -- (9.698,-1.027) -- (9.683,-1.002) -- (9.662,-0.981) -- (9.636,-0.967) -- (9.609,-0.961) -- (9.578,-0.963) -- (9.548,-0.975) -- (9.523,-0.994) -- (9.505,-1.022) -- (9.494,-1.055) -- (9.492,-1.091) -- (9.498,-1.128) -- (9.513,-1.166) -- (9.534,-1.202);
\node[font=\scriptsize, anchor=north] at (9.000,-2.020) {$D_{\{1,2\}}$};
\node[font=\tiny] at (8.184,-1.100) {$+$};
\node[font=\tiny] at (9.816,-1.100) {$-$};
\draw[Rone, ->, >={Latex[length=3.5pt]}] (6.750,0.350) -- node[font=\tiny, text=figink, above] {$\Gamma_1$} (7.500,0.350);
\draw[Rone, ->, >={Latex[length=3.5pt]}] (6.750,-1.700) -- node[font=\tiny, text=figink, above] {$\Gamma_1$} (7.500,-1.700);
\draw[Rone, ->, >={Latex[length=3.5pt]}] (5.250,-0.400) -- node[font=\tiny, text=figink, left] {$\Gamma_2$} (5.250,-0.915);
\draw[Rone, ->, >={Latex[length=3.5pt]}] (9.000,-0.400) -- node[font=\tiny, text=figink, right] {$\Gamma_2$} (9.000,-0.915);
\node[font=\tiny, text=figmid, align=center] at (7.125,-0.675) {all at level\\ $\Lambda_{\operatorname{cube}}$};
\end{tikzpicture}
\caption{A separated system of two local $R_1$ moves.  (a)~The stadium curve
$\gamma_0$, made of two unit semicircles and two segments, has thickness one;
the supports $B_1,B_2$ are at spatial distance at least $2$ and have disjoint
projections.  (b)~The composite states realize a typed $2$-cube with a
vertex-coherent lift at the explicit level
$\Lambda_{\operatorname{cube}}=\Len(\gamma_0)+\delta_1+\delta_2$
(\cref{thm:separated-cubes}); the two movies commute.  The curls are drawn
schematically; the thick curl movies themselves are assumed, as in
\cref{subsec:unknot-cubes}.  They have opposite crossing signs, so the four composite diagrams
are pairwise distinct and the cube does not collapse
(\cref{rem:cube-collapse}).}
\label{fig:separated-cube}
\end{figure}

\begin{definition}[Finite recognition data at scale $\Lambda$]
At scale $\Lambda$, the available recognition data are finite visible
rooted-ball patterns in $\calH_{\Lambda,u}(K)$ together with the finite
ambient checks needed to certify saturation as in
\cref{lem:saturation-criterion}; certificates are compared as abstract
decorated multigraphs.  A knot type is recognized at this scale if
one of these certificates is BC-characteristic for it.  This terminology
does not assert that the whole level graph is finite.  Following
\cite{OzawaLattice} we use ``finite knot theory'' only as an informal
descriptive label for data of this kind, and not as the name of a subfield.
\end{definition}


\section{Diagrammatic shadows of ideal-stratum persistence}
\label{sec:merge-trees}

This section records the projection-framed component persistence associated with the lifted graph and its relation to the unframed ideal-stratum theory of \cite{OzawaRFKS}.  Three merge functions occur, and \cref{subsec:three-merges} keeps them apart: the connected merge and the path merge of \cite{OzawaRFKS}, and the admissible merge used here.  The merge tree remains an important invariant, but here it is understood as a diagrammatic shadow of the ropelength-filtered graph-growth process; see \cref{fig:merge-recognition-principle}.  For general background on persistent topology and stability, see \cite{EdelsbrunnerHarer,CohenSteinerEdelsbrunnerHarer}.  The discrete counterpart of the whole picture, for lattice polygons and BFACF moves, is \cite{OzawaLattice}.

\begin{definition}[Strict diagrammatic accessible range]
For $\Lambda>\Rop(K)$, a strict diagrammatic accessible range is a connected
component of $\calG^{<,\operatorname{lift}}_{\Lambda,u}(K)$.
\end{definition}

For an ideal projection-framed admissible component $C\subset I_u(K)$ and
$\Lambda>\Rop(K)$, let $\iota^{<}_{\Lambda,u}(C)$ denote the set of vertices
of the strict lifted graph determined by regular-projection points of $C$.
This set lies in one connected component of the strict lifted graph, by
\cref{cor:epsilon-genericity} and
\cref{thm:strict-component-reconstruction}.

Let $C_i$ and $C_j$ be ideal projection-framed admissible components.
Their geometric merge scale, in the projection-framed space, is
\[
\begin{aligned}
  m(C_i,C_j)
  =
  \inf\bigl\{
    \Lambda\ge \Rop(K)
    \,\bigm|\,
    &\text{there is a $\Lambda$-admissible isotopy}\\
    &\text{joining a point of $C_i$ to a point of $C_j$}
  \bigr\}.
\end{aligned}
\]

For brevity write
\[
  \mathfrak G_\Lambda=\calG^{<,\operatorname{lift}}_{\Lambda,u}(K).
\]
Define the right-relaxed diagrammatic merge scale by
\[
 m_u^{\operatorname{diag},+}(C_i,C_j)
 =\inf\left\{\Lambda>\Rop(K):
 [\iota^{<}_{\Lambda,u}(C_i)]
 = [\iota^{<}_{\Lambda,u}(C_j)]
 \text{ in }\pi_0(\mathfrak G_\Lambda)
 \right\}.
\]
The superscript $+$ records that a connection at the closed level
$\Lambda$ may be detected at every arbitrarily slightly larger strict level.

\begin{theorem}[Unconditional agreement of merge scales]
\label{thm:agreement-merge-scales}
For all ideal admissible components $C_i,C_j$,
\[
  m_u^{\operatorname{diag},+}(C_i,C_j)=m(C_i,C_j).
\]
\end{theorem}

\begin{proof}
If the two ideal components are connected in the strict lifted graph at level
$\Lambda$, \cref{thm:strict-component-reconstruction} gives a strictly
$\Lambda$-admissible isotopy and therefore a closed
$\Lambda$-admissible isotopy.  Hence
\[
  m(C_i,C_j)\le m_u^{\operatorname{diag},+}(C_i,C_j).
\]
Conversely, suppose a closed $\lambda$-admissible isotopy joins the two
components.  For every $\varepsilon>0$,
\cref{cor:epsilon-genericity} replaces it by a projection-generic isotopy
strictly below $\lambda+\varepsilon$.  It therefore gives a path in
$\calG^{<,\operatorname{lift}}_{\lambda+\varepsilon,u}(K)$, so
\[
  m_u^{\operatorname{diag},+}(C_i,C_j)
  \le \lambda+\varepsilon.
\]
Taking first $\varepsilon\downarrow0$ and then the infimum over all such
$\lambda$ gives the reverse inequality.
\end{proof}

\subsection{Connected, path, and admissible merge}
\label{subsec:three-merges}

Three merge functions must be kept apart.  Following
\cite[Table~4]{OzawaRFKS}, for $x,y$ in the normalised framed space set
\[
  \mu_{K,u}(x,y)
  =\inf\bigl\{\Lambda\ :\ x,y\ \text{lie in one connected component of}\
  \calY_{\Lambda,u}(K)\bigr\},
\]
\[
  \mu^{\operatorname{path}}_{K,u}(x,y)
  =\inf\bigl\{\Lambda\ :\ \text{there is a path from } x \text{ to } y \text{ in}\
  \calY_{\Lambda,u}(K)\bigr\},
\]
and let $m$ be the admissible merge scale defined above, which uses jointly
$C^{1,1}$ isotopies with tangent bounded away from zero, that is, the
regularity class of \cite{OzawaSweptArea}.

\begin{lemma}[Comparison]
\label{lem:merge-comparison}
$\mu_{K,u}\le\mu^{\operatorname{path}}_{K,u}=m$.
\end{lemma}

\begin{proof}
The first inequality is formal.  For the second, let $\Gamma$ be a
$\Lambda$-admissible isotopy.  Its length $L(t)=\Len(\Gamma_t)$ is positive
and continuous, so
$\widehat\Gamma_t=\Lambda L(t)^{-1}\Gamma_t$ is a continuous family with
$\Len(\widehat\Gamma_t)=\Lambda$ and
$\Thi(\widehat\Gamma_t)=\Lambda/\Rop(\Gamma_t)\ge1$.  Hence
$\widehat\Gamma$, after constant-speed and centroid normalization, lies
in $W_\Lambda(K)$.  At each end, join its scale to the originally
prescribed representative by a dilation.  Both scales belong to the
interval $[\Thi(\gamma)^{-1},\Lambda/\Len(\gamma)]$, so the whole
endpoint dilation stays in $W_\Lambda(K)$.  This gives a path in
$\calY_{\Lambda,u}(K)$ with the prescribed endpoints.
Conversely, a path in $\calY_{\Lambda,u}(K)$ projects continuously into
the framed $C^1$ model and has ropelength at most $\Lambda$.
By \cref{thm:uniform-slack-approx,prop:reparam-bridge}, for every
$\varepsilon>0$ its endpoint similarity classes are joined by a
$(\Lambda+\varepsilon)$-admissible isotopy; no regularity of their
projections is needed here.  Taking infima gives
$m\le\mu^{\operatorname{path}}_{K,u}$.
\end{proof}

\begin{corollary}[Attainment of connected merge levels]
\label{cor:attained-merge}
The infimum defining $\mu_{K,u}$ is attained: if $x$ and $y$ lie in one
connected component of $\calY_{\Lambda+\varepsilon,u}(K)$ for every
$\varepsilon>0$, then they lie in one connected component of
$\calY_{\Lambda,u}(K)$.  Consequently
$d_{K,u}=\mu_{K,u}-\Rop(K)$ is a finite-valued ultrametric on the set of
connected components of the framed ideal stratum.
\end{corollary}

\begin{proof}
By \cref{thm:sublevel-compact} the spaces $\calY_{\Lambda,u}(K)$ are compact
Hausdorff, and $\bigcap_{\varepsilon>0}\calY_{\Lambda+\varepsilon,u}(K)
=\calY_{\Lambda,u}(K)$.  The argument of
\cite[Lemma~6.7 and Corollary~6.8]{OzawaRFKS} applies verbatim: a nested
intersection of compact connected sets is compact and connected.  The
ultrametric statement is then \cite[Theorem~6.12]{OzawaRFKS}, whose proof
uses only monotonicity, finiteness and this attainment.
\end{proof}

\begin{remark}[Which merge the diagrammatic model computes]
\label{rem:regularity-gap}
\Cref{thm:agreement-merge-scales} computes $m$, the admissible merge scale,
exactly.  The passage from an admissible isotopy to a projection-generic
one is free after arbitrarily small right relaxation
(\cref{cor:epsilon-genericity}).  Conversely,
\cref{thm:uniform-slack-approx} applies to a $C^1$ path with a uniform
ropelength bound.  Thus \cref{lem:merge-comparison} proves within this
paper that $\mu^{\operatorname{path}}_{K,u}=m$.

The remaining comparison with connected merge levels uses the separate
right-relaxed local-connectivity results of
\cite[Theorem~6.18 and Corollary~6.19]{OzawaRFKS}, in their framed form:
the local curve constructions and the compact residual group give the
same comparison with the projection direction retained.  Together these
give $\mu_{K,u}=\mu^{\operatorname{path}}_{K,u}=m$.
The present approximation theorems do not by themselves convert connected
sets into paths.  Nor do they prove attainment of the path infimum at an
exact critical level.  Exact-level local path connectedness is a separate
question; no equivalence between that property and attainment is asserted.
\end{remark}

\begin{definition}[Right-relaxed diagrammatic merge tree]
Assume that the ideal stratum has finitely many admissible components and
that only finitely many mergers occur below the levels considered.  These
finiteness hypotheses are part of the finite merge-tree model and are not
asserted for all knot types.  The right-relaxed diagrammatic merge tree
$\MT_u^{\operatorname{diag},+}(K)$ is the merge tree, in the standard
persistence sense \cite{EdelsbrunnerHarer}, of
\[
  \Lambda\longmapsto
  \pi_0\bigl(\calG^{<,\operatorname{lift}}_{\Lambda,u}(K)\bigr),
  \qquad \Lambda>\Rop(K),
\]
with the ideal components inserted as right germs at
$\Lambda=\Rop(K)$.  By \cref{thm:agreement-merge-scales}, all branch merge
heights agree with the geometric merge scales defined by closed admissible
isotopies.
\end{definition}

\begin{remark}
The merge tree records only the \(H_0\)-level part of graph growth.  It forgets diagram birth scales, edge birth scales, graph distances, crossing profiles, and characteristic recognition patterns.  Thus it is a useful but coarse invariant of the ropelength-growth Reidemeister graph.
\end{remark}

\begin{figure}[t]
\centering
\begin{tikzpicture}[
  >={Latex[length=4pt]},
  font=\scriptsize,
  v/.style={circle, fill=figink, inner sep=0pt, minimum size=2.6pt},
  vC1/.style={circle, draw=figink, fill=white, line width=0.6pt, inner sep=0pt, minimum size=5pt},
  vC2/.style={rectangle, draw=figink, fill=white, line width=0.6pt, inner sep=0pt, minimum size=4.4pt},
  vC3/.style={regular polygon, regular polygon sides=3, draw=figink, fill=white, line width=0.6pt,
              inner sep=0pt, minimum size=6.2pt},
  e/.style={figink, line width=0.55pt},
  enew/.style={figink, line width=0.9pt},
  hull/.style={fill=figlight, rounded corners=3pt},
  br/.style={figink, line width=1.1pt},
  snap/.style={draw=figsoft, rounded corners=3pt, line width=0.4pt},
  tick/.style={figink, line width=0.5pt}
]
\draw[->, line width=0.6pt] (0.2,0) -- (11.85,0) node[right] {$\Lambda$};
\foreach \x/\l in {0.5/{\Rop(K)}, 3.8/{m(C_1,C_2)}, 7.9/{m(C_1,C_3)}}{
  \draw[tick] (\x,0.07) -- (\x,-0.07);
  \node[font=\tiny, above] at (\x,0.02) {$\l$};
}
\begin{scope}[yshift=0.5cm]
\draw[br] (0.5,1.55) -- (3.8,1.55);
\draw[br] (0.5,1.15) -- (3.8,1.15);
\draw[br] (3.8,1.15) -- (3.8,1.55);
\draw[br] (3.8,1.35) -- (7.9,1.35);
\draw[br] (0.5,0.75) -- (4.8,0.75);
\draw[br] (1.4,0.40) -- (4.8,0.40);
\draw[br] (4.8,0.40) -- (4.8,0.75);
\draw[br] (4.8,0.575) -- (7.9,0.575);
\draw[br] (7.9,0.575) -- (7.9,1.35);
\draw[br, ->] (7.9,0.96) -- (11.6,0.96);
\node[vC1] at (0.5,1.55) {};
\node[vC2] at (0.5,1.15) {};
\node[vC3] at (0.5,0.75) {};
\node[font=\tiny, left=2pt] at (0.5,1.55) {$C_1$};
\node[font=\tiny, left=2pt] at (0.5,1.15) {$C_2$};
\node[font=\tiny, left=2pt] at (0.5,0.75) {$C_3$};
\node[v] at (1.4,0.40) {};
\node[font=\tiny, below=1pt] at (1.4,0.40) {birth};
\foreach \p in {(3.8,1.35),(4.8,0.575),(7.9,0.96)}{\fill[figink] \p circle (1.6pt);}
\node[font=\tiny, text=figmid, anchor=west] at (8.3,1.3) {merge tree $\MT^{\operatorname{diag},+}_u(K)$};
\end{scope}
\foreach \cx/\name in {2.35/{\Lambda_1}, 6.2/{\Lambda_2}, 10.05/{\Lambda_3}}{
  \draw[figsoft, densely dotted, line width=0.45pt] (\cx,2.2) -- (\cx,0);
  \draw[tick] (\cx,0.07) -- (\cx,-0.07);
  \node[font=\tiny, below] at (\cx,-0.04) {$\name$};
}
\begin{scope}[shift={(2.35,3.2)}]
\draw[snap] (-1.6,-1.0) rectangle (1.6,1.0);
\fill[hull] (-1.45,0.15) rectangle (-0.25,0.85);
\fill[hull] (-1.45,-0.85) rectangle (-0.55,-0.15);
\fill[hull] (0.35,0.15) rectangle (1.45,0.85);
\fill[hull] (0.35,-0.85) rectangle (1.45,-0.15);
\draw[e] (-1.25,0.5) -- (-0.85,0.7) -- (-0.45,0.4);
\draw[e] (-1.25,-0.5) -- (-0.75,-0.35);
\draw[e] (0.55,0.5) -- (1.2,0.45);
\draw[e] (0.6,-0.45) -- (1.2,-0.55);
\node[vC1] at (-1.25,0.5) {}; \node[v] at (-0.85,0.7) {}; \node[v] at (-0.45,0.4) {};
\node[vC2] at (-1.25,-0.5) {}; \node[v] at (-0.75,-0.35) {};
\node[vC3] at (0.55,0.5) {}; \node[v] at (1.2,0.45) {};
\node[v] at (0.6,-0.45) {}; \node[v] at (1.2,-0.55) {};
\end{scope}
\begin{scope}[shift={(6.2,3.2)}]
\draw[snap] (-1.6,-1.0) rectangle (1.6,1.0);
\fill[hull] (-1.45,-0.85) rectangle (-0.05,0.85);
\fill[hull] (0.35,-0.85) rectangle (1.45,0.85);
\draw[e] (-1.25,0.5) -- (-0.85,0.7) -- (-0.45,0.4);
\draw[e] (-1.25,-0.5) -- (-0.75,-0.35);
\draw[enew] (-0.45,0.4) -- (-0.3,-0.1) -- (-0.75,-0.35);
\draw[e] (0.55,0.5) -- (1.2,0.45);
\draw[e] (0.6,-0.45) -- (1.2,-0.55);
\draw[enew] (1.2,0.45) -- (1.3,0.0) -- (1.2,-0.55);
\draw[enew] (0.55,0.5) -- (0.6,-0.45);
\node[vC1] at (-1.25,0.5) {}; \node[v] at (-0.85,0.7) {}; \node[v] at (-0.45,0.4) {};
\node[vC2] at (-1.25,-0.5) {}; \node[v] at (-0.75,-0.35) {}; \node[v] at (-0.3,-0.1) {};
\node[vC3] at (0.55,0.5) {}; \node[v] at (1.2,0.45) {}; \node[v] at (1.3,0.0) {};
\node[v] at (0.6,-0.45) {}; \node[v] at (1.2,-0.55) {};
\end{scope}
\begin{scope}[shift={(10.05,3.2)}]
\draw[snap] (-1.6,-1.0) rectangle (1.6,1.0);
\fill[hull] (-1.45,-0.85) rectangle (1.45,0.85);
\draw[e] (-1.25,0.5) -- (-0.85,0.7) -- (-0.45,0.4);
\draw[e] (-1.25,-0.5) -- (-0.75,-0.35);
\draw[e] (-0.45,0.4) -- (-0.3,-0.1) -- (-0.75,-0.35);
\draw[e] (0.55,0.5) -- (1.2,0.45);
\draw[e] (0.6,-0.45) -- (1.2,-0.55);
\draw[e] (1.2,0.45) -- (1.3,0.0) -- (1.2,-0.55);
\draw[e] (0.55,0.5) -- (0.6,-0.45);
\draw[enew] (-0.3,-0.1) -- (0.1,0.05) -- (0.55,0.5);
\draw[enew] (-0.85,0.7) -- (-0.2,0.72);
\node[vC1] at (-1.25,0.5) {}; \node[v] at (-0.85,0.7) {}; \node[v] at (-0.45,0.4) {};
\node[vC2] at (-1.25,-0.5) {}; \node[v] at (-0.75,-0.35) {}; \node[v] at (-0.3,-0.1) {};
\node[vC3] at (0.55,0.5) {}; \node[v] at (1.2,0.45) {}; \node[v] at (1.3,0.0) {};
\node[v] at (0.6,-0.45) {}; \node[v] at (1.2,-0.55) {};
\node[v] at (0.1,0.05) {}; \node[v] at (-0.2,0.72) {};
\end{scope}
\node[font=\tiny, text=figmid, anchor=south west] at (0.75,4.22)
  {strict lifted graph $\calG^{<,\operatorname{lift}}_{\Lambda,u}(K)$ at three levels};
\begin{scope}[yshift=-1.25cm]
\node[font=\tiny, text=figmid, anchor=west] at (0.2,0.62) {in $\calH_{\Lambda,u}(K)$: a certificate $Q$};
\draw[tick] (7.35,1.32) -- (7.35,1.18);
\node[font=\tiny, below] at (7.35,1.21) {$L_{\operatorname{char},u}(K)$};
\draw[tick] (8.9,1.32) -- (8.9,1.18);
\node[font=\tiny, below] at (8.9,1.21) {$L_u(K;Q)$};
\foreach \cx/\stage in {2.35/0, 6.2/1, 10.05/2}{
  \begin{scope}[shift={(\cx,-0.05)}]
    \coordinate (r) at (0,0);
    \coordinate (q1) at (0.55,0.3); \coordinate (q2) at (-0.55,0.3);
    \coordinate (q3) at (-0.5,-0.32); \coordinate (q4) at (0.5,-0.32);
    \draw[Rone] (r) -- (q1);
    \draw[Rthree] (r) -- (q3);
    \draw[Rtwo] (r) to[bend left=18] (q2);
    \ifnum\stage>0
      \draw[Rone] (r) -- (q4);
      \draw[Rtwo] (r) to[bend right=18] (q2);
    \else
      \draw[figsoft, line width=0.35pt] (r) -- (q4);
      \draw[figsoft, line width=0.35pt] (r) to[bend right=18] (q2);
    \fi
    \ifnum\stage>1
      \draw[Rtwo] (q1) -- (q4);
    \else
      \draw[figsoft, line width=0.35pt] (q1) -- (q4);
    \fi
    \foreach \p in {r,q1,q2,q3,q4}{\fill[figink] (\p) circle (1.3pt);}
    \draw[figink, line width=0.5pt] (r) circle (2.4pt);
  \end{scope}
}
\node[font=\tiny, text=figmid] at (2.35,-0.62) {partly visible};
\node[font=\tiny, text=figmid] at (6.2,-0.62) {partly visible};
\node[font=\tiny] at (10.05,-0.62) {$Q\subset\calH_{\Lambda_3,u}(K)$};
\end{scope}
\end{tikzpicture}
\caption{Diagrammatic merge scales and finite recognition (schematic).
Top: the strict lifted graph at three levels $\Lambda_1<\Lambda_2<\Lambda_3$,
with its connected components shaded; the circle, square and triangle mark
vertices determined by the ideal components $C_1,C_2,C_3$, and heavier edges
are those that are new at the level shown.  Middle: the right-relaxed
diagrammatic merge tree retains only how components appear and merge, the
$H_0$-level shadow of this graph growth.  Its leaves are the ideal components,
inserted as right germs at $\Rop(K)$, and components born later; its branch
heights between ideal components are the geometric merge scales
(\cref{thm:agreement-merge-scales}).  Bottom: recognition takes place in the
monotone image $\calH_{\Lambda,u}(K)$; a BC-characteristic certificate $Q$
(thin grey edges: not yet visible) is contained in it for every
$\Lambda>L_u(K;Q)$, and $L_{\operatorname{char},u}(K)\le L_u(K;Q)<\infty$
(\cref{thm:finite-visibility,thm:finite-recognizability}).  How the recognition scale compares with
the merge heights is not determined in general.}
\label{fig:merge-recognition-principle}
\end{figure}


\section{Examples and first-stage growth}
\label{sec:examples}

\subsection{The unknot}

For the unknot \(U\), the ideal level is completely understood.

\begin{proposition}[The ideal level of the unknot]
\label{prop:unknot-ideal}
The unframed ideal stratum of the unknot consists of one point after
quotienting by all orientation-preserving Euclidean similarities.  The
projection-framed ideal stratum $I_u(U)$ is connected.  Hence the
projection-framed ideal diagrammatic merge tree of $U$ has one leaf and no
nontrivial merge.
\end{proposition}

\begin{proof}
If \(\Thi(\gamma)=1\), then the curvature satisfies \(\kappa(s)\le1\) almost everywhere.  Hence, with respect to arclength,
\[
  \int_0^{\Len(\gamma)}\kappa(s)\,ds
  \le
  \int_0^{\Len(\gamma)}1\,ds
  =\Len(\gamma).
\]
Fenchel's theorem gives \(\int_0^{\Len(\gamma)}\kappa(s)\,ds\ge2\pi\) for every closed curve~\cite{Fenchel}.  Combining the two inequalities gives \(\Len(\gamma)\ge2\pi\).  Equality in the two inequalities forces \(\kappa=1\) almost everywhere and equality in Fenchel's theorem; the standard equality case gives a round planar circle of radius one.  Thus \(\Rop(U)=2\pi\), and every minimizer is a round unit circle up to Euclidean similarity.  Varying its spatial orientation gives a connected family in the projection-framed quotient, and every regular projection is the crossingless diagram.
\end{proof}

The merge tree of the unknot is trivial, but its ropelength-growth Reidemeister graph is not.  At \(\Lambda=2\pi\), the visible graph starts from the crossingless diagram.  As \(\Lambda\) increases, new unknot diagrams and new liftable Reidemeister transitions should appear.  For example, a thick geometric version of an \(R1\) move requires positive length beyond \(2\pi\), so the first nontrivial graph-growth event is expected to occur only after a positive ropelength gap.

This growth should be regarded as background local diagrammatic noise: local unknotting patterns can occur inside many knot diagrams and therefore do not by themselves identify a nontrivial knot type.  Finite characteristic patterns must detect knot-specific structure rather than such universal local behavior.

\begin{definition}[First Reidemeister birth scales for the unknot]
For \(i=1,2,3\), define
\[
\begin{aligned}
  \Lambda_{R_i,u}(U)
  =
  \inf\bigl\{
    \Lambda\ge 2\pi
    \,\bigm|\,
    & \text{a liftable } R_i\text{-transition appears in }
      \calH_{\Lambda,u}(U) \\
    & \text{but not in }
      \calH_{2\pi,u}(U)
  \bigr\}.
\end{aligned}
\]
\end{definition}

\begin{problem}
Estimate or compute \(\Lambda_{R1,u}(U)\), \(\Lambda_{R2,u}(U)\), and \(\Lambda_{R3,u}(U)\).  More generally, study the growth of \(\rad_{\Lambda,u}(U)\) from the crossingless diagram.
\end{problem}

\subsection{Separated systems on thickness-one unknots}
\label{subsec:unknot-cubes}

\Cref{thm:separated-cubes} gives a concrete route to visible patterns for
the unknot at explicit levels.  The round circle is the only thickness-one
unknot of length $2\pi$ and leaves no room for separated supports, but
thickness-one unknots exist at every length $L\ge2\pi$: the stadium curve
consisting of two unit semicircles joined by two parallel segments of
length $\ell$ has length $2\pi+2\ell$ and thickness one.  Along the
straight segments there is room for local moves with pairwise separated
supports, provided each local movie stays on the outward side of its
segment, away from the opposite strand.  Granting a single $R1$-curl movie
through curves of thickness at least one, based on a straight strand, with
support ball of radius $r$ and maximal length increment $\delta$ --- the
existence of such a movie at some finite $(r,\delta)$ is expected, but it
is not constructed here --- one obtains, for every $k$, a separated system
of $k$ local moves on a stadium curve with $\ell$ of order $k(r+1)$, and
hence visible transitions below level roughly $2\pi+2k(2r+2)+k\delta$.
Whenever the $2^{k}$ composite diagrams are pairwise distinct, for instance
by mixing local moves with pairwise distinguishable effects
(\cref{rem:cube-collapse}), the full typed $k$-cube is vertex-coherently
lifted at that level.  A linear investment of ropelength buys an
exponentially large coherently lifted pattern, in contrast with the
universal bound $A(n)$ of \cref{thm:effective-visibility}.

\begin{problem}[Explicit thick curl movies]
Construct an explicit $R1$-curl movie through curves of thickness at least
one, with explicit support radius and length increment, and deduce explicit
coherently lifted cube patterns for the unknot.
\end{problem}

\subsection{Rigid rotation is free}
\label{subsec:rotation-free}

Before quantitative questions are posed, one degeneracy must be recorded.

\begin{lemma}[Rotation paths]
\label{lem:rotation-free}
Let $\gamma$ represent a point of $\calX_{\Lambda,u}(K)$ and let
$(A_t)_{t\in[0,1]}$ be a path in $\SO(3)$ with $A_0=\operatorname{id}$.  Then
$t\mapsto[A_t\gamma]$ is an admissible isotopy with
$\Rop(A_t\gamma)=\Rop(\gamma)$ for all $t$, and it is constant if and only if
every $A_t$ fixes $u$.
\end{lemma}

\begin{proof}
Rotations are isometries, so they preserve length and thickness, and the
family is jointly smooth with tangent bounded away from zero.  Two
representatives define the same point of $\calX_{\Lambda,u}(K)$ exactly when
they differ by an element of $\Sim^+_u(\mathbb R^3)$, and a rotation lies in
that group precisely when it fixes $u$.
\end{proof}

Writing $T:S^1\to\Sph^2$ for the unit tangent of $\gamma$, the projection of
$A_t\gamma$ fails to be an immersion exactly when $A_t^{-1}u\in\pm T(S^1)$.
Transporting the motion to the direction sphere, a rotation path is a path
$v_t=A_t^{-1}u$ on $\Sph^2$, and a transverse crossing of the \emph{cusp
locus} $\pm T(S^1)$ is an $R_1$ transition --- the birth-of-a-kink
perestroika of plane curves \cite{ArnoldPlaneCurves}.  By
\cref{lem:rotation-free} nothing is paid for it in ropelength.

\begin{proposition}
\label{prop:R1-collapse}
Let $\gamma$ be a ropelength minimiser for $K$ with diagram $D$ in a
direction $u_0$, and suppose the chamber of $u_0$ in the complement of the
projection discriminant of $\gamma$ has a cusp arc on its boundary.  Then
$\calH_{\Rop(K),u}(K)$ contains an $R_1$ edge based at $D$: the first
$R_1$ level of $K$ at $D$ is the ideal level itself.
\end{proposition}

\begin{proof}
Take a path $v_t$ from $u_0$ crossing that cusp arc once, transversally, and
the corresponding rotation path.  By \cref{lem:rotation-free} it lies in
$\calX_{\Rop(K),u}(K)$ at constant ropelength, and its projection is regular
except at the single wall time, where it has one cusp.
\end{proof}

\begin{figure}[t]
\centering
\includegraphics[width=\textwidth]{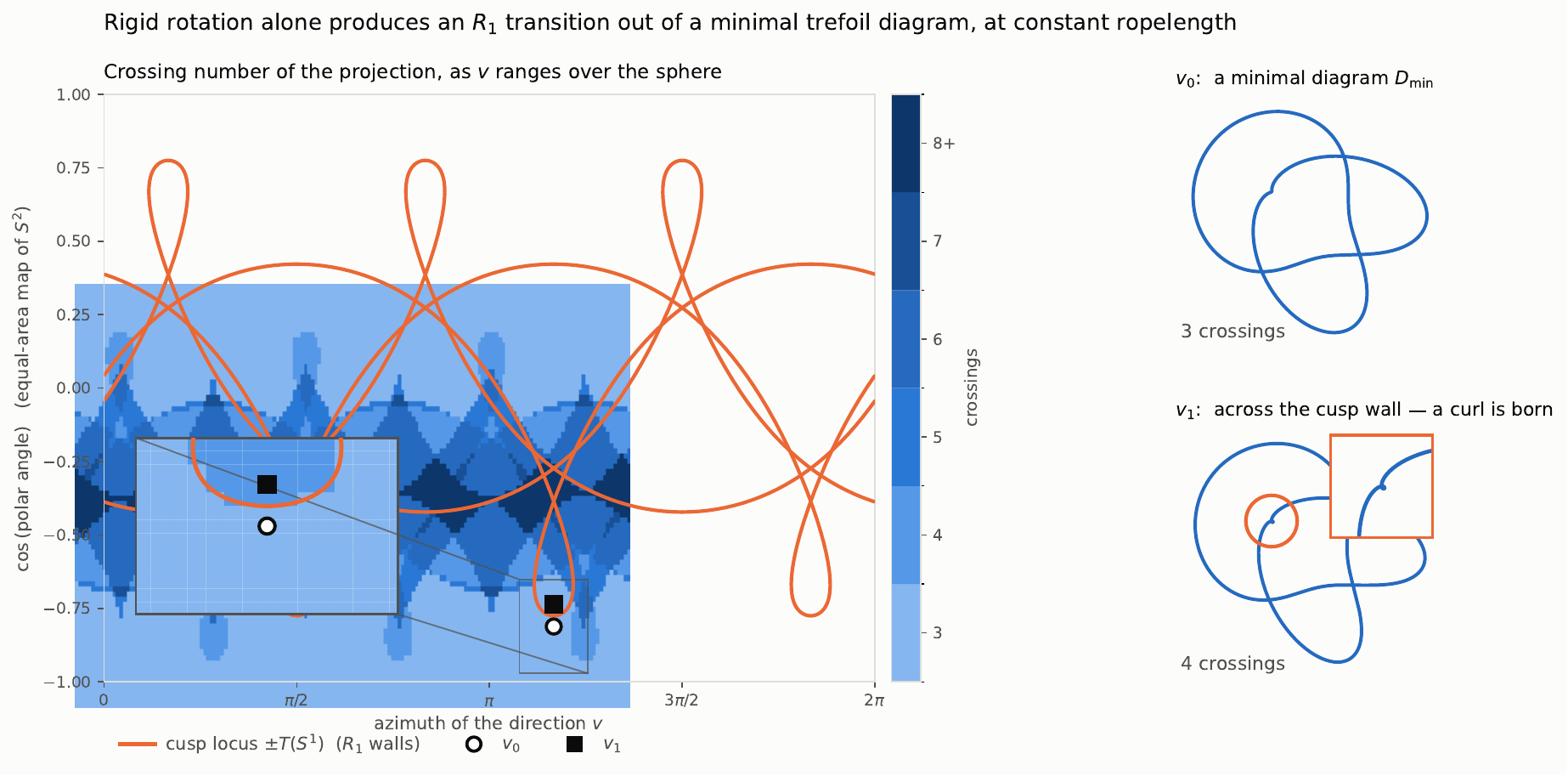}
\caption{The hypothesis of \cref{prop:R1-collapse}, checked for the thick
trefoil $\gamma(t)=((1+r\cos3t)\cos2t,(1+r\cos3t)\sin2t,r\sin3t)$ with
$r=0.45$, for which $\Len(\gamma)\approx15.34$, $\Thi(\gamma)\approx0.447$ and
$\Rop(\gamma)\approx34.30$.  (a)~The crossing number of $\Pi_v(\gamma)$ for
$v$ in the upper hemisphere, in Lambert azimuthal equal-area projection
centred at the vertical direction $e_3$ (the picture is invariant under
$v\mapsto-v$), computed on a grid of about $8\cdot10^5$ directions, together
with the walls: the cusp locus $\pm T(S^1)$, the self-tangency directions and
the trisecant directions.  The chamber containing $e_3$, where the projection
is the standard three-crossing diagram, has three cusp arcs on its boundary.
(b)~Enlargement near one of them.  Along the rotation path from $u_0$ to
$u_1$, at angles $32^\circ$ and $46.5^\circ$ from $e_3$ in the plane spanned
by $e_2$ and $e_3$, the diagram changes exactly once: the path crosses the cusp
arc transversely at $-T(\pi)$, at angle $\arctan(1.1/1.35)\approx39.17^\circ$.
(c), (d)~The two projections, with three and four crossings; the boxed region
is enlarged and shows the curl born at $\gamma(\pi)$.}
\label{fig:rotation-R1}
\end{figure}

\begin{remark}[The unknot is different]
\label{rem:unknot-different}
At $\Lambda=2\pi$ every representative of $U$ is a round circle
(\cref{prop:unknot-ideal}), whose tangent indicatrix is a great circle.  For
$v$ off it the projection is a crossingless ellipse; for $v$ on it the
projection degenerates to a doubly covered segment, the whole curve
collapsing rather than one point of it.  That is not one of the three
codimension-one walls, so no transverse $R_1$ crossing occurs, and by
\cref{lem:rotation-free} rotations are the only paths available at level
$2\pi$.  Hence $\Lambda_{R_1,u}(U)>2\pi$: the positive gap expected for the
unknot is real, and the reason is that the round circle is exactly the case
where the cusp transition is degenerate.
\end{remark}

\begin{remark}[What is unaffected]
\label{rem:rotation-unaffected}
\Cref{prop:R1-collapse} concerns ``some $R_1$ based at $D$''.  Diagram birth
scales are unaffected, since whether a given diagram occurs below a level is
a genuine geometric question; so are edge birth scales with a \emph{specified}
target diagram, which is the form the growth theory should use.  Separated
systems and cube patterns are unaffected, since rotation produces one curl
per wall crossing at a site dictated by the tangent indicatrix, whereas
\cref{thm:separated-cubes} demands $k$ pairwise separated supports.  Merge
scales are unaffected, since a rotation orbit is connected and so lies inside
one admissible component.  Finite recognition is affected only favourably:
part of a characteristic ball is already visible at the ideal level.
\end{remark}

\subsection{Rotation atlases and mirror merges}
\label{subsec:mirror}

For a configuration $x$ write $\calA_\infty(x)$ for the set of regular planar
types $\Pi_u(Ax)$, $A\in\SO(3)$, its \emph{planar atlas}.

\begin{lemma}[Rotation atlases are connected]
\label{lem:rotation-atlas}
Let $\Rop(x)<\Lambda$.  As $A$ ranges over the rotations for which
$\Pi_u(Ax)$ is regular, the lifted vertices $(\Pi_u(Ax),a_\Lambda(Ax))$ lie
in one component of $\int P_\Lambda$.  In particular $\calA_\infty(x)$ lies
in one component of $\operatorname{supp}P_\Lambda$.
\end{lemma}

\begin{proof}
By \cref{lem:rotation-free}, a rotation path between two regular views is a
strictly $\Lambda$-admissible isotopy with regular projected endpoints.  By
\cref{prop:parametric-regularization} it can be replaced, relative to its
endpoints and strictly below $\Lambda$, by a projection-generic isotopy, which
is a path in $\int P_\Lambda$.
\end{proof}

Thus a crossing-number barrier between two diagrams is never a ropelength
barrier between diagrams that occur in one rotation atlas.  An example where
the crossing-number barrier is real is recorded at the end of this subsection.

Let $K$ be amphichiral and let $\rho$ be the reflection in $u^\perp$.  It
fixes $u$, commutes with $\SO(2)_u$, preserves ropelength and admissibility,
and so acts on $\calX^<_{\Lambda,u}(K)$; write $\tau$ for this action.  On
planar types it acts by crossing switch $p\mapsto p^*$, which keeps the shadow
and the face at infinity.  Hence $\rho$ induces an involution of $P_\Lambda$
covering $p\mapsto p^*$ on $G_\infty(K)$ and of $\varphi_*P_\Lambda$ covering
$D\mapsto D^*$ on $G_S(K)$.

\begin{lemma}[Reflection law for atlases]
\label{lem:reflection-law}
For every configuration $x$, $\calA_\infty(\tau x)=\calA_\infty(x)^*$.  In
particular $\tau x=x$ implies $\calA_\infty(x)=\calA_\infty(x)^*$.
\end{lemma}

\begin{proof}
For $A\in\SO(3)$ we have $A\rho=\rho(\rho A\rho)$, hence
$\Pi_u(A\rho\gamma)=\Pi_u\bigl((\rho A\rho)\gamma\bigr)^*$, and
$A\mapsto\rho A\rho$ is a bijection of $\SO(3)$.
\end{proof}

The converse is not claimed: a mirror-invariant atlas does not certify a
mirror-symmetric configuration.

\begin{proposition}[Mirror merges in the three graphs]
\label{prop:mirror-trichotomy}
Let $x$ be regular with planar view $p=\Pi_ux$ and spherical class
$D=\varphi(p)$.  Then $\tau x$ is regular with view $p^*$, and:
\begin{enumerate}[label=\textup{(\roman*)}]
\item if $D^*=D$, then $m^{\calH}_u(x,\tau x)\le\Rop(x)$, so the spherical
diagram image cannot detect a mirror barrier at $x$;
\item if $p^*=p$, then also $m^{\calH_\infty}_u(x,\tau x)\le\Rop(x)$, and a
positive mirror barrier $m_u(x,\tau x)>\Rop(x)$ is exactly a persistent stalk
separation at $p$ by the pair $a(x),\tau a(x)$;
\item if $p^*\ne p$, then $m^{\calH_\infty}_u(x,\tau x)$ is the first level at
which $p$ and $p^*$ lie in one component of $\operatorname{supp}P_\Lambda$; it
is at most $\Rop(x)$ whenever some $q\in\calA_\infty(x)$ has
$q^*\in\calA_\infty(x)$, and may exceed $\Rop(x)$ otherwise.
\end{enumerate}
\end{proposition}

\begin{proof}
$\Pi_u(\rho\gamma)=\Pi_u(\gamma)^*$.  (i) Both stalk elements lie over the
vertex $D$ of $G_S(K)$ at every level above $\Rop(x)$.  (ii) Similarly, both
lie over $p$.  If $m_u(x,\tau x)>\Rop(x)$, choose
$\Rop(x)<\Lambda<\Lambda'<m_u(x,\tau x)$; then $a_\Lambda(x)$ and
$a_\Lambda(\tau x)=\tau a_\Lambda(x)$ are distinct at level $\Lambda'$.
Conversely, such a separation gives $m_u(x,\tau x)\ge\Lambda'>\Rop(x)$.
(iii) The first statement is the definition.  By \cref{lem:rotation-atlas},
$\calA_\infty(x)$ and $\calA_\infty(\tau x)$ each lie in one component of the
support; by \cref{lem:reflection-law}, $\calA_\infty(\tau x)=\calA_\infty(x)^*$,
and under the stated hypothesis the two atlases share $q^*$.
\end{proof}

The planar diagram image is therefore the first diagram-intrinsic object that
can detect a mirror barrier, and only through case~(iii).  Self-switch planar
types arise from symmetric configurations.

\begin{proposition}[Symmetric axial views]
\label{prop:axial-views}
Suppose $\gamma$ is invariant, up to reparametrization, under $\rho\circ R$
for a rotation $R\in\SO(2)_u$, and that $\Pi_u(\gamma)$ is regular of planar
type $p_\gamma$.  Then $p_\gamma^*=p_\gamma$ and $\tau[\gamma]=[\gamma]$ in
$\calX_{\Lambda,u}(K)$.
\end{proposition}

\begin{proof}
$\Pi_u(\rho R\gamma)$ is the crossing switch of the planar rotation of
$\Pi_u(\gamma)$ by $R$, and a planar rotation is an orientation-preserving
similarity.  Moreover $\rho\gamma=R\gamma$ up to reparametrization, and
$R\in\SO(2)_u$ is divided out.
\end{proof}

For instance,
$\gamma(t)=(\Re w(t),\Im w(t),\sin2t+\tfrac12\cos6t)$ with
$w(t)=e^{it}+2e^{-3it}$ satisfies $\gamma(t+\pi/2)=\rho R_{\pi/2}\gamma(t)$,
where $R_{\pi/2}$ is the quarter-turn about $e_3$.  A numerical computation
suggests that it is a figure-eight knot with an eight-crossing regular view
along $e_3$; when this holds, that view is a self-switch planar type of $4_1$.

Symmetric configurations also give upper bounds for mirror merges of
arbitrary configurations.  Since $\rho$ commutes with $\SO(2)_u$, it acts as a
homeomorphism $\tau$ of each normalised framed sublevel
$\calY_{\Lambda,u}(K)$, preserving the length filtration.

\begin{proposition}[Symmetric hubs]
\label{prop:symmetric-hub}
Let $s$ satisfy $\tau s=s$.  Then for every $x$,
\[
  \mu_{K,u}(x,\tau x)\le\mu_{K,u}(x,s),\qquad
  m_u(x,\tau x)\le m_u(x,s),
\]
the second for regular $x$ and $s$.
\end{proposition}

\begin{proof}
If a connected subset of $\calY_{\Lambda,u}(K)$ contains $x$ and $s$, its
image under $\tau$ is connected and contains $\tau x$ and $s$, so the union is
connected and contains $x$ and $\tau x$.  For $m_u$, concatenate a strictly
admissible isotopy from $x$ to $s$ with the reflected reverse isotopy from $s$
to $\tau x$; both have the same maximal ropelength.
\end{proof}

By \cref{prop:axial-views}, the configuration $\gamma$ above is such a hub
for $4_1$ whenever it represents $4_1$.  A certified low-ropelength path from
a given $x$ to it is then a certified upper bound for the mirror merge of
$x$, by \cref{prop:certificates}.  Failure to find such a path is not a lower
bound.  A lower bound requires an obstruction valid on the whole sublevel,
and the following criterion is the simplest one.

\begin{proposition}[Sign separators]
\label{prop:sign-separator}
Let $x\in\calY_{\Lambda,u}(K)$ with $\Len(x)<\Lambda$, and suppose that there
is a continuous function $h:\calY_{\Lambda,u}(K)\to\R\setminus\{0\}$ with
$h(x)\,h(\tau x)<0$.  Then $x$ and $\tau x$ lie in different connected
components of $\calY_{\Lambda',u}(K)$ for every $\Lambda'\le\Lambda$, and
\[
  m(x,\tau x)\ \ge\ \mu_{K,u}(x,\tau x)\ \ge\ \Lambda.
\]
\end{proposition}

\begin{proof}
A connected subset of $\calY_{\Lambda',u}(K)\subset\calY_{\Lambda,u}(K)$
cannot meet both $h^{-1}(0,\infty)$ and $h^{-1}(-\infty,0)$.  Hence $x$ and
$\tau x$ are not in one connected component at any level $\Lambda'\le\Lambda$,
which gives $\mu_{K,u}(x,\tau x)\ge\Lambda$.  The first inequality is
\cref{lem:merge-comparison}.
\end{proof}

An everywhere nonzero function with $h\circ\tau=-h$ would be such a
separator.  Any $\tau$-fixed point of $\calY_{\Lambda,u}(K)$ rules this out,
so for $4_1$ such a function can exist only below the length of every
symmetric hub.  The requirement concerns the whole sublevel, including
configurations with singular projections; a quantity defined only on one
fibre does not meet it.  Constructing a separator for some amphichiral knot
is open; it would give a positive mirror barrier
(\cref{prob:stalk-separation}).

For the standard diagrams the situation is the opposite.  We use the PD
convention in which a row lists the four edge ends at a crossing
counterclockwise, starting with the incoming under-strand.  Let
\[
\begin{aligned}
  D_4&=[4,2,5,1],[8,6,1,5],[6,3,7,4],[2,7,3,8],\\
  D_6&=[4,2,5,1],[8,4,9,3],[12,9,1,10],[10,5,11,6],[6,11,7,12],[2,8,3,7]
\end{aligned}
\]
be the standard diagrams of $4_1$ and $6_3$.

\begin{proposition}[No self-switch planar types]
\label{prop:no-self-switch}
$D_4$ has six faces, two automorphisms, four planar types and two
crossing-switching isomorphisms $D_4\to D_4^*$; $D_6$ has eight faces, one
automorphism, eight planar types and one crossing-switching isomorphism.  None
of these isomorphisms fixes a face.  Hence $D_4^*=D_4$ and $D_6^*=D_6$ in
$G_S$, but no planar type of $D_4$ or $D_6$ is fixed by crossing switch.
\end{proposition}

\begin{proof}
This is a finite check on combinatorial maps.  An isomorphism of connected
maps is determined by the image of one dart, so all of them are enumerated by
propagating each possible image of a root dart through edge pairing and
cyclic order.  The induced permutations of faces are fixed-point free, of
cycle type $(2,2,2)$ for both isomorphisms of $D_4$ and $(2,2,2,2)$ for $D_6$.
The ancillary file \texttt{verify\_planar\_types.py} carries out the check
using only the Python standard library.
\end{proof}

\begin{corollary}[Topological exchange in the spherical fibre]
\label{cor:topological-exchange}
Let $D\in\{D_4,D_6\}$ and let $x$ be regular with spherical view $D$.  Then
$x$ and $\tau x$ lie in different path components of the spherical fibre
$\bigsqcup_{\varphi(p)=D}F^<_{\Lambda,u}(p)$ at every level, for topological
reasons.  A fixed-fibre argument for a mirror barrier of $4_1$ or $6_3$ must
therefore use a self-switch planar type, as in \cref{prop:axial-views}, and
the standard views fall under case~\textup{(iii)} of
\cref{prop:mirror-trichotomy}.
\end{corollary}

\begin{proof}
The spherical fibre is the disjoint union of the open sets
$F^<_{\Lambda,u}(p)$, so every connected subset lies in one of them by
\cref{lem:planar-type}.  If $x$ has view $p$, then $\tau x$ has view
$p^*\ne p$ by \cref{prop:no-self-switch}.
\end{proof}

We close with the example announced after \cref{lem:rotation-atlas}.  It
shows that crossing-number filtrations and ropelength filtrations have
different merge behaviour.

\begin{proposition}[A crossing-number barrier for $6_3$]
\label{prop:63-hard}
There are four seven-crossing diagrams $H_1,\dots,H_4$ of $6_3$, forming two
mirror pairs $H_1^*=H_4$, $H_2^*=H_3$, distinct in $G_S$ even after
identifying a diagram with its view from below, such that:
\begin{enumerate}[label=\textup{(\roman*)}]
\item no Reidemeister move from $H_j$ preserves or lowers the crossing number;
\item each $H_j$ is joined to $D_6$ by a sequence
$\mathrm{R}1^+,\mathrm{R}3,\mathrm{R}3,\mathrm{R}2^-$ through diagrams with
$7,8,8,8,6$ crossings.
\end{enumerate}
Consequently $H_j$ is an isolated vertex of the subgraph of $G_S(6_3)$ spanned
by diagrams with at most seven crossings, and is joined to $D_6$ and to
$H_j^*$ in the subgraph with at most eight crossings.
\end{proposition}

\begin{proof}
(i) Every bigon and triangle of $H_j$ has each edge over at one end and under
at the other, so no bigon admits an $\mathrm{R}2$ reduction and no triangle an
$\mathrm{R}3$ move, and there are no monogons.  (ii) The sequences, with all
intermediate PD codes, are listed in the ancillary file
\texttt{move\_logs.json}.  The file \texttt{verify\_move\_logs.py} performs
each claimed move on PD codes and tests isomorphism with the next diagram,
preserving the sphere orientation and crossing data.  The file
\texttt{verify\_local\_obstructions.py} checks (i) and the mirror pairing.
Since $D_6^*=D_6$, reversing the sequence for $H_j^*$ joins $D_6$ to $H_j^*$.
\end{proof}

Computational certification of crossing-number excess has precedents, notably
the hard unknot diagrams of \cite{BurtonEtAl}.  By contrast, whenever $H_j$ is a vertex of $\calH^<_{\Lambda,u}(6_3)$, it is
realized by some $x$ with $\Rop(x)<\Lambda$.  By \cref{lem:rotation-atlas} it
then lies in one component of the support with every other view of $x$.  The
seven-crossing barrier is therefore not a barrier of the ropelength filtration
in the sense of \cref{cor:image-lower-bounds}.

\subsection{The trefoil as a first nontrivial test case}

The trefoil \(3_1\) is the first natural test case for quantitative finite recognition.  A standard minimal diagram \(D_{\min}\) has three crossings, and known ropelength estimates place the ideal scale in the range
\[
  31.32 < \Rop(3_1) < 32.74317,
\]
using quadrisecant lower bounds and explicit numerical upper bounds \cite{DenneDiaoSullivan,BaranskaPieranskiStasiak}.  Thus the first nontrivial graph-growth events for a near-ideal trefoil occur on a scale already comparable with the tight trefoil itself.

Let \(D_{\min}\) be a three-crossing minimal diagram obtained from a near-ideal representative by a generic projection direction \(u\).  We write
\(D_{\min}\xrightarrow{R_1}_{\Lambda,u}\ast\) if
\(\calH_{\Lambda,u}(3_1)\) contains a liftable
\(R_1\)-transition based at \(D_{\min}\), and define
\[
  \Lambda_{R_1,u}(3_1;D_{\min})
  =
  \inf\bigl\{
    \Lambda\ge \Rop(3_1)
    \mid
    D_{\min}\xrightarrow{R_1}_{\Lambda,u}\ast
  \bigr\}.
\]
\begin{sloppypar}
A separated local-curl model suggests the heuristic scale \(\Rop(3_1)+2\pi\), since a visible local curl in a thickness-one curve must carry substantial turning and hence substantial length.  That model is, however, a controlled subclass of liftable \(R_1\)-events, and by \cref{prop:R1-collapse} it does not describe the unrestricted quantity at all: rigid rotation realises an \(R_1\) at constant ropelength.  The heuristic is best read as an estimate for the edge to a \emph{specified} target diagram, where non-separated events may still interact with nearby strands and with the global reach of the curve, so that the total length need not decompose into ``trefoil core length plus kink length.''
\end{sloppypar}

\begin{problem}[First visible trefoil growth]
\sloppy
By \cref{prop:R1-collapse} the unrestricted quantity
\(\Lambda_{R_1,u}(3_1;D_{\min})\) equals \(\Rop(3_1)\) as soon as some
minimal-crossing chamber of a near-ideal trefoil has a cusp arc on its
boundary, which \cref{fig:rotation-R1} verifies for an explicit thick
trefoil and which is an open condition, hence checkable from certified
near-ideal coordinates.  The quantitative question should therefore be posed
with a specified target: for a given four-crossing diagram \(D'\) obtained
from \(D_{\min}\) by one \(R_1\), estimate the first level at which the edge
\(D_{\min}\to D'\) is liftable, and compare it with the separated local-curl
scale \(\Rop(3_1)+2\pi\).
\end{problem}

\subsection{Nontrivial knots}

For nontrivial knots, exact computation of the lifted ropelength-growth graph appears difficult.  It requires information about ropelength-critical representatives, admissible deformations under a length bound, projection discriminants, and liftable Reidemeister transitions.

Nevertheless, the finite-recognition viewpoint gives a concrete target.  For a given knot \(K\), one seeks a finite BC-characteristic certificate \(Q\subset G_S(K)\) and asks at what ropelength level it becomes visible in \(\calH_{\Lambda,u}(K)\).  This is often more realistic than trying to compute the whole filtered graph.

\begin{problem}
\sloppy
For the trefoil knot, identify a finite BC-characteristic Reidemeister-ball certificate and estimate its visibility length in the ropelength-filtered diagram-image graph.
\end{problem}

\begin{problem}
Compare \(L_{\operatorname{char},u}(K)\) with classical invariants such as crossing number, bridge number, ropelength, and diagram complexity in the sense of Barbensi--Celoria.
\end{problem}



\section{Compact parameters and regular fibres}
\label{sec:compact-parameters}
\begin{theorem}[Relative compact-parameter smoothing]
\label{thm:compact-parameter-smoothing}
Let $P$ be a compact polyhedron, $A\subset P$ a subpolyhedron, and
$F:P\to W^{2,\infty}(S^1,\mathbb R^3)$ a strongly continuous family of
regular embeddings with $\Rop(F_p)\le L$. Suppose on a neighbourhood of $A$
it is a continuous family in the smooth topology. For every $\varepsilon>0$
there is a strong continuous homotopy, fixed near $A$, from $F$ to a
continuous smooth family, through embeddings of ropelength $<L+\varepsilon$.
\end{theorem}
\begin{proof}
Use periodic convolution $S_h$ with a smooth nonnegative kernel of integral
one. Choose a continuous cutoff $\chi$ which vanishes on a neighbourhood of
$A$ and is one outside a smaller closed subset of the region where $F$ is
smooth. Keep $h>0$ fixed and put
\[
 H(v,p)=F_p+v\chi(p)(S_hF_p-F_p).
\]
For fixed $h$ this is strongly continuous. At $v=1$ the nonsmoothed term is
supported where the family is already smooth, so the endpoint is continuous
in every smooth seminorm.
We justify the uniform geometric estimate. Compactness of the strong family
gives positive minimum speed, uniformly bounded second derivative and
uniform $C^1$ convergence $S_hF_p\to F_p$. Write $\kappa_p$ for the essential
maximum curvature of $F_p$; this is continuous in the strong topology. In
projecting $(S_hF_p)''$ onto the normal plane to the tangent of $H(v,p)$,
all tangents sampled by the kernel differ by $o(1)$ uniformly. The normal
part of $F_p''$ has norm at most $\kappa_p|F_p'|^2$, and the remaining
projection error is $o(1)$ by the uniform second-derivative bound. Division
by $|\partial_sH|^2$ gives curvature at most $\kappa_p+o(1)$ uniformly.
Doubly critical chords cannot approach the diagonal, by the uniform speed
and tangent-continuity estimates. If their lengths violated the corresponding
lower bound for $F_p$ by a fixed amount, compactness of $P$ and of the separated
parameter pairs would give a violating doubly critical chord of a limiting
$F_p$. The thickness formula and length continuity now imply
$\Thi(H(v,p))\ge\Thi(F_p)-o(1)$ and $\Len(H(v,p))=\Len(F_p)+o(1)$,
uniformly. Embeddedness follows from uniform $C^1$ tubular stability.
Choose one sufficiently small $h$ to obtain the stated ropelength bound.
No strong convergence $S_hF\to F$ of second derivatives is used, and no
strong continuity of arclength renormalization is required.
\end{proof}

\begin{remark}[Relation with \cref{lem:parametric-smoothing}]
\label{rem:compact-vs-joint}
\Cref{thm:compact-parameter-smoothing} produces smooth \emph{slices}
depending continuously on $p$ in the smooth topology; it does not make the
family jointly smooth in the parameter. The one-parameter construction of
\cref{lem:parametric-smoothing} and \cref{app:parametric-smoothing-proof}
yields an isotopy that is smooth jointly in the curve and time variables
away from the endpoint collars. Joint smoothness is the input required by
the multijet transversality argument of
\cref{prop:parametric-regularization}, and this is why the longer
appendix proof is retained for the one-parameter case.
\end{remark}

\subsection{What a higher reconstruction theorem would require}
Let $Z$ be a specified strict framed sublevel. A family model must retain
actual singular simplices in $Z$, including regular-fibre topology and
attaching maps, not merely abstract diagrams. Suppose a simplicial subset
$C\subset\operatorname{Sing}Z$ is closed also under affine restrictions to
subsimplices. Assume the following relative genericity property: for every
finite triangulated pair $(P,A)$, a map $P\to Z$ whose restriction to $A$ is
represented in $C$ can, after subdivision, be deformed relative to $A$ into
a map represented in $C$. Then the evaluation $|C|\to Z$ is a weak
homotopy equivalence. Indeed, application to spheres proves surjectivity;
application to disks bounding represented spheres proves injectivity.
Subdivision and its standard homotopies are legitimate by the restriction
hypothesis. This is a conditional criterion, not a proof of relative
genericity for the projection discriminant. Smoothing alone proves none of
the required classification of higher events or compatibility of attachments.

Movies supported in pairwise disjoint balls give genuine cubical families
when their endpoint jets match and uniform reach and separation from all
other strands are verified. Length contributions then add. Disjoint supports
without those geometric estimates do not imply a common ropelength bound.

\subsection{Fixed projections and fibre barriers}
\label{subsec:fibre-barriers}
Fix a regular planar immersion $f$ and direction $u$. Its spatial lifts are
$f+h u$. The prescribed over/under conditions are strict linear inequalities
in the height function at the finitely many double points. Their solution
set is convex; such lifts are regularly embedded because $f$ is immersed.
Thus interpolation joins any two smooth lifts with a finite strong
ropelength peak. Extending a smooth planar isotopy gives the corresponding
unfiltered assertion for smooth representatives of a planar diagram type.
Neither assertion gives an endpoint-level ropelength bound.

For a specified planar type $D$, define its fibre barrier by
\[
 B_D(x,y)=\inf_H\sup_t\Rop(H_t),
\]
with $H$ restricted to that fibre and to a specified admissible path class.
In the marked weak model of length-one curves, put
$C_a=\overline{\mathcal F(D)}\cap X_a$, where the closure is taken in the
ambient marked $C^1$ space. If $x,y\in C_M$ lie in different connected
components of $C_M$, then $B_D(x,y)>M$ for paths that remain in the fibre.
Otherwise paths of peak at most $M+1/n$ put $x,y$ in the same compact
connected component of $C_{M+1/n}$. These nested components have connected
intersection containing both points in $C_M$, a contradiction. Notice the
order: this is closure of the entire fibre intersected with the sublevel,
not the closure of the fibre already truncated at that level.

A simple flattening is not a minimax certificate. A unit circle tilted by
$\beta$ and with its height multiplied by $a\in(0,1)$ becomes a planar
ellipse with semiaxes $1$ and
$b=(\cos^2\beta+a^2\sin^2\beta)^{1/2}$. Its reach is $b^2$, and its length
is at least $2\pi b$, so its ropelength is at least $2\pi/b>2\pi$ when
$b<1$. This tests one deformation, not the infimum over all fibre paths.
The existence of a positive fixed-planar-fibre minimax barrier remains a
separate question in this paper.

In the language of \cref{sec:cosheaf}, a planar type $p$ whose strict fibre
is disconnected at level $\Lambda$ has a stalk $P_\Lambda(p)$ with more than
one element.  This affects merges only if it persists in $H_0$
(\cref{thm:image-merges}): two fibre components of $p$ that are joined
through other fibres at a slightly larger level are invisible to every merge
function.  For a mirror pair the relevant fibre is that of a self-switch
planar type; by \cref{cor:topological-exchange} the standard diagrams of
$4_1$ and $6_3$ provide none, so the classical localization of a mirror
barrier in the fibre of a self-switch spherical diagram is vacuous for these
knots.

\section{Further directions}
\label{sec:further}

The unconditional part of the paper consists of compactness of the normalised
projection-framed sublevels and the resulting attainment of connected merge
levels, one-parameter relative parametric smoothing with reach control,
strict-sublevel component reconstruction, exact recovery of the admissible
merge scale through the right-relaxed lifted graph, vertex-coherent
visibility for face-consistently labeled patterns (including all tree-shaped
patterns and all separated cube systems), finite visibility, and
unconditional finite recognizability.  The fibre cosheaf of
\cref{sec:cosheaf} organizes the degree-zero results and isolates what a
diagram image can see.  The remaining questions fall into
three groups: the higher-degree extensions of
\cref{cor:conj82-dim1,cor:conj81-pi0}; attainment at the critical level,
\cref{prob:regularisation}; and quantitative bounds.

\begin{conjecture}[Finite witness principle for standard structures]
\label{conj:finite-principle}
For a specified natural class $\mathfrak I$ of geometric or topological knot
structures, there should be finite decorated certificates in
$\calH_{\Lambda,u}(K)$ from which each member of $\mathfrak I$ can be
recovered by a class-specific finite rule.  In favorable cases the required
scale should be bounded in terms of $L_{\operatorname{char},u}(K)$.
\end{conjecture}

This is a program for explicitly specified structures, not a claim about
arbitrary set-theoretically defined invariants.

\begin{remark}[Higher parameters]\label{prob:higher-smoothing}
The strong parameterized smoothing assertion is proved in
\cref{thm:compact-parameter-smoothing}. Higher discriminant reconstruction
requires the separate relative-genericity hypothesis stated there.
\end{remark}

\begin{problem}[Movie relations and two-cells]
\label{prob:two-cells}
Construct the two-cells of the lifted Reidemeister complex
$\calR^{\operatorname{lift}}_{\Lambda,u}(K)$ from codimension-two strata of
the projection discriminant, and prove that their evaluation recovers the fundamental group of the framed space.  By
\cref{cor:conj81-pi0} the degree-zero statement already holds; in the
language of \cref{sec:cosheaf} it is the descent theorem
\cref{thm:descent}, and the problem asks for descent of the homotopy cosheaf
whose stalks are the fibres and whose costalks are spaces of movies.
\end{problem}

\begin{problem}[Stalk separation]
\label{prob:stalk-separation}
Decide whether the fibre cosheaf of some knot type has a persistent stalk
separation over $G_\infty(K)$, equivalently, by \cref{thm:image-merges},
whether the planar diagram image fails to compute merges.  For an amphichiral
knot and a self-switch planar view, such as the axial view of
\cref{prop:axial-views}, this is the question of a positive mirror barrier
(\cref{prop:mirror-trichotomy}).
\end{problem}

\begin{problem}[Planar mirror distance]
\label{prob:planar-mirror}
For each planar type $p$ of $D_4$ and of $D_6$, determine the first level at
which $p$ and $p^*$ lie in one component of the planar diagram image, and the
least crossing ceiling at which they are joined in $G_\infty(K)$.  By
\cref{prop:no-self-switch} these are the diagrammatic lower bounds of
\cref{cor:image-lower-bounds} for the mirror merges of standard views.
\end{problem}

\begin{problem}[Attainment at the critical level]
\label{prob:regularisation}
Decide whether the infimum defining $\mu^{\operatorname{path}}_{K,u}$ is
attained.  A related question is whether $\calY_{\Lambda,u}(K)$ is locally
path connected without right relaxation; these two assertions are not
identified.  By
\cite[Corollary~6.26 and Remark~6.27]{OzawaRFKS} the obstruction is confined
to the exact-length shell $\{\Len=\Lambda\}$, onto which the space
deformation retracts, and the difficulty is that ropelength is only lower
semicontinuous, so the question needs the structure of ropelength-critical
configurations rather than a sharper interpolation estimate.  The same
obstruction governs finiteness of $\pi_0$ at a single level,
\cite[Question~6.28]{OzawaRFKS}.
\end{problem}

\begin{problem}[Per-move ropelength increment]
\label{prob:increment}
Bound the ropelength increment of a single Reidemeister move performed on a
thick representative, uniformly in the ambient configuration.  By
\cref{subsec:move-counts} such a bound is exactly what would connect
$A(n)$ to the move-count bounds of \cite{Lackenby,CowardLackenby}.
\end{problem}

\begin{problem}[Lattice comparison]
\label{prob:lattice}
Relate $L_{\operatorname{char},u}(K)$ to the lattice recognition length
proposed in \cite{OzawaLattice}, and the merge scale $m$ of
\cref{thm:agreement-merge-scales} to the discrete merge scale of
\cite{OzawaLattice}, in the sense of
\cite{OzawaLattice}.
\end{problem}

\begin{problem}[Exact closed-level components]
Determine natural hypotheses under which the exact closed-level admissible
component relation in $\calX_{\Lambda,u}(K)$ agrees with the right-relaxed
relation $\sim_{\Lambda+}$.  A discrepancy could occur only at a critical
ropelength level and does not change the merge-scale invariant.
\end{problem}

\begin{problem}[Face-consistent characteristic certificates]
Determine whether every knot type admits a BC-characteristic certificate
with a face-consistent planar labeling, for instance a tree-shaped one.
This is a finite combinatorial question about $G_S(K)$.  A positive answer
makes the coherent recognition length
$L_{\operatorname{char},u}^{\operatorname{coh}}(K)$ finite for all knot
types; by \cref{thm:finite-recognizability}, ordinary finite
recognizability does not depend on it.
\end{problem}

\begin{problem}[Quantitative coherent visibility]
When the coherent recognition length
$L_{\operatorname{char},u}^{\operatorname{coh}}(K)$ is finite, estimate its
gap with $L_{\operatorname{char},u}(K)$.  The proof of
\cref{thm:coherent-finite-lift} gives finiteness but no useful bound;
\cref{thm:separated-cubes} gives explicit levels for cube patterns.
\end{problem}

\begin{problem}[Quantitative finite recognition]
\sloppy
Compute or estimate $L_{\operatorname{char},u}(K)$ for knots with few
crossings, by exhibiting a concrete BC-characteristic ball and explicit
ropelength bounds
for all of its vertices and Reidemeister edges.
\end{problem}

\appendix
\section{A complete proof of relative parametric smoothing}
\label{app:parametric-smoothing-proof}

This appendix proves \cref{lem:parametric-smoothing} without invoking a
parametric extension of the smoothing theorem for arbitrary positive-reach
manifolds of \cite{DalPozKourimskaLieutierWintraecken}, whose relative
one-parameter form is not available in the literature.  The argument is special to closed curves.  Its two ingredients
are normalized arclength convolution and the thickness formula
\[
 \Thi(\gamma)
 =\min\left\{\operatorname{MinRad}(\gamma),
              \frac12\operatorname{dcsd}(\gamma)\right\},
\]
where $\operatorname{MinRad}$ is the reciprocal of the essential supremum
of curvature and $\operatorname{dcsd}$ is the doubly-critical
self-distance; see
\cite{LitherlandSimonDurumericRawdon,CantarellaKusnerSullivan}.

\begin{figure}[t]
\centering

\caption{The relative mollification \eqref{eq:app-mollifier}.  Top: the
parameter cylinder, with $\theta$ vertical and its top and bottom edges
identified.  On the fixed smaller collars (dark) the radius $\lambda_j$
vanishes and $\gamma_j=\gamma$; on the rest of the constant endpoint collars
(light) the family does not depend on $t$, so \eqref{eq:app-mollifier} is a
spatial mollification of one fixed curve; elsewhere both variables are
averaged, over discs whose radii are drawn to scale.  Bottom: the radius
vanishes quadratically at the seams, which gives piecewise joint $C^{1,1}$
regularity there, and $\|\lambda_j\|_\infty\to0$ gives the uniform $C^1$
convergence \eqref{eq:app-c1-convergence}.  The dashed profile is a later
radius $\lambda_{j+1}$.}
\label{fig:mollifier}
\end{figure}

\begin{proof}[Proof of \cref{lem:parametric-smoothing}]
Write
\[
 r=r_\Gamma=\inf_{t\in[0,1]}\Thi(\Gamma_t)>0
\]
and fix $0<\delta<r$.  We shall construct a sequence of relative smoothings
whose lengths converge uniformly to the original lengths and whose
thicknesses are eventually at least $r-\delta$.

\smallskip
\noindent\emph{Step 1: normalized arclength coordinates.}
Choose an orientation and a base point on $S^1$.  Put
\[
 L(t)=\Len(\Gamma_t).
\]
The function $L$ is positive and continuous, since $\partial_s\Gamma$ is
continuous and bounded away from zero.  For each $t$, define
\[
 a_t(s)=\frac{1}{L(t)}\int_0^s|\partial_u\Gamma(u,t)|\,du
 \quad\text{modulo }1.
\]
The maps $a_t:S^1\to S^1$ are orientation-preserving homeomorphisms, and
their inverses $b_t$ depend continuously on $(\theta,t)$.  Set
\[
 \gamma(\theta,t)=\Gamma(b_t(\theta),t).
\]
Then $\gamma$ has the same time slices as unparametrized curves and
\[
 |\partial_\theta\gamma(\theta,t)|=L(t).
 \tag{A.1}\label{eq:app-unit-speed}
\]
Moreover, $\gamma$ and $\partial_\theta\gamma$ are continuous on the
parameter cylinder.  Indeed,
\[
 \partial_\theta\gamma(\theta,t)
 =L(t)\,
 \frac{\partial_s\Gamma(b_t(\theta),t)}
      {|\partial_s\Gamma(b_t(\theta),t)|}.
\]
The family remains constant on the endpoint collars.

Since $\Thi(\Gamma_t)\ge r$, the a.e. curvature of every slice is at most
$r^{-1}$.  In the normalized arclength parameter this gives
\[
 |\partial_{\theta\theta}\gamma(\theta,t)|
 \le \frac{L(t)^2}{r}
 \quad\text{for a.e. }\theta.
 \tag{A.2}\label{eq:app-curvature-original}
\]
Equivalently, $\partial_\theta\gamma(\cdot,t)$ is Lipschitz with that
constant.  Let
\[
 L_- =\min_t L(t)>0,
 \qquad
 L_+ =\max_t L(t)<\infty.
\]

\smallskip
\noindent\emph{Step 2: relative variable-radius mollification.}
Extend $\gamma$ periodically in $\theta$ and, using the constant endpoint
collars, constantly a little beyond the interval $[0,1]$ in the time
variable.  Let $\rho\in C_c^\infty(\mathbb R^2)$ be nonnegative, supported
in the unit ball, and satisfy $\int\rho=1$.

Choose an even mollifier $\rho$ and $C^{1,1}$ functions
$\lambda_j:[0,1]\to[0,\infty)$, smooth away from the collar boundaries,
such that
\begin{enumerate}[label=(\roman*)]
\item $\lambda_j=0$ on fixed smaller endpoint collars;
\item $\lambda_j>0$ on the complement of those collars;
\item $\|\lambda_j\|_\infty\to0$;
\item near either boundary point $t_0$ of a fixed collar,
      $\lambda_j(t)=c_j(t-t_0)^2$ on the non-collar side;
\item all points sampled below remain in the chosen extension of the
      parameter cylinder.
\end{enumerate}
Define
\[
 \gamma_j(\theta,t)
 =\int_{\mathbb R^2}\rho(x,y)\,
   \gamma\bigl(\theta-\lambda_j(t)x,
               t-\lambda_j(t)y\bigr)\,dx\,dy.
 \tag{A.3}\label{eq:app-mollifier}
\]
When $\lambda_j(t)=0$, the right-hand side is exactly $\gamma(\theta,t)$.
Thus the endpoint collars are fixed.  Where $\lambda_j(t)>0$, changing
variables in the integral writes \eqref{eq:app-mollifier} as integration against a
smooth compactly supported kernel depending smoothly on $(\theta,t)$;
hence $\gamma_j$ is smooth there.

Since the scale depends only on $t$, differentiation in the $\theta$
direction commutes with the smoothing operator.  In the classical or
distributional sense,
\[
 \partial_\theta\gamma_j
 =S_{\lambda_j}(\partial_\theta\gamma),
 \qquad
 \partial_{\theta\theta}\gamma_j
 =S_{\lambda_j}(\partial_{\theta\theta}\gamma),
 \tag{A.4}\label{eq:app-derivatives}
\]
where $S_{\lambda_j}$ denotes the averaging operator in \eqref{eq:app-mollifier}.
Uniform continuity of $\gamma$ and $\partial_\theta\gamma$ gives
\[
 \|\gamma_j-\gamma\|_{C^0}
 +\|\partial_\theta\gamma_j-
       \partial_\theta\gamma\|_{C^0}
 \longrightarrow0.
 \tag{A.5}\label{eq:app-c1-convergence}
\]
This is the asserted uniform slice-wise $C^1$ approximation.

For completeness, the quadratic vanishing in (iv) gives the regularity at
the seams of the fixed collars.  There the family is constant in $t$, so
\eqref{eq:app-mollifier} reduces to spatial mollification of a fixed
$C^{1,1}$ curve at scale $h=\lambda_j(t)$.  Evenness of the mollifier and
the Lipschitz bound on the first derivative give the standard estimates
\[
 \|S_hf-f\|_{C^0}=O(h^2),
 \qquad
 \|\partial_\theta S_hf-\partial_\theta f\|_{C^0}=O(h),
\]
\[
 \|\partial_hS_hf\|_{C^0}=O(h),
 \qquad
 \|\partial_h\partial_\theta S_hf\|_{C^0}=O(1).
\]
With $h=c_j(t-t_0)^2$, these estimates show that $\gamma_j$ and its first
$\theta$- and $t$-derivatives extend Lipschitz-continuously across the seam.
Thus $\gamma_j$ is piecewise jointly $C^{1,1}$ on the closed parameter
cylinder and smooth off the fixed collars.

\smallskip
\noindent\emph{Step 3: uniform curvature-radius control.}
Let
\[
 \omega_L(h)=
 \sup\{|L(t)-L(t')|:|t-t'|\le h\}
\]
be the modulus of continuity of $L$, and put
\[
 e_j=\|\partial_\theta\gamma_j-
          \partial_\theta\gamma\|_{C^0}.
\]
By \eqref{eq:app-curvature-original}--\eqref{eq:app-derivatives},
\[
 |\partial_{\theta\theta}\gamma_j(\theta,t)|
 \le
 \frac{\bigl(L(t)+\omega_L(\|\lambda_j\|_\infty)\bigr)^2}{r}.
 \tag{A.6}\label{eq:app-second-derivative}
\]
On the other hand, \eqref{eq:app-unit-speed} and \eqref{eq:app-c1-convergence} give
\[
 |\partial_\theta\gamma_j(\theta,t)|
 \ge L(t)-e_j.
 \tag{A.7}\label{eq:app-speed-lower}
\]
For a regular parametrized curve, curvature is at most
$|\gamma''|/|\gamma'|^2$.  Therefore
\[
 \kappa_{\gamma_{j,t}}(\theta)
 \le
 \frac{\bigl(L(t)+\omega_L(\|\lambda_j\|_\infty)\bigr)^2}
      {r\bigl(L(t)-e_j\bigr)^2}.
 \tag{A.8}\label{eq:app-curvature-bound}
\]
The right-hand side converges uniformly to $r^{-1}$.  Consequently, for
all sufficiently large $j$,
\[
 \operatorname{MinRad}(\gamma_{j,t})\ge r-\delta
 \qquad\text{for every }t.
 \tag{A.9}\label{eq:app-minrad}
\]
In particular, the speeds are uniformly bounded below and the first
$\theta$-derivatives are uniformly Lipschitz.

\smallskip
\noindent\emph{Step 4: embeddedness for all large $j$.}
We first show that every slice $\gamma_{j,t}$ is embedded when $j$ is
large.  Otherwise there are $j_k\to\infty$, times $t_k$, and distinct
parameters $\theta_k,\phi_k$ with
\[
 \gamma_{j_k}(\theta_k,t_k)
 =\gamma_{j_k}(\phi_k,t_k).
\]
After passing to a subsequence, all three parameters converge.  If the two
limiting circle parameters are distinct, \eqref{eq:app-c1-convergence} gives a
self-intersection of the limiting original slice, a contradiction.

If the two limiting parameters coincide, let $h_k$ be the signed shortest
circle difference from $\theta_k$ to $\phi_k$.  Then $h_k\to0$.  Uniform
lower speed and the uniform Lipschitz bound for
$\partial_\theta\gamma_{j_k}$ give
\[
 \begin{aligned}
 \bigl|\gamma_{j_k}(\theta_k+h_k,t_k)
       -\gamma_{j_k}(\theta_k,t_k)\bigr|
 &\ge |h_k|\,
       |\partial_\theta\gamma_{j_k}(\theta_k,t_k)|
       -C h_k^2\\
 &>0
 \end{aligned}
\]
for all large $k$, again a contradiction.

\smallskip
\noindent\emph{Step 5: uniform doubly-critical-distance control.}
We claim that, for all sufficiently large $j$,
\[
 \operatorname{dcsd}(\gamma_{j,t})
 \ge 2(r-\delta)
 \qquad\text{for every }t.
 \tag{A.10}\label{eq:app-dcsd}
\]
Suppose not.  Passing to a subsequence, choose times $t_j$ and doubly
critical pairs $(\theta_j,\phi_j)$ such that the chord
\[
 c_j=\gamma_j(\phi_j,t_j)-\gamma_j(\theta_j,t_j)
\]
satisfies
\[
 |c_j|<2(r-\delta),
 \qquad
 c_j\perp\partial_\theta\gamma_j(\theta_j,t_j),
 \qquad
 c_j\perp\partial_\theta\gamma_j(\phi_j,t_j).
 \tag{A.11}\label{eq:app-critical-pair}
\]
After taking a subsequence, $t_j\to t$, $\theta_j\to\theta$, and
$\phi_j\to\phi$.

If $\theta\ne\phi$, uniform $C^1$ convergence passes the two
orthogonality relations to the limit.  Thus $(\theta,\phi)$ is a doubly
critical pair of the original curve $\gamma_t$, and
\[
 |\gamma(\phi,t)-\gamma(\theta,t)|
 \le2(r-\delta)<2r,
\]
contradicting $\Thi(\gamma_t)\ge r$ and the thickness formula.

It remains to exclude $\theta=\phi$.  Let $h_j$ again denote the signed
shortest parameter difference.  With uniform constants $m>0$ and $C<\infty$
such that
\[
 |\partial_\theta\gamma_j|\ge m,
 \qquad
 \operatorname{Lip}_\theta(\partial_\theta\gamma_j)\le C,
\]
we have
\[
 \begin{aligned}
 c_j\cdot\partial_\theta\gamma_j(\theta_j,t_j)
 &=h_j|\partial_\theta\gamma_j(\theta_j,t_j)|^2+E_j,\\
 |E_j|&\le \frac12 C
 |\partial_\theta\gamma_j(\theta_j,t_j)|h_j^2.
 \end{aligned}
 \tag{A.12}\label{eq:app-near-diagonal}
\]
For $0<|h_j|$ sufficiently small, the first line of \eqref{eq:app-near-diagonal} is
nonzero, contradicting the first orthogonality relation in \eqref{eq:app-critical-pair}.
This proves \eqref{eq:app-dcsd}.

Combining \eqref{eq:app-minrad}, \eqref{eq:app-dcsd}, and the thickness formula yields
\[
 \Thi(\gamma_{j,t})\ge r-\delta
 \qquad\text{for every }t
\]
for all sufficiently large $j$.

\smallskip
\noindent\emph{Step 6: length control and conclusion.}
By \eqref{eq:app-unit-speed} and \eqref{eq:app-c1-convergence},
\[
 \begin{aligned}
 \bigl|\Len(\gamma_{j,t})-L(t)\bigr|
 &\le\int_{S^1}
 \bigl|\,|\partial_\theta\gamma_j|-L(t)\bigr|\,d\theta\\
 &\le e_j,
 \end{aligned}
\]
uniformly in $t$.  Choose a sufficiently large $j$ so that $e_j<\epsilon$
and the thickness estimate above holds.  The family has embedded time
slices, is fixed on the endpoint collars, and is smooth on their complement.
It therefore gives the required piecewise jointly $C^{1,1}$, hence
admissible, isotopy, smooth away from the fixed collars; no spatial
smoothness is asserted on a collar carrying a nonsmooth endpoint.  Taking $\Gamma^{\mathrm{sm}}=\gamma_j$ completes the proof.
\end{proof}

\section{Strong path approximation and relative deformation}
\label{app:strong-path-approx}

Throughout this appendix the source $S^1=\mathbb R/\mathbb Z$ is fixed,
and $B=W^{2,\infty}(S^1,\mathbb R^3)$ carries the norm
\[
 \|f\|_B=\|f\|_\infty+\|f'\|_\infty+\|f''\|_{L^\infty}.
\]
No source reparametrization is used in either construction below.

\begin{lemma}[Openness of the strong strict sublevel]
\label{lem:strong-sublevel-open}
On regular $C^{1,1}$ embeddings, inverse thickness and ropelength are
locally Lipschitz in the $B$ norm.  Consequently $U_\Lambda(K)$ in
\cref{lem:strong-path-approx} is open in $B$.
\end{lemma}

\begin{proof}
For a regular embedding $\alpha$, Federer's formula
\cite[Theorem~4.18]{Federer} gives
\[
 \Thi(\alpha)^{-1}=\sup_{s\ne r}Q_\alpha(s,r),\qquad
 Q_\alpha(s,r)=
 \frac{2|P^\perp_{\alpha'(s)}(\alpha(r)-\alpha(s))|}
 {|\alpha(r)-\alpha(s)|^2},
\]
where $P^\perp_v$ is orthogonal projection onto $v^\perp$.
Fix $\alpha$.  Regularity, embeddedness, and compactness imply
$|\alpha(s+h)-\alpha(s)|\ge c|h|$ for some $c>0$ and
$0<|h|\le1/2$.  For every sufficiently $B$-close $\beta$ the same bound
holds with $c/2$, and the speeds stay bounded away from zero.
Write
\[
 R_\alpha(s,h)=\alpha(s+h)-\alpha(s)-h\alpha'(s).
\]
Then, uniformly in $s,h$,
\[
 |R_\alpha(s,h)|\le\tfrac12\|\alpha''\|_\infty h^2,
 \qquad
 |R_\beta(s,h)-R_\alpha(s,h)|
 \le\tfrac12\|\beta''-\alpha''\|_\infty h^2.
\]
The projections depend Lipschitz-continuously on their nonzero tangent
vectors in this neighborhood.  Since they annihilate the linear terms,
the numerators of $Q_\beta$ and $Q_\alpha$ differ by at most
$C\|\beta-\alpha\|_B h^2$.  Their denominators are bounded below by a
positive constant times $h^2$, and their difference is at most
$C\|\beta-\alpha\|_B h^2$.  The numerators themselves are $O(h^2)$.
It follows that
\[
 \sup_{s\ne r}|Q_\beta(s,r)-Q_\alpha(s,r)|
 \le C\|\beta-\alpha\|_B.
\]
Taking suprema proves the inverse-thickness assertion.  Length is locally
Lipschitz in $C^1$, so ropelength is locally Lipschitz as well.  Regular
embeddings of a given knot type form a $C^1$-open set, proving the last
assertion.
\end{proof}

\begin{proof}[Proof of \cref{lem:strong-path-approx}]
By \cref{lem:strong-sublevel-open} and compactness of $\gamma([0,1])$,
there is $\rho>0$ such that every $B$-ball
$B_\rho(\gamma_t)$ is contained in $U_\Lambda(K)$.  Fix
$0<\delta<\min\{\rho,\epsilon\}$.  Uniform continuity supplies a finite
partition $0=t_0<\cdots<t_m=1$ such that, for $t\in[t_{j-1},t_j]$,
\[
 \|\gamma_t-\gamma_{t_{j-1}}\|_B<\delta,
 \qquad \|\gamma_t-\gamma_{t_j}\|_B<\delta.
\]
On this interval put
\[
 \theta=\frac{t-t_{j-1}}{t_j-t_{j-1}},\qquad
 \widetilde\gamma_t=(1-\theta)\gamma_{t_{j-1}}+
                              \theta\gamma_{t_j}.
\]
Then $\|\widetilde\gamma_t-\gamma_t\|_B<\delta$.  Thus every slice stays
in $U_\Lambda(K)$, and continuity on the compact time interval gives
$\sup_t\Rop(\widetilde\gamma_t)<\Lambda$.
Writing $a=\gamma_{t_{j-1}}$, $b=\gamma_{t_j}$, and
$h_j=t_j-t_{j-1}>0$, the derivatives of the resulting cylinder map are
\[
 \partial_s\widetilde\Gamma=(1-\theta)a'+\theta b',\qquad
 \partial_t\widetilde\Gamma=(b-a)/h_j.
\]
Both are jointly Lipschitz on the closed cylinder over this time interval.
Hence the family is piecewise jointly $C^{1,1}$.  Its spatial derivative
is continuous and nonzero on the full compact cylinder, so its magnitude
has a positive minimum.
The straight-line homotopy
\[
 H(v,t)=(1-v)\gamma_t+v\widetilde\gamma_t
\]
stays within $\rho$ of $\gamma_t$ and fixes both endpoints.  It is strongly
continuous.  For the relative assertion, keep the original family on
slightly larger closed endpoint collars contained in its smooth region,
and apply the same construction to the intervening closed interval.
The resulting finite subdivision includes the two collar boundaries.
\end{proof}

\begin{lemma}[Relative smoothing of a strong path]
\label{lem:strong-relative-smoothing}
Let $G:[0,1]\to B$ be a strongly continuous path of regular embeddings
of type $K$ with $M=\max_t\Rop(G_t)<L$.  Suppose that $G$ is jointly
smooth near prescribed closed endpoint collars $A$.  Then $G$ admits a
strongly continuous homotopy, fixed near $A$ and entirely in $U_L(K)$,
to a family of smooth embeddings continuous in the $C^\infty$ topology.
\end{lemma}

\begin{proof}
Let $S_h$ be periodic convolution in the fixed source variable with a
nonnegative smooth probability kernel supported in $[-h,h]$.  Choose a
smooth $\chi:[0,1]\to[0,1]$, zero near $A$, equal to one outside an open
set on which $G$ is jointly smooth.  For a fixed $h>0$ define
\begin{equation}\label{eq:strong-linear-smoothing}
 H_h(v,t)=G_t+v\chi(t)(S_hG_t-G_t),\qquad 0\le v\le1.
\end{equation}
Since $S_h:B\to B$ is bounded linear, this is a strongly continuous
homotopy.  It fixes a neighborhood of $A$.  Its endpoint at $v=1$ has
smooth slices: where $\chi=1$ it is $S_hG_t$, and where $\chi\ne1$ the
original slices are already smooth.  The endpoint family is continuous
in every $C^k$ norm, by the bounded map $S_h:B\to C^k$ and the smoothness
of $G$ on the transition region.

It remains to check the ropelength bound uniformly for the entire homotopy.
The compact family $G_t$ has speed bounded below by some $m>0$ and a
uniform $B$ bound.  Uniformly in $v,t$, $H_h(v,t)$ tends to $G_t$ in
$C^1$ as $h\downarrow0$.  Consequently the slices of $H_h$ are regular
embeddings of type $K$ for all sufficiently small $h$.
Let
\[
 \kappa_t=\operatorname*{ess\,sup}_s
 \frac{|P^\perp_{G_t'(s)}G_t''(s)|}{|G_t'(s)|^2}.
\]
This quantity is continuous in $t$ by strong continuity, and is uniformly
bounded.  We claim that
\begin{equation}\label{eq:strong-smoothing-curvature}
 \kappa(H_h(v,t))\le\kappa_t+o(1)
\end{equation}
uniformly in $v,t$.  To verify this, put $w=v\chi(t)$.  The second
derivative of $H_h$ is the nonnegative average
$(1-w)G_t''+wS_hG_t''$.  For every point $s-\xi$ sampled by the kernel,
$G_t'(s-\xi)$ differs from $\partial_sH_h(v,t)(s)$ by $o(1)$ uniformly.
Their perpendicular projections therefore differ by $o(1)$ as operators,
and their speeds differ by $o(1)$.  Projecting the preceding average onto
$(\partial_sH_h)^\perp$ gives, almost everywhere,
\[
 |P^\perp_{\partial_sH_h}\partial_{ss}H_h|
 \le\kappa_t\bigl(|\partial_sH_h|+o(1)\bigr)^2+o(1).
\]
The uniform speed lower bound proves
\eqref{eq:strong-smoothing-curvature}.

We next show that, for every $\eta>0$, all sufficiently small $h$ satisfy
\begin{equation}\label{eq:strong-smoothing-reach}
 \Thi(H_h(v,t))\ge\Thi(G_t)-\eta
 \quad\text{for all }(v,t).
\end{equation}
Otherwise choose $h_n\downarrow0$ and violating parameters $(v_n,t_n)$,
and pass to a subsequence with $t_n\to t_*$.  Then
$H_{h_n}(v_n,t_n)\to G_{t_*}$ in $C^1$, and
\eqref{eq:strong-smoothing-curvature} gives
\[
 \liminf_n\operatorname{MinRad}(H_{h_n}(v_n,t_n))
 \ge\operatorname{MinRad}(G_{t_*}).
\]
The speed lower bound and uniform second-derivative bound exclude doubly
critical parameter pairs uniformly near the diagonal: for a short signed
increment $a$, the chord dotted with the initial derivative equals
$a|\partial_sH_h|^2+O(a^2)$ and is nonzero.
Minimizing doubly critical pairs therefore have a subsequence converging
to a pair of distinct parameters.  By $C^1$ convergence the limiting
pair is doubly critical for $G_{t_*}$.  Hence
\[
 \liminf_n\operatorname{dcsd}(H_{h_n}(v_n,t_n))
 \ge\operatorname{dcsd}(G_{t_*}).
\]
The thickness formula now implies
$\liminf_n\Thi(H_{h_n}(v_n,t_n))\ge\Thi(G_{t_*})$.
But $\Thi(G_{t_n})\to\Thi(G_{t_*})$ by
\cref{lem:strong-sublevel-open}, contradicting the violation and proving
\eqref{eq:strong-smoothing-reach}.

The original family has a positive thickness minimum.  Length converges
uniformly along $H_h$, and \eqref{eq:strong-smoothing-reach} holds for
arbitrarily small $\eta$.  Thus
\[
 \limsup_{h\downarrow0}\ \sup_{v,t}\Rop(H_h(v,t))\le M<L.
\]
Fix one sufficiently small positive $h$ in
\eqref{eq:strong-linear-smoothing}.  This gives the asserted homotopy.
There is no claim that $S_hG_t\to G_t$ in $B$ as $h\downarrow0$:
the homotopy parameter is the linear coefficient $v$, not the kernel
width.
\end{proof}

\section{Uniform-slack paths and reparametrization bridges}
\label{app:uniform-slack}

In this appendix $S^1=\mathbb R/\mathbb Z$ and
$B=W^{2,\infty}(S^1,\mathbb R^3)$.  Curves are initially oriented;
source reversal is addressed at the end.  Write
\[
\Kmax(c)=\operatorname*{ess\,sup}_s
    \frac{|c'(s)\times c''(s)|}{|c'(s)|^3},
\]
\[
 \Roploc(c)=\Len(c)\Kmax(c),\qquad
 \Ropglob(c)=\frac{2\Len(c)}{\operatorname{dcsd}(c)}.
\]
The thickness formula gives
$\Rop(c)=\max\{\Roploc(c),\Ropglob(c)\}$ for embeddings.
For a constant-speed curve of length $L$, one has
$c'\cdot c''=0$ almost everywhere and
$\|c''\|_\infty=L^2\Kmax(c)$.

\begin{lemma}[Continuity and semicontinuity used in the comparison]
\label{lem:comparison-prelim}
The following statements hold.
\begin{enumerate}[label=\textup{(\roman*)}]
\item Basepoint-preserving constant-speed reparametrization is continuous
in $C^1$ on regular curves.
\item $\Kmax$ is continuous in the strong $B$ topology on regular curves.
Consequently the map from parameterized strong curves to their
unparametrized $\dStar$ classes is continuous.
\item If regular embeddings $c_n\to c$ in $C^1$ and
$\sup_n\Kmax(c_n)<\infty$, then
$\liminf_n\operatorname{dcsd}(c_n)\ge\operatorname{dcsd}(c)$.
\item Ropelength is upper semicontinuous in $\dStar$.  It is continuous
in the strong quotient topology.  Thus any continuous path in either
strict sublevel has a uniform strict margin.
\end{enumerate}
The assertions about quotient continuity and upper semicontinuity also
hold for the framed normalized models of \cref{def:comparison-topologies}.
\end{lemma}

\begin{proof}
For (i) set
\[
 A_c(s)=\Len(c)^{-1}\int_0^s|c'(r)|\,dr,\qquad
 R(c)=c\circ A_c^{-1}.
\]
If $c_n\to c$ in $C^1$, their speeds are uniformly bounded below,
$A_{c_n}\to A_c$ in $C^1$, and their inverses converge in $C^1$.
Indeed the inverses converge uniformly by a common inverse Lipschitz
bound, and their derivatives are
$1/(A_{c_n}'\circ A_{c_n}^{-1})$.
The chain rule and uniform continuity of $c'$ prove (i).
For (ii), strong convergence gives convergence of $c''$ in $L^\infty$,
while the first derivatives converge uniformly and stay away from zero.
The displayed formula for curvature therefore converges in $L^\infty$.
Combine this with (i).

For (iii), put $\kappa_0=\sup_n\Kmax(c_n)>0$ and
$V=\sup_n\|c_n'\|_\infty$.  Every doubly critical pair has both connecting
arcs of length at least $\sqrt6/\kappa_0$.  To see this, use arclength
coordinates on either arc, with unit tangent $T$, and write its length
as $\ell$.  The initial orthogonality gives
\[
 0=\int_0^\ell T(s)\cdot T(0)\,ds
 \ge\int_0^\ell(1-\kappa_0^2s^2/2)\,ds
 =\ell-\kappa_0^2\ell^3/6.
\]
Thus the source parameters of a doubly critical pair have circle distance
at least $\sqrt6/(\kappa_0V)$.  Doubly critical pairs exist (a diameter
pair is one) and form a compact set away from the diagonal, so their
minimum distance is attained.  Along a subsequence realizing the liminf,
choose minimizing pairs and pass to limiting distinct source parameters.
The two orthogonality relations pass to the limit by $C^1$ convergence,
proving (iii).

For (iv), choose phases so that constant-speed representatives converge
in $C^1$.  Under $\dStar$ convergence the lengths and curvature suprema
converge, while (iii) gives upper semicontinuity of $\Ropglob$.
Use the maximum formula for $\Rop$.  Continuity in the strong quotient
follows from \cref{lem:strong-sublevel-open} and the quotient definition,
since ropelength is invariant under reparametrization and similarity.
An upper semicontinuous real function on a compact interval attains its
maximum, giving the uniform strict margin.

Length and arclength-centroid normalization are continuous in $C^1$;
under strong convergence the normalized curvature supremum
$\Len(c)\Kmax(c)$ also converges.  The remaining compact groups act
isometrically in $C^1$ and preserve curvature.  The same arguments prove
the framed assertions.  No continuity of a moving source translation
in the strong norm is used.
\end{proof}

\begin{lemma}[Constant-speed interpolation]
\label{lem:constant-speed-interpolation}
Let $a,b\in B$ have constant speeds $L_a,L_b>0$.
Put $T_a=a'/L_a$, $T_b=b'/L_b$ and
$\theta=\|T_a-T_b\|_\infty<2$.  For $0\le\tau\le1$, set
$c=(1-\tau)a+\tau b$ and $\bar L=(1-\tau)L_a+\tau L_b$.  Then
\[
 \bar L\sqrt{1-\theta^2/4}\le|c'|\le\bar L,
 \qquad
 \Roploc(c)\le
 \frac{\max\{\Roploc(a),\Roploc(b)\}}{1-\theta^2/4}.
\]
Embeddedness is not assumed or concluded in this lemma.
\end{lemma}

\begin{proof}
Write $p=(1-\tau)L_a$ and $q=\tau L_b$.  Since the tangents are unit,
\[
 |c'|^2=|pT_a+qT_b|^2
 =(p+q)^2-pq|T_a-T_b|^2
 \ge\bar L^2(1-\theta^2/4).
\]
The upper bound is the triangle inequality and gives $\Len(c)\le\bar L$.
Also
\[
 |c''|\le(1-\tau)L_a^2\Kmax(a)+\tau L_b^2\Kmax(b)
 \le\bar L\max\{\Roploc(a),\Roploc(b)\}
\]
almost everywhere.  The curvature inequality
$\kappa_c\le|c''|/|c'|^2$ proves the assertion.
\end{proof}

\begin{proof}[Proof of \cref{thm:uniform-slack-approx}]
Put $L_- =\min_t\Len(x_t)>0$ and
\[
 \omega(h)=\sup_{|t-t'|\le h}\dOne(x_t,x_{t'}),
 \qquad\omega(h)\longrightarrow0.
\]
Length is continuous in $\dOne$.  Choose a constant-speed representative
$a_0$ of $x_0$.  For $t_j=j/m$, choose successively a constant-speed
representative $a_j$ of $x_{t_j}$ minimizing
$\|a_j-a_{j-1}\|_{C^1}$.  This minimum equals
$\dOne(x_{t_j},x_{t_{j-1}})$ and is attained, since phases act isometrically
in $C^1$ and the phase circle is compact.  Define $a_t$ by affine
interpolation between consecutive $a_j$.
This is a strongly continuous finite piecewise-affine path.  Moreover
\[
 \theta_j=\left\|\frac{a_j'}{\Len(a_j)}-
                     \frac{a_{j-1}'}{\Len(a_{j-1})}\right\|_\infty
 \le\frac{2\omega(1/m)}{L_-}.
\]
For all sufficiently large $m$ these numbers are at most one.
By \cref{lem:constant-speed-interpolation}, every interpolant is regular,
has speed at least $\sqrt3 L_-/2$, and satisfies
\begin{equation}\label{eq:uniform-slack-local-bound}
 \Roploc(a_t)\le \frac{M}{1-\theta_j^2/4}=M+o(1)
\end{equation}
uniformly in the interval and in the choices of minimizing phases.

We verify embeddedness, the global ropelength bound, and approximation
uniformly by a sequential argument.  Consider any sequence of mesh sizes
$m_n\to\infty$ and interpolants $c_n$ between the adjacent representatives
at times $p_n$ and $p_n+1/m_n$.  After passing to a subsequence, $p_n\to p$.
The first representative is, after a phase choice, $C^1$-close to a fixed
representative of $x_p$.  Compactness of the phase circle gives a further
subsequence on which it converges in $C^1$ to a constant-speed embedding
$\alpha$ representing $x_p$.  The adjacent representative and $c_n$
then have the same $C^1$ limit.

Regular embeddings of type $K$ form a $C^1$-open set, so $c_n$ is
eventually an embedding of type $K$.  For completeness, a small $C^1$
ball about a regular embedding has a uniform lower chord bound
$|c(s)-c(r)|\ge b\,d(s,r)$: near the diagonal use projection onto the
limiting tangent, and away from it use compactness and injectivity.
Straight segments in this ball are isotopies, hence preserve knot type.
This standard local argument can also be extended ambiently by extending
the Lipschitz velocity field off the curves and taking its flow.
By \eqref{eq:uniform-slack-local-bound} and
$\Len(c_n)\to\Len(\alpha)>0$, the curvature suprema of $c_n$ are bounded.
Thus \cref{lem:comparison-prelim}(iii) gives
\[
 \limsup_n\Ropglob(c_n)\le\Ropglob(\alpha)\le M.
\]
Together with \eqref{eq:uniform-slack-local-bound}, this excludes
$\Rop(c_n)\ge M'$ for all large $n$.
Finally, \cref{lem:comparison-prelim}(i) gives
$\dOne([c_n],[\alpha])\to0$; the corresponding times of $x$ also converge
to $p$.  Thus a failure of uniform $\epsilon$-approximation is impossible.
These arguments prove all the required uniform assertions for sufficiently
fine subdivisions.

The derivative computation for affine pieces in
\cref{lem:strong-path-approx} proves joint $C^{1,1}$ regularity.
Strong continuity and
\cref{lem:comparison-prelim}(ii) give $\dStar$ continuity of their images.
For framed similarity classes first use centred, length-one representatives.
The same greedy construction minimizes over the compact group
$S^1\times SO(2)_u$, and the subsequence argument uses compactness of
this group.  After interpolation, normalize again when computing the
quotient distance; this is $C^1$-continuous and does not change ropelength.
For unoriented curves replace $S^1$ by $O(2)$ throughout.
\end{proof}

\subsection{Prescribed endpoint parametrizations}
\label{app:reparam-bridge}

\begin{lemma}[Interpolation of nearby reparametrizations]
\label{lem:reparam-interpolation}
Let $\alpha\in B$ be a regular constant-speed curve of length $L$ and
curvature supremum $k$.  Let $\phi_1,\phi_2$ be increasing
$C^{1,1}$ lifts with $\phi_i(s+1)=\phi_i(s)+1$,
$\phi_i'\ge b>0$, and $\|\phi_i''\|_\infty\le D$.
Set $\vartheta=Lk\|\phi_1-\phi_2\|_\infty\le1$ and
$e=\|\phi_1'-\phi_2'\|_\infty$.
Every interpolant $c=(1-\tau)\alpha\circ\phi_1+
\tau\alpha\circ\phi_2$ has speed at least $\sqrt3 Lb/2$ and satisfies
\[
 \Roploc(c)\le(1-\vartheta^2/4)^{-3/2}
 \left[Lk\left(1+\frac{e^2}{4b^2}\right)
                 +\frac{D\vartheta}{b^2}\right].
\]
\end{lemma}

\begin{proof}
Put $T=\alpha'/L$, $T_i=T\circ\phi_i$,
$p=(1-\tau)\phi_1'$, $q=\tau\phi_2'$.
Then $\|T'\|_\infty\le Lk$, $|T_1-T_2|\le\vartheta$,
$c'=L(pT_1+qT_2)$, and
$|c'|\ge L(p+q)\sqrt{1-\vartheta^2/4}$.
Almost everywhere,
\begin{align*}
 c''={}&L\bigl((1-\tau)\phi_1''T_1+\tau\phi_2''T_2\bigr)\\
       &+L\bigl((1-\tau)(\phi_1')^2T'\circ\phi_1
                         +\tau(\phi_2')^2T'\circ\phi_2\bigr).
\end{align*}
The chain rule is legitimate since each $\phi_i$ is bi-Lipschitz and
preserves null sets in both directions.  Crossing the first line with
$c'$ gives norm at most $L^2(p+q)D\vartheta$.  The second gives at most
\[
 L^3k(p+q)\bigl((1-\tau)(\phi_1')^2+\tau(\phi_2')^2\bigr).
\]
Use the variance identity
\[
 (1-\tau)(\phi_1')^2+\tau(\phi_2')^2
 =(p+q)^2+\tau(1-\tau)(\phi_1'-\phi_2')^2
\]
and divide by the lower bound for $|c'|^3$.
Finally $p+q\ge b$, $\tau(1-\tau)\le1/4$, and
$\Len(c)\le L\int_0^1(p+q)\,ds=L$ give the stated estimate.
\end{proof}

\begin{proof}[Proof of \cref{prop:reparam-bridge}]
First let $\alpha$ have constant speed and let $\psi$ be an increasing
lift of a source diffeomorphism.  The lifts
$\psi_u=(1-u)\mathrm{id}+u\psi$ have a common positive derivative lower
bound $b$ and a common second-derivative bound $D$.
Interpolate affinely between the finitely many curves
$\alpha\circ\psi_{j/N}$, $j=0,\ldots,N$.
The $\vartheta$ and $e$ of \cref{lem:reparam-interpolation} tend uniformly
to zero for adjacent samples as $N\to\infty$.  Hence all interpolants
are regular and have $\Roploc\le\Roploc(\alpha)+o(1)$.

The map $u\mapsto\alpha\circ\psi_u$ is $C^1$-continuous, whether or not
it is strongly continuous.  Any sequence of interpolants from increasingly
fine meshes has, after a subsequence, a $C^1$ limit
$\alpha\circ\psi_u$ for some $u$.  The local knot-type argument above
gives embeddedness for all sufficiently fine meshes; the curvature bounds
and \cref{lem:comparison-prelim}(iii) give
$\Ropglob\le\Ropglob(\alpha)+o(1)$ uniformly.
Thus, for a sufficiently fine mesh, the entire polygonal path satisfies
$\Rop<\Lambda$.  Its endpoints are exactly $\alpha$ and
$\alpha\circ\psi$.  Its unparametrized images tend uniformly in
$\dOne$ to the constant image $[\alpha]$, by the same subsequence
argument and constant-speed reparametrization continuity.

For a general $\gamma$, let $\alpha=R(\gamma)$ and write
$\gamma=\alpha\circ A_\gamma$ and
$\gamma\circ\phi=\alpha\circ(A_\gamma\circ\phi)$.
Both changes of parameter are $C^{1,1}$ diffeomorphisms.
Apply the preceding construction twice and reverse the first path.
The concatenation proves the proposition, including its image-closeness
assertion.  No source reversal has been used.
\end{proof}

\subsection{Comparison of endpoint relations}
\label{app:strict-relations}

\begin{proof}[Proof of \cref{cor:strict-path-comparison}]
Relation (i) implies (ii) by the quotient map, and (ii) implies (iii)
by the continuous identity map to the $\star$ model.  Relation (iii)
implies (iv) by \cref{lem:comparison-prelim}(iv): a continuous path in
the strict sublevel has a uniform strict margin.  Relation (v) also
implies (iv), since a piecewise jointly $C^{1,1}$ family is continuous
in $C^1$, and the margin is part of strict admissibility.

From (iv), \cref{thm:uniform-slack-approx} produces a finite
piecewise-affine strong path between constant-speed representatives.
Use \cref{prop:reparam-bridge} at both ends to reach any prescribed
orientation-compatible representatives.  A finite concatenation of these
paths has a uniform strict margin and is again piecewise affine.
It satisfies both (i) and (v), proving all the equivalences.

For framed similarities use the normalized version of
\cref{thm:uniform-slack-approx}.  Endpoint similarities can be joined to
the identity through smooth paths in $\Sim_u^+$, which preserve
ropelength and give strongly continuous paths on a fixed curve.
If affine pieces are desired, apply \cref{lem:strong-path-approx} to
these paths.  Thus taking the similarity quotient introduces no new
endpoint identifications beyond the stated relations.  For unoriented
curves, allow reversal when selecting endpoint representatives.  One
then compares the corresponding quotient component sets; this does not
require an oriented curve to be isotopic to its reversed parametrization.

For a strict regular-projection fiber the original path lies in the
$C^1$-open locus of a fixed diagram.  Its compact image has a uniform
$C^1$ neighborhood inside that locus.  Choose the approximation in
\cref{thm:uniform-slack-approx} within this neighborhood, and keep the
reparametrization bridges arbitrarily close to their endpoint images.
Source changes and the allowed ambient similarities preserve the diagram
class.  All inserted paths therefore remain in the same fiber.
This proves the fiber assertion and the natural component bijections.
Finally, the parameterized strong strict sublevel is open in $B$ by
\cref{lem:strong-sublevel-open}, hence locally path connected.
Its quotient map by source diffeomorphisms and framed similarities is
open: the saturation of an open set is a union of its images under fixed
homeomorphisms.  Given a quotient neighborhood, take a small convex ball
in its preimage; the image of that ball is an open path-connected
neighborhood.  Thus the strong quotient strict sublevel is locally path
connected as asserted.  This argument does not use joint continuity of
the source action and does not concern the closed $C^1$ model.
\end{proof}

\begin{remark}[What is not being approximated]
The proof of \cref{thm:uniform-slack-approx} controls the curvature part
from above.  It does not require convergence of second derivatives in
$L^\infty$ or convergence of curvature suprema along the original
$C^1$ path.  Curvature may decrease during interpolation, so the resulting
path need not approximate the original one in $\dStar$.  The theorem
therefore compares endpoint relations and supplies no relative homotopy
classification of arbitrary quotient paths.
\end{remark}

\section*{Acknowledgements}

\paragraph{Use of generative AI}
The author used ChatGPT (OpenAI) and Claude (Anthropic) as interactive aids during the
preparation of this manuscript, including for mathematical discussion,
the exploration of possible formulations and proof strategies, and the
improvement of exposition. No AI system is an author of, or bears responsibility for, any result in this paper. The author reviewed and edited all AI-assisted material, checked the mathematical statements, proofs, computations and references, and takes full responsibility for the content of this manuscript.


\end{document}